\documentclass[a4paper]{article}

\usepackage[utf8]{inputenc}
\usepackage[width=15cm,height=22cm]{geometry}
\usepackage{amsmath,amsfonts,amssymb,amsthm}
\usepackage[hidelinks,pdfusetitle]{hyperref}
\usepackage[capitalise]{cleveref}
\usepackage[dvipsnames]{pstricks}
\usepackage{bbm}
\usepackage{stmaryrd}

\newtheorem{proposition}{Proposition}[section]
\newtheorem{theorem}[proposition]{Theorem}

\newtheorem{definition}[proposition]{Definition}
\newtheorem{corollary}[proposition]{Corollary}
\newtheorem{lemma}[proposition]{Lemma}
\newtheorem{remark}[proposition]{Remark}
\numberwithin{equation}{section}
\newtheorem{example}[proposition]{Example}

\newcommand*\samethanks[1][\value{footnote}]{\footnotemark[#1]}

\title{Control theory and splitting methods}

\author{Karine Beauchard\texorpdfstring{\thanks{Univ Rennes, ENS Rennes, INRIA, CNRS, IRMAR - UMR 6625, F-35000 Rennes, France}}{},
Adrien Busnot Laurent\texorpdfstring{\samethanks}{},
Fr\'ed\'eric Marbach\texorpdfstring{\thanks{DMA, École normale supérieure, Université PSL, CNRS, 75005 Paris, France}}{}}

\newcommand{\N}{\mathbb{N}}
\newcommand{\K}{\mathbb{K}}
\newcommand{\C}{\mathbb{C}}
\newcommand{\R}{\mathbb{R}}

\newcommand{\cA}{\mathcal{A}} 
\newcommand{\cB}{\mathcal{B}} 
\newcommand{\cC}{\mathcal{C}} 
\newcommand{\cE}{\mathcal{E}} 
\newcommand{\cL}{\mathcal{L}} 
\newcommand{\cR}{\mathcal{R}} 
\newcommand{\cU}{\mathcal{U}} 
\newcommand{\cZ}{\mathcal{Z}} 

\newcommand{\fD}{\mathfrak{D}}
\newcommand{\ft}{\mathfrak{t}}

\newcommand{\ad}{\operatorname{ad}}
\newcommand{\BCH}{\operatorname{BCH}}
\newcommand{\Br}{\operatorname{Br}}

\newcommand{\Hol}{\operatorname{Hol}}
\newcommand{\Id}{\operatorname{Id}}
\newcommand{\Lie}{\operatorname{Lie}}
\newcommand{\Op}{\operatorname{Op}}

\newcommand{\sign}{\operatorname{sign}}
\newcommand{\val}{\operatorname{val}}
\newcommand{\vect}{\operatorname{span}}

\newcommand{\dd}{\,\mathrm{d}}
\newcommand{\eval}{\text{\normalfont\scshape e}}
\newcommand{\intset}[1]{\llbracket #1 \rrbracket}

\begin{document}

\maketitle

\begin{abstract}
    Our goal is to highlight some deep connections between numerical splitting methods and control theory.
    We consider evolution equations of the form $\dot{x} = f_0(x) + f_1(x)$, where $f_0$ encodes non-reversible dynamics, motivating schemes that involve only forward flows of $f_0$.
    In this context, a splitting method can be interpreted as a trajectory of the control-affine system $\dot{x}(t)=f_0(x(t))+u(t)f_1(x(t))$, associated with a control~$u$ that is a finite sum of Dirac masses.
    The goal is then to find a control such that the flow generated by $f_0 + u(t)f_1$ is as close as possible to the flow of $f_0+f_1$.

    Using this interpretation and classical tools from control theory, we revisit well-known results on numerical splitting methods and prove several new ones.
    First, we show that there exist numerical schemes of arbitrary order involving only forward flows of $f_0$, provided one allows complex coefficients for $f_1$.
    Equivalently, for complex-valued controls, we prove that the Lie algebra rank condition is equivalent to small-time local controllability.
    Second, for real-valued coefficients, we show that the well-known order restrictions are linked to so-called ``bad'' Lie brackets from control theory, which are known to obstruct small-time local controllability.
    We investigate the conditions under which high-order methods exist, thanks to a basis of the free Lie algebra that we recently constructed.
\end{abstract}

\setcounter{tocdepth}{1}
\tableofcontents

\section{Introduction}

In this article, we highlight some deep connections between numerical splitting methods and control theory and use this perspective to provide new proofs of known results and conjectures in the order theory of splitting methods.
We focus on situations where $f_0$ encodes heuristically non-reversible dynamics, so that one is interested in schemes only involving forward flows of $f_0$.

\subsection{Order theory for splitting methods}
\label{sec:biblio}

\paragraph{Splitting methods}
Splitting methods aim to solve numerically evolution equations of the form
\begin{equation}
    \label{ODE}
    \dot{x}(t)=f_0(x(t))+f_1(x(t))+\dots+f_p(x(t)),
\end{equation}
where the flows associated with the vector fields $f_j$ are easy to integrate numerically with high precision or an exact solution is available. 
For simplicity, we restrict to the case $p = 1$.

For a given time $T$, a splitting method approximates $x(T)$ by composing flows associated with the $f_j$.
A standard splitting method is of order $N$ for \eqref{ODE} if for all smooth vector fields $f_0$, $f_1$, the following estimate holds
\begin{equation}
    \label{general_splitting}
    x(T)= e^{\alpha_1 T f_0} e^{\beta_1 T f_1} \dots e^{\alpha_k T f_0} e^{\beta_k T f_1} x(0) + \underset{T \to 0}{O}(T^{N+1}),
\end{equation}
where $e^{\alpha_j T f_j}$ denotes the flow at time $\alpha_j T$ associated with the vector field $f_j$.

For instance, the Lie--Trotter splitting is of order one:
\begin{equation}
    \label{eq:Lie-Trotter}
    x(T)=e^{T f_0} e^{T f_1} x(0) +\underset{T \to 0}{O}(T^{2}),
\end{equation}
and the Strang splitting is of order two:
\begin{equation}
    \label{eq:Strang}
    x(T)=e^{T/2 f_0} e^{T f_1}e^{T/2 f_0} x(0) +\underset{T \to 0}{O}(T^{3}).
\end{equation}
Splitting methods are popular for their straightforward implementation, versatility, accuracy, and stability. 
They also exhibit good geometric properties \cite{Blanes24smf,Hairer06gni,McLachlan02sm}, for instance preserving energy, volume, symmetries, or symplecticity.
They are widely used for the numerical approximation of ODEs and PDEs \cite{Blanes05otn,Faou12gni,Lubich08fqt}, or in stochastic settings, for instance in molecular dynamics \cite{Chada24ukl,Leimkuhler13rco}.

\paragraph{Order theory}

In this paper, we focus on the construction of high-order splittings (and not on the stability analysis or the preservation of geometric properties).
The order theory of splitting methods relies on the Baker--Campbell--Hausdorff formula or the Magnus formula \cite{Magnus54ote}. 
In particular, methods of arbitrarily high order can be constructed, for instance with composition methods.
The derivation of the order conditions can be found in \cite{McLachlan02sm}.
The modern formulation is based on the algebraic framework of free Lie algebras \cite{Reutenauer93fla} or Hopf algebras (see the word series \cite{Murua17wsf, Murua18hat} and the review \cite[Sec.\ 2]{Blanes24smf}).
Free Lie algebras are also used in control theory \cite{P1,MR872457}.

\paragraph{Order theory for semigroups}
The order theory becomes more involved when positivity constraints are imposed on the coefficients.
We study the existence of splitting schemes with coefficients in $(\mathbb{A},\mathbb{B})$, that is, when $\alpha_i\in\mathbb{A}$, $\beta_i\in\mathbb{B}$ for all~$i$.
For non-reversible problems (such as the heat equation or stochastic equations), one must impose a condition of the form $\alpha_i>0$ on the coefficients in \eqref{general_splitting}. 
In this context, it is known \cite{Blanes05otn,Goldman96nto,Sheng89slp, Suzuki91gto} that the maximum order for splitting methods with coefficients in $(\R^+,\R)$ (sometimes called forward splitting methods) is $N=2$.

To go further, a possible solution is to use complex coefficients. 
Complex coefficients first appeared in the context of Hamiltonian systems \cite{Chambers03siw} and quantum mechanics \cite{Bandrauk06cis, Prosen06hon} with low order, then in \cite{Castella09smw,Hansen09hos} for parabolic problems, in the spirit of \cite{Hansen09esf}.
The large number of complex solutions for the order conditions offers more flexibility in the choice of coefficients, and can lead to schemes with smaller truncation errors and new symmetries.
On the other hand, the use of complex arithmetic incurs an additional cost when solving real-valued problems, and extending the flows to the complex plane must be handled carefully to avoid order reduction.

The papers \cite{Castella09smw,Hansen09hos} use symmetric composition methods to create splitting methods in $(\C^+,\C^+)$ up to order 14 (i.e.\ where all coefficients are complex numbers with positive real part), though the error constants deteriorate in some cases~\cite{Blanes13oho}.
In \cite{Blanes13oho}, it is proved that splitting methods in $(\C^+,\C^+)$ exist up to order 44, by building upon a splitting method of order 6 in $(\R^+,\C^+)$.
In this paper, we prove in particular that splittings in $(\R^+,\C)$ exist up to any order, thereby \emph{positively answering the open question} raised in \cite[Remark~2.7]{Blanes13oho} (with unconstrained complex $\beta_i$).

\paragraph{Commutator flows and degeneracies}
An alternative solution to go beyond the order barrier is to introduce flows associated to specific commutators in the splitting methods, assuming that these flows are explicitly available.
The first example in the literature is the Takahashi--Imada splitting \cite{Takahashi84mcc} (see also \cite{Chin97sif,Koseleff93cfp,Rowlands91ana}). 
In the context of Hamiltonian dynamics, such splittings are known as \emph{splitting methods with modified potentials} \cite{Lopez97esi, Rowlands91ana, Wisdom96sc}. 
We refer to \cite[Sections 3.2 and 8]{Blanes24smf} for a survey of splittings with commutators in the literature.

One must distinguish two concepts (see \cref{sec:flows-vs-degeneracies}).
First, one can look for general splitting methods where one allows the use of commutator flows. 
Second, one can exploit the degeneracies of specific systems for which some commutators vanish.
Motivated by \cite{Blanes05otn}, we provide new existence results for schemes in $(\R^+,\R)$ with commutators and necessary degeneracy conditions to obtain high order schemes.
In particular, we \emph{prove the conjecture} of \cite[Section V]{Chin2004}.

\paragraph{Link with control theory}

The link between splitting methods with $\mathbb{A} = \R^+$ and control theory is the following.
A splitting method can be seen as a trajectory of the control system
\begin{equation}
    \label{eq:x-f0f1}
    \dot{x}(t)=f_0(x(t))+u(t)f_1(x(t))
\end{equation}
associated with a control $u$ that is a sum of Dirac masses.
Here, $\alpha_j$ denotes the duration of the $j$-th step and $\beta_j$ the amplitude of the corresponding jump.
The goal is that the flow of the time-dependent vector field $f_0 + u(t) f_1$ on $[0,T]$ approximates $e^{T(f_0+f_1)}$ for some $T > 0$.

With this in mind, we revisit known results on numerical splitting methods, and prove new ones.
In particular, we identify Lie brackets that are obstructions to both high order numerical splitting methods and small-time local controllability.

\subsection{Definitions and notations}

\subsubsection{Formal brackets and evaluated Lie brackets}

\begin{definition}[Formal brackets]
    \label{def:free-magma}
    Let $X=\{X_0,X_1\}$ be a set of two indeterminates.
    We denote by $\Br(X)$ the \emph{free magma} over $X$, which can be defined by induction: $X \subset \Br(X)$ and, if $a, b \in \Br(X)$, then the ordered pair $(a, b)$ belongs to $\Br(X)$. 
    For $b \in \Br(X)$, let $n_0(b)$ and $n_1(b)$ be the number of $X_0$ and $X_1$ appearing in $b$.
    The length of $b$, denoted $|b|$, thus satisfies $|b| = n_0(b) + n_1(b)$.
\end{definition}

\begin{definition}
    \label{def:Mnu-Wj}
    We define the formal brackets $M_0 := X_1$, and for $j \geq 1$,
    \begin{equation}
        M_j := (M_{j-1}, X_0)
        \quad \text{and} \quad
        W_j := (M_{j-1}, M_j).
    \end{equation}
    The brackets $M_j$ contain $X_1$ only once and satisfy $|M_j|=j+1$. 
    The brackets $W_j$ contain $X_1$ twice and satisfy $|W_j|=2j+1$. 
    For instance $M_1=(X_1,X_0)$ and $W_1=(X_1,(X_1,X_0))$.
\end{definition}

\begin{definition}[Lie bracket of vector fields]
    \label{def:Lie-fields}
    For smooth vector fields $f, g$, denote by $[f,g] = \ad_f(g)$ their Lie bracket which is the smooth vector field defined by $[f,g](x)=Dg(x)f(x)-Df(x)g(x)$. 
    We also use the notation $\ad_f^{n+1}(g)=[f,\ad_f^n(g)]$ for $n \in \N$.
\end{definition}

\begin{definition}[Evaluated Lie bracket]
    \label{def:eval-fb}
    If $f_0,f_1$ are smooth vector fields and $b \in \Br(X)$, we denote by $f_b$ the vector field obtained by replacing each $X_j$ with $f_j$ in $b$. 
    For instance, with the notations above, $f_{M_0} = f_{X_1} = f_1$, $f_{M_1}=[f_1,f_0]$ and $f_{W_1}=[f_1,[f_1,f_0]]=\ad_{f_1}^2(f_0)$.
\end{definition}

\subsubsection{Splitting methods}

In the sequel, we rely on the following definition of the order of a splitting method, which reflects the asymmetry between coefficients allowed along $f_0$ and $f_1$.

\begin{definition}[Order of a splitting method with $(\mathbb{A},\mathbb{B})$ coefficients]
    \label{def:SM}
    Let $\mathbb{A},\mathbb{B} \subset \C$ and $N \in \N^*$.
    We say that $\alpha=(\alpha_1,\dots,\alpha_k) \in \mathbb{A}^k$ and $\beta = (\beta_1,\dots,\beta_{k}) \in \mathbb{B}^k$ with $k \in \N^*$ is a \emph{splitting method of order (at least) $N$} when, for all $f_0, f_1 \in \cC^\infty(\K^d;\K^d)$ (with $\K = \R$ or $\C$ and $d \in \N^*$),
    \begin{equation}
        \label{Def_SM:estim}
        e^{T(f_0+f_1)} = e^{\alpha_1 T f_0} e^{\beta_1 T f_1} \dotsb e^{\alpha_k T f_0} e^{\beta_k T f_1} + \underset{T \to 0}{O}\left(T^{N+1}\right)
    \end{equation}
    in the following sense: for all $x_0 \in \K^d$, there exists $C=C(f_0, f_1, x_0)$ such that, for $T$ small enough,
    \begin{equation}
        \left|e^{T(f_0+f_1)} x_0 - e^{\alpha_1 T f_0} e^{\beta_1 T f_1} \dotsb e^{\alpha_k T f_0} e^{\beta_k T f_1} x_0 \right| \leq C T^{N+1}.
    \end{equation}
\end{definition}

When complex coefficients appear, we implicitly assume that the vector fields are holomorphic, so that their flows are well defined for sufficiently small complex times (see also \cref{sec:holomorphic-systems}).


In \cref{def:SM}, the parameters $k, \alpha, \beta$ must be independent of $f_0$, $f_1$, and $T$.
A weaker notion consists in requiring \eqref{Def_SM:estim} only for a subclass of pairs of vector fields, possibly satisfying additional structural relations, enabling higher-order methods.

\begin{definition}[Splitting method relative to a subclass]
    Let $\mathcal{P}$ be a set of pairs of vector fields.
    We say that a splitting method is \emph{of order $N$ relative to $\mathcal{P}$} when \eqref{Def_SM:estim} holds for any $(f_0,f_1) \in \mathcal{P}$.
\end{definition}

The splitting methods of \cref{def:SM} only involve flows of $f_0$ and $f_1$.
For some systems, although the flow of $f_0 + f_1$ is not directly available, one may also have access to the flows of some specific commutators of $f_0$ and $f_1$ (for example, the flow of $f_{W_1} = [f_1,[f_1,f_0]]$, and more generally the flow of $f_b$ for some $b \in \Br(X)$).
This leads to the following definition where we allow a set $\cE \subset \Br(X) \setminus \{ X_0 \}$ of ``enabled'' brackets.

\begin{definition}[Splitting method with commutator flows]
    Let $\cE \subset \Br(X) \setminus \{ X_0 \}$.
    With the notations of \cref{def:SM}, a \emph{splitting method of order $N$ involving $X_0$ and $\cE$} is specified by additional data $b_1, \dotsc, b_k \in \cE$ such that
    \begin{equation}
        e^{T(f_0+f_1)} =
        e^{\alpha_1 T f_0}
        e^{\beta_1 T^{|b_1|} f_{b_1}} e^{\alpha_2 T f_0}
        e^{\beta_2 T^{|b_2|} f_{b_2}} \dotsb
        e^{\alpha_k T f_0} e^{\beta_k T^{|b_k|} f_{b_k}}
        + \underset{T \to 0}{O}\left(T^{N+1}\right).
    \end{equation}
    Setting $\alpha_j = 0$ allows one to concatenate multiple flows associated with different enabled brackets.
\end{definition}

\subsubsection{Controllability}

Given smooth vector fields $f_0, f_1, \dotsc, f_m$ defined in a neighborhood of $0 \in \K^d$ where $\K = \R$ or $\C$, we consider the control-affine system
\begin{equation}
    \label{eq:syst-control}
    \dot{x}(t) = f_0(x(t)) + u_1 (t) f_1(x(t)) + \dotsb + u_m(t) f_m(x(t)),
\end{equation}
where $u_j \in L^1((0,T);\K)$ are the control inputs.
When there is no control in front of $f_0 \neq 0$, such a system is called ``with drift'' and one typically assumes that $f_0(0) = 0$ and studies its behavior near the equilibrium $(x,u) = (0,0)$.
When there is a control $u_0(t)$ in front of $f_0$ (or equivalently when $f_0 = 0$), the system is called ``driftless'' and $(x,u) = (0,0)$ is still an equilibrium.

Among the many notions of controllability (see \cite[Section~1.2]{P2} for further discussion), we use:

\begin{definition}[Small-state STLC]
    \label{def:STLC}
    We say that \eqref{eq:syst-control} is \emph{small-state small-time locally controllable} at $0$ when, for every $T > 0$, for every $\delta > 0$, there exists $r > 0$ such that, for every target state $x^* \in B(0,r)$, there exists $u \in L^1((0,T);\K^m)$ such that the solution to \eqref{eq:syst-control} associated with the initial condition $x(0) = 0$ and the control $u$ satisfies $x(T) = x^*$ and $x([0,T]) \subset B(0,\delta)$.
\end{definition}

\subsection{Main results}
\label{section:main_results}

\subsubsection{Arbitrary order \texorpdfstring{$(\R,\R)$}{(R,R)} splitting methods}

It is well-known that splitting methods with $(\R,\R)$ coefficients of arbitrary order exist.
In our notation, one has the following classical result.

\begin{theorem}
    \label{thm:main-Chow-split}
    For every $N\in\N^*$, there exists an $(\R,\R)$ splitting method of order $N$.
\end{theorem}

Without sign constraints on the real coefficients, an $(\R,\R)$ splitting method corresponds to a trajectory of the driftless control-affine system
\begin{equation}
    \label{syst_ss_drift}
    \dot{x}(t)= u_0(t) f_0(x(t)) + u_1(t) f_1 (x(t))
\end{equation}
with state $x(t) \in \R^d$ and control $(u_0,u_1):\R^+\to \R^2$, where $u_0$ and $u_1$ are piecewise constant functions with disjoint supports and do not depend on $(f_0,f_1)$.

In control theory, the analogue of \cref{thm:main-Chow-split} is the following result, known as the Chow--Rashevskii necessary and sufficient condition for the controllability of driftless control-affine systems \cite{MR0001880,rashevsky1938connecting}, involving the ``Lie algebra rank condition'' $\Lie_\R(f_0,f_1)(0) = \R^d$, where $\Lie_\R(f_0,f_1)$ denotes the Lie algebra spanned over $\R$ by $f_0$ and $f_1$ (see \cite[Theorems 3.17 and 3.19]{coron:hal-03588552}).

\begin{theorem}
    \label{thm:main-Chow-control}
    Let $f_0, f_1$ be real-analytic vector fields on a neighborhood of $0$ in $\R^d$.
    System~\eqref{syst_ss_drift} is small-state STLC with real-valued controls $u_0,u_1$ if and only if $\Lie_\R(f_0, f_1)(0)=\R^d$.
\end{theorem}


\subsubsection{Arbitrary order \texorpdfstring{$(\R^+,\C)$}{(R+,C)} splitting methods}

In \cite[Remark~2.7]{Blanes13oho}, it is mentioned that the existence of $(\R^+,\C)$ splitting methods of arbitrary order is an open question. We provide here the following positive answer.

\begin{theorem}
    \label{thm:main-complex-split}
    For every $N\in\N^*$, there exists an $(\R^+,\C)$ splitting method of order $N$.
\end{theorem}

\cref{thm:main-complex-split} extends the lower order methods recalled in \cref{sec:biblio} to arbitrary orders.
It is an abstract existence result.
Indeed, our proof relies on an inversion argument, and only provides a way to compute an approximate solution $(\alpha,\beta)$.

An $(\R^+,\C)$ splitting method corresponds to a trajectory of the scalar-input system~\eqref{eq:x-f0f1} with state $x(t) \in \C^d$ and control $u$ which is a finite sum of Dirac masses with complex amplitudes.
The following control statement is the analogue of \cref{thm:main-complex-split} for the control system \eqref{eq:x-f0f1}.

\begin{theorem}
    \label{thm:main-complex-control}
    Let $f_0, f_1$ be holomorphic vector fields on a neighborhood of $0$ in $\C^d$ with $f_0(0)=0$.
    System \eqref{eq:x-f0f1} is small-state STLC with complex-valued control $u$ if and only if $\Lie_\C (f_0,f_1)(0) = \C^d$.
\end{theorem}

For control theorists, \cref{thm:main-complex-control} may seem surprising at first sight.
Indeed, for control-affine systems with drift of the form \eqref{eq:x-f0f1} and real-valued controls $u$, no necessary and sufficient condition for controllability is known. 
In particular, the Lie algebra rank condition does not imply the controllability, as shown by the canonical example on $\R^2$ given by
\begin{equation}
    \label{eq:syst-W1}
    \begin{cases}
        \dot{x}_1 = u \\
        \dot{x}_2 = x_1^2
    \end{cases}
\end{equation}
for which $f_1(0) = (1,0)$ and $[f_1,[f_1,f_0]](0) = (0,2)$, so that $\Lie_{\R}(f_0,f_1)(0) = \R^2$ but, for real-valued $u$, one has $x_1 \in \R$ so $\dot{x}_2 \geq 0$, preventing controllability, since one cannot reach a target state $(x_1^*,x_2^*)$ with $x_2^*<0$ starting from the initial state $(0,0)$.

\cref{thm:main-complex-control} is nevertheless reasonable, since, as in \eqref{eq:syst-W1}, all known obstructions to controllability rely on the presence of positive drifts in the dynamics (see \cite{P2}).
In particular, one checks that \eqref{eq:syst-W1} is controllable in $\C^2$ with complex-valued controls $u(t) \in \C$ since $i^2 = -1$.

\subsubsection{The first obstruction to controllability and to $(\R^+,\R)$ splitting methods}

It is well known that splitting methods with $(\R^+,\R)$ coefficients suffer from severe order limitations.
In particular, one has the following result (see e.g.\ \cite{Blanes05otn}).

\begin{theorem}
    \label{thm:max-2}
    The maximal order of an $(\R^+,\R)$ splitting method is 2.
\end{theorem}

We claim in the sequel that the source of this order restriction is the positive-definite quadratic form governing the coordinate associated with the ``bad'' bracket $W_1 = (X_1,(X_1,X_0))$ in the free system \eqref{formalODE_A} and that this is linked with a well-known obstruction to controllability.

An $(\R^+,\R)$ splitting method corresponds to a trajectory of the control system \eqref{eq:x-f0f1} with state $x(t) \in \R^d$ and control $u:\R^+ \to \R$, where $u$ is a finite sum of Dirac masses with real amplitudes.
In control theory, the following result, due to Sussmann \cite[Proposition 6.3]{MR710995} (see also \cite[Theorem~1.10]{P2} for a modern proof using the Magnus formula and ruling out small-state STLC), is the first necessary condition for controllability.

\begin{theorem}
    \label{thm:W1-control}
    Let $f_0, f_1$ be smooth vector fields on a neighborhood of $0 \in \R^d$ such that $f_0(0)=0$.
    If the system \eqref{eq:x-f0f1} is small-state STLC then $f_{W_1}(0) \in \vect\{ f_{M_{\nu}}(0) \mid \nu \in \N \}$.
\end{theorem}

An analogue of \cref{thm:W1-control} for splitting methods is the following result.

\begin{theorem}
    \label{Thm:CN_W1}
    Let $f_0, f_1$ be smooth vector fields on $\R^d$.
    If there exists an $(\R^+,\R)$ splitting method of order $3$ relative to $(f_0, f_1)$, then $f_{W_1}$ and $f_{M_2}$ are linearly dependent.
\end{theorem}

The difference between the compensation conditions in \cref{thm:W1-control,Thm:CN_W1} comes from the scaling. 
In small-state STLC, the control $u$ is small: the displacement along $f_{W_1}(0)$ is quadratic in $u$, whereas the displacements along the $f_{M_\nu}(0)$ are linear, so any of them may compensate for the drift generated by $W_1$. By contrast, in \cref{def:SM}, the coefficients of a splitting method are independent of the final time $T$, so compensation can only occur between brackets of the same total length. Since the only other bracket of length $3$ is $M_2=((X_1,X_0),X_0)$, the compensation condition only involves $f_{W_1}$ and $f_{M_2}$. If the coefficients were allowed to depend polynomially on~$T$, the condition could involve $f_{W_1}$ and $\{ f_{M_\nu} \mid \nu \leq 2 \}$.

Another indication that $W_1$ is the source of \cref{thm:max-2} is that this order restriction disappears as soon as one can compute the flow of $f_{W_1}$ (in fact, the flow of $-f_{W_1}$ already suffices). 

\begin{theorem}
    \label{thm:main-exists-4}
    There exists an $(\R^+,\R)$ splitting method of order 4 involving $X_0$ and $X_1$, $W_1$.
\end{theorem}

This observation already appears in \cite[Section 5]{Blanes05otn}, where an $(\R^+,\R^+)$ splitting method involving $X_0$, $X_1$, and $-W_1$ is constructed. 
Our proof relies on control theory and yields only an $(\R^+,\R)$ scheme, but it extends readily to higher-order schemes (see below). 
As in \cref{thm:main-complex-split} above and \cref{Thm:Deg,thm:max-N} below, our positive results are abstract existence statements: the proofs provide a way to approximate $(\alpha,\beta)$, but not to compute it exactly.
\subsubsection{The next obstructions to controllability and to $(\R^+,\R)$ splitting methods}

The obstruction $W_1$ is far from being the only one.
In fact, even in situations where the first necessary conditions $f_{W_1}(0) \in \vect\{f_{M_\nu}(0) \mid \nu\in \N\}$ (for control theory) or $f_{W_1}$ and $f_{M_2}$ are linearly dependent (for splitting) hold, or one incorporates the flow of $f_{W_1}$, other obstructions occur. 
The next obstruction comes from $W_2 = \ad^2_{(X_1,X_0)}(X_0)$ (see \cref{def:Mnu-Wj}).

In the spirit of \cite{JDE} for control theory, we exhibit a list of obstructions to the construction of $(\R^+,\R)$ splitting methods, which are associated with quadratic quantities in the control (or coefficients along $f_1$). 
We prove the following results.

\begin{theorem}
    \label{Thm:Deg}
    Let $N \in \N^*$ and $f_0, f_1$ be smooth vector fields on $\R^d$ such that $f_{W_j}=0$ for $j = 1, \dots, N-1$. If there exists an $(\R^+,\R)$ splitting method of order $(2N+1)$ relative to $(f_0, f_1)$ then $f_{M_{2N}}$ and $f_{W_N}$ are linearly dependent.
\end{theorem}

\begin{theorem}
    \label{thm:max-N}
    For any $N \in \N^*$, there exists an $(\R^+,\R)$ splitting method of order $2N$ involving $X_0$ and $X_1, W_1, \dots, W_{N-1}$, and this is the maximal possible order of such methods.
\end{theorem}

\cref{thm:max-N} was conjectured in 2004 in \cite[Section V]{Chin2004}, which contains a proof for $N = 2$ and $N = 3$.
The case $N = 2$ was also investigated numerically in \cite[Section~5]{Blanes05otn} and theoretically in \cite{AuzingerHofstatterKoch2019}.

The philosophy is iterative. 
First, $W_1$ obstructs order $3$ for methods involving only $X_0$ and $X_1$; it is therefore natural to add $W_1$, or equivalently to restrict to vector fields such that $f_{W_1}=0$. 
Second, $W_2$ obstructs order $5$ for methods involving $X_0$, $X_1$, and $W_1$; one is then naturally led to add $W_2$, or to assume $f_{W_1}=f_{W_2}=0$.
Iterating this principle yields a purely quadratic obstruction theory for splitting methods. 
Other approaches are of course possible, for instance by adjoining other elements of $\cL(X)$ or by considering different degeneracies, and the algebraic formalism developed here can be used to test them.

The control-theoretic counterpart to \cref{thm:max-N} is the following new result.

\begin{theorem}
    \label{thm:Control_2N}
    Let $N \in \N^*$ and $\cE=\{X_1,W_1,\dots,W_{N-1}\}$.
    For any smooth vector fields $f_0, f_1$ on a neighborhood of $0$ in $\R^d$ such that $f_0(0)=0$ and $\cL_{2N}(f)(0) = \R^d$, the system
    \begin{equation}
        \label{EDOf0fA}
        \dot{x}(t)= f_0(x(t)) + \sum_{b \in \cE} u_b(t) f_b(x(t))
    \end{equation}
    is small-state STLC.
\end{theorem}

\subsection{Open problems}

We conclude with several open questions suggested by the results of this work.

\begin{enumerate}
    \item \emph{Do there exist $(\C^+,\C^+)$ splitting methods of arbitrary order?}
    \cref{thm:main-complex-split} yields $(\R^+,\C)$ splitting methods of arbitrary order, which is a different problem.

    \item \emph{For which sets of enabled brackets $\cE \subset \Br(X) \setminus \{ X_0 \}$ does there exist splitting methods of arbitrary order?}
    By \cref{thm:max-N}, $\cE = \{ X_1 \} \cup \{ W_j \mid j \geq 1 \}$ works, but other sets could too.

    \item \emph{What are the sharp degeneracy conditions on a pair of smooth vector fields $f_0,f_1$ on $\R^d$ that guarantee the existence of a splitting method of order $N$ for this pair?}
    \cref{Thm:Deg} only yields a partial answer (for instance for $N = 2$, we assume that $f_{W_1} = 0$ to obtain information on $f_{W_2}$).
\end{enumerate}

\subsection{Structure of the article}

In \cref{sec:Prerequis}, we recall prerequisites on control-affine systems.

In \cref{sec:Magnus-system}, we introduce the formal differential equation associated with system \eqref{EDOf0fA} and a Magnus-type representation formula for its solution, involving formal Lie series.
We then study the differential equation solved by truncations of this Lie series, that we call the Magnus system.

In \cref{Sec:Control_split_Magnus}, we clarify the fundamental link between the controllability of the Magnus system, the existence of splitting methods of order $N$, and the controllability of the system \eqref{EDOf0fA}. 
This fundamental link is used to prove all the results of this article.

We prove \cref{thm:main-complex-split,thm:main-complex-control} in \cref{Sec:Complex}. 
We prove \cref{thm:Control_2N} and the existence part of \cref{thm:max-N} in \cref{sec:high_order_comm}, and the upper bound of \cref{thm:max-N} in \cref{sec:max-order}.
We prove \cref{Thm:Deg} in \cref{sec:high_order_deg}.

In \cref{sec:error-proofs}, we prove new error estimates for the Chen--Fliess and Magnus formulas (stated in \cref{subsec:error_Magnus_syst} and used throughout the article).

In \cref{s:appendix}, we gather proofs of several classical control results used in the article. 
These results may be skipped on a first reading, to get to the heart of the article more quickly.

\section{Prerequisites about control-affine systems} \label{sec:Prerequis}

In this section $\K=\R$ or $\C$ denotes the base field.
Given $m \in \N^*$ and smooth vector fields $f_0, f_1, \dotsc, f_m$ on $\K^d$ with $f_0(0) = 0$, we consider the control-affine system
\begin{equation}
    \label{eq:syst-f0-f1-fm}
    \dot{x}(t)= f_0(x(t)) + u_1(t) f_1(x(t)) + \dotsb + u_m(t) f_m(x(t)).
\end{equation}
We assume in this section that the vector fields $f_0, f_1, \dotsc, f_m$ are all complete.
In particular, for any $T > 0$, $u \in L^1((0,T);\K^m)$ and $x_0 \in \K^d$, it is standard that system \eqref{eq:syst-f0-f1-fm} admits a unique Carathéodory solution on $[0,T]$ with initial condition $x_0$.
We denote it by $x(t;f,u,x_0)$.

In \cref{s:impulsive}, we define a set of impulsive controls, and the associated solutions. 
In \cref{subsec:Ex=App}, we state the equivalence between exact and approximate local controllability of system \eqref{eq:syst-f0-f1-fm}.
In \cref{subsec:extended_syst}, we present an ``extension'' argument, which is often used to prove its controllability.

\subsection{Impulsive controls and associated solutions}
\label{s:impulsive}

To relate controllability to splitting, we will use the following class of impulsive controls.

\begin{definition}[Set $\fD$ of impulsive controls]
    Given $T > 0$ and $m \in \N^*$, we denote by~$\fD$ the set of \emph{impulsive controls} on $[0,T]$ with values in $\K^m$.
    Each $u \in \fD$ is a finite ordered list $(\tau_k, \beta_k, i_k)_{k \in \intset{1,n}}$ where $0 \leq \tau_1 \leq \dotsb \leq \tau_n \leq T$ are the times of the pulses, $\beta_k \in \K$ their amplitudes, and $1 \leq i_k \leq m$ the indices of the active control channels. 
\end{definition}

The notation $\fD$ does not record the choice of $T$, $\K$ or $m$, but there will be no ambiguity when it is used in the article.
When the impulse times are pairwise distinct, an impulsive control may be identified with the vector-valued measure $\beta_1 e_{i_1} \delta_{\tau_1} + \dotsb + \beta_n e_{i_n} \delta_{\tau_n}$.
If several impulses occur at the same time, this identification is no longer sufficient, because it does not record their order.

\begin{definition}[Solution associated with an impulsive control]
    \label{def:impulsive-solution}
    Let $T > 0$, $u \in \fD$ and $x_0 \in \K^d$.
    We define an associated \emph{càdlàg} (right-continuous, left limits at each point) \emph{impulsive trajectory} to~\eqref{eq:syst-f0-f1-fm} as follows.
    Set $y_0 := x_0$, $\tau_0 := 0$ and define inductively, for $k = 1, \dotsc, n$,
    \begin{equation}
        y_k^- := e^{(\tau_k-\tau_{k-1})f_0} y_{k-1},
        \qquad
        y_k := e^{\beta_k f_{i_k}} y_k^-.
    \end{equation}
    We then set, for $t \in [0,T]$,
    \begin{equation}
        x(t;f,u,x_0):=e^{(t-\tau_{k_t})f_0} y_{k_t},
        \qquad
        k_t := \max\{ k \in \intset{0,n} \mid \tau_k \leq t\}.
    \end{equation}
\end{definition}

The state $x(t)$ contains the contribution of all the ordered impulses up to time $t$.
In particular, when $0 = \tau_1 < \tau_2$, one may have $x(0) = e^{\beta_1 f_{i_1}}(x_0) \neq x_0$.

\begin{example}
    The Lie--Trotter and Strang splitting of \eqref{eq:Lie-Trotter} and \eqref{eq:Strang} can be seen as associated with the impulsive controls $u = ((0, T, 1))$ and $u = ((T/2, T, 1))$.

    For a more elaborate example, if $u = ((0, 1, 1), (1, 2, 1), (3, -7, 2), (3, 5, 3))$, then
    \begin{equation}
        x(t;f,u,x_0)= \left\lbrace
        \begin{array}{ll}
            e^{t f_0} e^{f_1} x_0 \quad & \text{ if } t \in [0,1), \\
            e^{(t-1) f_0} e^{2 f_1} e^{f_0} e^{f_1} x_0 & \text{ if } t \in [1,3), \\
            e^{(t-3) f_0} e^{5 f_3} e^{-7 f_2} e^{2 f_0} e^{2 f_1} e^{f_0}e^{f_1} x_0
            & \text{ if } t \in [3,\infty).
        \end{array}
        \right.
    \end{equation}
    
\end{example}

\begin{proposition}
    \label{prop:impulsive-reg}
    Let $T > 0$, $u \in \fD$ and $x_0 \in \K^d$.
    There exists a sequence $u^\varepsilon \in L^1((0,T);\K^m)$ of piecewise-constant controls, bounded in $L^1$, such that:
    \begin{equation}
        \forall t \in (0,T], \quad 
        \lim_{\varepsilon \to 0} x(t;f,u^\varepsilon,x_0) = x(t;f,u,x_0).
    \end{equation}
\end{proposition}

\begin{proof}
    Let $u = (\tau_k,\beta_k,i_k)_{k\in\intset{1,n}} \in \fD$.
    For $\varepsilon>0$, define
    \begin{equation}
        I_k^\varepsilon :=
        \begin{cases}
            [(k-1)\varepsilon,k\varepsilon],
            &\text{if } \tau_k=0,\\[1ex]
            [\,\tau_k-(n-k+1)\varepsilon,\;\tau_k-(n-k)\varepsilon\,],
            &\text{if } \tau_k>0.
        \end{cases}
    \end{equation}
    Since the sequence $(\tau_k)_{k\in\intset{1,n}}$ is nondecreasing, one checks that for $\varepsilon>0$ small enough the intervals $I_k^\varepsilon$ are pairwise disjoint and ordered in the same way as the indices $k$. 
    Moreover, if $\tau_k=0$ then $I_k^\varepsilon$ lies to the right of $0$, while if $\tau_k>0$ then $I_k^\varepsilon$ lies to the left of~$\tau_k$.

    We define a piecewise-constant $u^\varepsilon \in L^1((0,T);\K^m)$ by
    \begin{equation}
        u^\varepsilon(t) := \sum_{k=1}^n \frac{\beta_k}{\varepsilon} \mathbf{1}_{I^\varepsilon_k}(t) e_{i_k}.
    \end{equation}
    Then $\|u^\varepsilon\|_{L^1(0,T)} \leq \sum_{k=1}^n |\beta_k|$ and, for each $t \in (0,T]$, $x(t;f,u^\varepsilon,x_0) \to x(t;f,u,x_0)$.
\end{proof}

\subsection{Exact and approximate STLC} 
\label{subsec:Ex=App}

For $T > 0$ and either $\cU = \fD$ or $\cU = L^1$, define the set of reachable states for system \eqref{eq:syst-f0-f1-fm} as
\begin{equation}
    R_T(\cU) := \{ x(T;f,u,0) \mid u \in \cU \} \subset \K^d.
\end{equation}
Since $f_0(0) = 0$, one checks that $R_T(\cU)$ is non-decreasing with respect to $T$.

In the sequel, we will use exact and approximate small-time local controllability interchangeably, since they are equivalent in our setting (see \cref{s:exact=app-proof} for a self-contained proof).

\begin{proposition}[STLC]
    \label{p:approx=exact}
    Assume that $(\Lie \{ f_0, f_1, \dotsc, f_m \})(0) = \K^d$.
    Let $\cU = \fD$ or $\cU = L^1$.
    The following notions are equivalent:
    \begin{flalign*}
        & \quad - \quad \text{small-time local exact $\cU$-controllability:} 
        && \forall T > 0, \quad 0 \in \operatorname{int} R_T(\cU), &\\
        & \quad - \quad \text{small-time local approximate $\cU$-controllability:} 
        && \forall T > 0, \quad 0 \in \operatorname{int} \overline{R_T(\cU)}. &
    \end{flalign*}
\end{proposition}

Moreover, approximate $\fD$-STLC and $L^1$-STLC are equivalent (so their exact counterparts too).

\begin{proposition}
    \label{p:D=L1}
    For any $T > 0$, one has $\overline{R_T(\fD)} = \overline{R_T(L^1)}$.
\end{proposition}

\begin{proof}
    On the one hand, given $u \in \fD$, by \cref{prop:impulsive-reg}, there exists a sequence of regularizations $u^\varepsilon \in L^1((0,T);\K^m)$ for which $x(T;f,u^\varepsilon,0) \to x(T;f,u,0)$.

    On the other hand, given $u \in L^1$, by density we first approximate it by a sequence of step functions.
    Within each step of duration $\tau > 0$, control value $\alpha \in \K^m$, and starting from an initial point $x_0$, the Lie--Trotter product formula (see \cref{lem:LT}) yields
    \begin{equation}
        e^{\tau (f_0 + \alpha_1 f_1 + \dotsb + \alpha_m f_m)}(x_0) = \lim_{N \to +\infty} \left( e^{\frac{\tau}{N} f_0} e^{\frac{\tau \alpha_1}{N} f_1} \dotsb e^{\frac{\tau \alpha_m}{N} f_m} \right)^N (x_0).
    \end{equation}
    Hence we can construct a sequence $u^N \in \fD$ for which $x(T;f,u^N,0) \to x(T;f,u,0)$.
\end{proof}

\subsection{Extension method}
\label{subsec:extended_syst}

A classical strategy to establish the STLC of a system is to apply ``extension'' or ``saturation'' operations, to enlarge the set $\{f_1,\dots,f_m\}$, until the following result can be applied.

\begin{lemma}
    \label{Lem_gen}
    If $\vect \left\{f_1(0),\dots,f_m(0) \right\}=\K^d$, then system \eqref{eq:syst-f0-f1-fm} is STLC.
\end{lemma}

\begin{proof}
    One may assume that the vectors $f_1(0),\dotsc,f_d(0)$ are linearly independent. 
    By the inverse mapping theorem, the map $F : \alpha \in \K^d \mapsto e^{\alpha_d f_d} \dots e^{\alpha_1 f_1}(0) \in \K^d$ is a local diffeomorphism of $\K^d$ around $0$. 
    For $x^* \in \K^d$ small enough, set $\alpha=F^{-1}(x^*)$ and define
    \begin{equation}
        u = ((T, \alpha_k, k))_{k \in \intset{1,d}}
        \quad \text{so} \quad
        x(T;f,u,0) = e^{\alpha_d f_d} \dotsb e^{\alpha_1 f_1} e^{T f_0}(0) = F(\alpha) = x^*.
    \end{equation}
    Hence $0 \in \operatorname{int} \overline{R_T(\fD)}$.
    By \cref{p:approx=exact,p:D=L1}, system \eqref{eq:syst-f0-f1-fm} is STLC.
\end{proof}

An example of legitimate extension is given by the following statement.

\begin{proposition}
    \label{p:enrich-f1f2}
    For $m \geq 2$, system \eqref{eq:syst-f0-f1-fm} is
    small-time locally approximately $\fD$-controllable iff the extended system
    \begin{equation}
        \label{eq:syst-extended-f1f2}
        \dot{x} = f_0(x) + u_1 f_1(x) + \dotsb + u_m f_m(x) + u_{m+1} [f_1,f_2](x)
    \end{equation}
    is small-time locally approximately $\fD$-controllable.
\end{proposition}

\begin{proof}
    The $\fD$-controlled trajectories of the system \eqref{eq:syst-f0-f1-fm} involve the composition of a finite number of flows of the form $e^{a_j f_j}$ with $a_j \in \K$ (and $a_0 \in \R^+$). 
    The $\fD$-controlled trajectories of the extended system \eqref{eq:syst-extended-f1f2} involve the composition of a finite number of flows of the form $e^{a_j f_j}$ and $e^{\alpha [f_1,f_2]}$ with $a_j, \alpha \in \K$ (and $a_0 \in \R^+$). 
    Thus, it suffices to prove that one can approximate flows of the form $e^{\alpha[f_1,f_2]}$ by $\fD$-controlled trajectories of the initial system \eqref{eq:syst-f0-f1-fm}. 

    Uniformly within any compact set of $\K^d$, as $t \to 0$, one has
    \begin{equation}
        e^{- \alpha \sqrt{t} f_1} e^{-\sqrt{t} f_2} e^{+ \alpha \sqrt{t} f_1} e^{+\sqrt{t} f_2}(y) = y + \alpha t [f_1,f_2](y) + O(t \sqrt{t}).
    \end{equation}
    Hence (see \cref{lem:abstract-product}), for any $x_0 \in \K^d$, as $N \to \infty$, 
    \begin{equation}
        \left( e^{- \frac{\alpha}{N} f_1}
        e^{- \frac{1}{N} f_2}
        e^{+ \frac{\alpha}{N} f_1}
        e^{+ \frac{1}{N} f_2}
        \right)^{N^2}(x_0) \to e^{\alpha [f_1,f_2]}(x_0),
    \end{equation}
    which proves that we can approximate the flow $e^{\alpha [f_1,f_2]}$ by $\fD$-controlled trajectories of \eqref{eq:syst-f0-f1-fm}.
\end{proof}

Combining \cref{Lem_gen,p:enrich-f1f2}, we recover the following famous result.

\begin{corollary}
    If $(\Lie \left\{ f_1, \dotsc, f_m \right\})(0) = \K^d$, then system \eqref{eq:syst-f0-f1-fm} is STLC.
\end{corollary}

\section{Construction of the Magnus system} 
\label{sec:Magnus-system}

In this section, $\K=\R$ or $\C$ denotes the base field, and $X=\{X_0, X_1\}$ is a set of two noncommutative indeterminates.

In \cref{subsec:alg}, we introduce the free algebra and free Lie algebra generated by $X$ over $\K$, together with the algebra of formal power series.
In \cref{subsec=hom_formal_eq}, we introduce the formal differential equation associated with \eqref{EDOf0fA}.
In \cref{subsec:Magnus_A}, we derive a Magnus-type representation formula for this solution in terms of a formal Lie series. 
In \cref{subsec:Trunc}, we study its truncation to degree $\leq N$ and prove that it solves an affine system posed on the free nilpotent Lie algebra. 
In \cref{subsec:LARC_Magnus_syst}, we analyze the associated Lie algebra rank condition.

\subsection{Free algebras} \label{subsec:alg}

\begin{definition}[Free associative algebra]
    \label{def:free.algebra}
    Let $\cA(X) = \cA^0(X) \oplus \cA^1(X) \oplus \dotsb$ be the algebra of noncommutative polynomials in $X_0$, $X_1$ with coefficients in $\K$, where $\cA^n(X)$ is the finite dimensional $\K$-vector space spanned by words of length $n$ over $X$ (e.g.\ $\cA^0(X) = \K$ and $\cA^1(X) = \K X_0 \oplus \K X_1$).
\end{definition}

\begin{definition}[Lie bracket]
    The algebra $\cA(X)$ is endowed with the Lie bracket $[a,b] := ab - ba$, which satisfies $[a, a] = 0$ and the Jacobi identity $[a,[b,c]]+[c,[a,b]]+[b,[c,a]]=0$.
\end{definition}

\begin{definition}[Free Lie algebra]
    \label{def:free.lie}
    Let $\cL(X)$ be the Lie subalgebra of $\cA(X)$ generated by $X$, i.e.\ its smallest vector subspace containing $X_0, X_1$ and stable under Lie brackets. 
    There is an \emph{evaluation} map $\eval$ from the free magma $\Br(X)$ defined in \cref{def:free-magma} to~$\cL(X)$.
    It is defined by $\eval(X_i) = X_i$ for $X_i \in X$ and $\eval((b_1,b_2)) = [\eval(b_1),\eval(b_2)]$ for $b_1,b_2 \in \Br(X)$. 
    Then $\eval(\Br(X))$ spans $\cL(X)$.
\end{definition}

\begin{definition}[Formal power series]
    Let $\widehat{\cA}(X)$ be the unital associative algebra of formal power series in $X_0$, $X_1$: its elements are sequences $a = (a^{\langle n \rangle})_{n\in\N}$ usually written $a = \sum_{n \in \N} a^{\langle n \rangle}$, where $a^{\langle n \rangle} \in \cA^n(X)$; in particular, $a^{\langle 0 \rangle} \in \K$ is the constant term.
    Let $\widehat{\cL}(X)$ be the Lie algebra of formal power series $a \in \widehat{\cA}(X)$ for which $a^{\langle n \rangle} \in \cL(X)$ for each $n \in \N$.
\end{definition}

For $a \in \widehat{\cA}(X)$ with $a^{\langle 0 \rangle}=0$,
\begin{equation}
    \label{def:exp_log}
    \exp(a) = \sum_{k=0}^\infty \frac{a^k}{k!}
    \quad \text{ and } \quad
    \log(1+a) = \sum_{k=1}^\infty \frac{(-1)^{k+1}}{k} a^k
\end{equation}
are well defined elements of $\widehat{\cA}(X)$.
Moreover, the following identities hold in $\widehat{\cA}(X)$:
\begin{equation}
    \label{exp(log)}
    \exp(\log(1+a)) = 1+a \quad \text{and} \quad \log(\exp(a)) = a.
\end{equation}

\subsection{The formal differential equation}
\label{subsec=hom_formal_eq}

Fix a finite subset $\cE$ of $\Br(X) \setminus \{ X_0 \}$ and a control $u=(u_b)_{b \in \cE} \in L^1(\R^+;\K^{|\cE|})$.
We consider the formal differential equation posed on $\widehat{\cA}(X)$:
\begin{equation}
    \label{formalODE_A}
    \dot{S}(t)=S(t) \left( X_0+\sum_{b \in \cE} u_b(t) b \right)
    \quad \text{and} \quad S(0) = 1. 
\end{equation}
In \eqref{formalODE_A}, and in all this work, we identify $b \in \Br(X)$ with its evaluations $\eval(b)$ in $\cL(X)$, or in $\widehat{\cL}(X)$.

\begin{definition}
    \label{def:formalODE}
    The \emph{solution to \eqref{formalODE_A}} is the formal-series-valued function $S: \R^+ \to \widehat{\cA}(X)$, denoted $S(t,X,u)$, whose homogeneous components $S^{\langle n \rangle} : \R^+ \to \cA^n(X)$ are defined, for every $t \geq 0$, by $S^{\langle 0 \rangle}(t) = 1$ and, for every $n \in \N^*$, by
    \begin{equation}
        \label{eq:xn.xn1}
        S^{\langle n \rangle}(t) = \int_0^t \left(
        S^{\langle n-1 \rangle}(\tau) X_0
        +\sum_{b \in \cE} u_b(\tau) S^{\langle n-|b| \rangle}(\tau) b
        \right) \dd\tau
    \end{equation}
    with the convention that $S^{\langle n \rangle} = 0$ for $n < 0$.
\end{definition}

Iterating this integral formula yields an expansion of $S(t,X,u)$ in $\widehat{\cA}(X)$, called the Chen series \cite{MR0073174,zbMATH03126609} and popularized in control theory by \cite{MR613847}.
The following proposition emphasizes an important homogeneity property of this solution.

\begin{proposition}[Dilation $\Lambda$]
    \label{Prop:S-hom}
    For $u \in L^1(\R^+;\K^{|\cE|})$ and $a \in \K$, define $\lambda_a u$ by $(\lambda_a u)_b := a^{n_1(b)} u_b$ for $b \in \cE$.
    Then, for all $t \geq 0$, $S(t,X,\lambda_a u) = \Lambda_a S(t,X,u)$ where $\Lambda_a : \widehat{\cA}(X) \to \widehat{\cA}(X)$ is the unique morphism of algebras such that $\Lambda_a(X_0) = X_0$ and $\Lambda_a(X_1) = a X_1$.
\end{proposition}

\begin{proof}
    If $u$ is continuous, since $\Lambda_a$ is a morphism of algebras, one has
    \begin{equation}
        \frac{\dd}{\dd t} [ \Lambda_a S ]
        = \Lambda_a \dot{S}
        = \Lambda_a \Big[ S \big(X_0 + \sum_{b \in \cE} u_b b\big) \Big] 
        = (\Lambda_a S) \big(X_0 + \sum_{b \in \cE} (\lambda_a u)_b b \big)
    \end{equation}
    where we used that $\Lambda_a X_0 = X_0$ and $\Lambda_a(u_b b) = u_b \Lambda_a(b) = u_b a^{n_1(b)} b = (\lambda_a u)_b b$.
    This proves that $\Lambda_a S(t,X,u)$ satisfies the same differential equation as $S(t,X,\lambda_a u)$. 
    So both are equal.
    When $u \in L^1$, we pass to the limit in the Duhamel formulation.
\end{proof}

\subsection{A Magnus-type representation formula}
\label{subsec:Magnus_A}

It is known since the work of Magnus \cite{Magnus54ote} in 1954 that the solution to the linear system \eqref{formalODE_A} can be written as the exponential of a Lie series.
We describe here a slightly different representation which better reflects the particular role of $X_0$ (see \cite{P1} for an in-depth comparison of both representation formulas).
We need the following definitions.

\begin{definition}
    \label{def:oslash}
    Let $\cL^\oslash(X)$ denote the Lie ideal of $\cL(X)$ generated by $X_1$, and $\widehat{\cL}^\oslash(X)$ the associated space of formal power series.
    One has $\cL(X) = \K X_0 \oplus \cL^\oslash(X)$.
    Moreover, Lazard's elimination (see~\cite{Viennot1978}) proves that $\cL^\oslash(X)$ is actually isomorphic to $\cL(\{ M_j \mid j \in \N \})$, the free Lie algebra generated by the family $(M_j)_{j \in \N}$ (recall \cref{def:Mnu-Wj}).
\end{definition}

\begin{definition}
    \label{def:Omega_b}
    For $b \in \Br(X) \setminus \{ X_0 \}$, we define the map $\Omega_b:\widehat{\cL}^\oslash(X) \rightarrow \widehat{\cL}^\oslash(X)$ by
    \begin{equation}
        \Omega_b(\cZ) := \sum_{n=0}^\infty \frac{(-1)^n B_n}{n!} \ad_{\cZ}^n(b),
    \end{equation}
    where $(B_n)_{n\in\N}$ are the Bernoulli numbers (see \cref{subsec:Bernoulli}). 
\end{definition}

\begin{proposition}
    \label{Prop:Magnus}
    For $t \in \R^+$ and $u \in L^1(\R^+;\K^{|\cE|})$, define
    $\cZ(t,X,u) := \log( \exp( -t X_0 ) S(t,X,u))$ in $\widehat{\cA}(X)$, so that one has
    \begin{equation}
        \label{Magnus_1.1}
        S(t,X,u)=\exp(t X_0) \exp(\cZ(t,X,u)).
    \end{equation}
    Then the map $\cZ$ inherits the same homogeneity properties as the map $S$ (see \cref{Prop:S-hom}).
    Moreover, for every $u \in L^1(\R^+;\K^{|\cE|})$,
    $\cZ( \cdot, X;u)$ solves the formal differential equation on $\widehat{\cL}^\oslash(X)$
    \begin{equation}
        \label{EDO_Z_1.1}
        \dot{\cZ}(t)=[\cZ(t),X_0] + \sum_{b \in \cE} u_b(t) \Omega_{b}(\cZ(t))
    \end{equation}
    with initial condition $\cZ(0)=0$.
    Thus, for every $t \in \R^+$, $\cZ(t,X,u) \in \widehat{\cL}^\oslash(X)$.
    
    As in \cref{def:formalODE}, the formal differential equation \eqref{EDO_Z_1.1} is understood in projection on the finite dimensional spaces $\cA^n(X)$ and in the integral sense.
\end{proposition}

\begin{proof}
    In this proof, to simplify notations, we write $\cZ(t)$ and $S(t)$ instead of $\cZ(t,X,u)$ and $S(t,X,u)$. 
    We have $\exp(\cZ(t))=\exp(-tX_0)S(t)$. 
    Thus we deduce from \eqref{formalODE_A} that
    \begin{equation}
        \exp(-\cZ(t)) \frac{\dd}{\dd t} \exp(\cZ(t)) =
        - \exp(-\cZ(t)) X_0 \exp(\cZ(t)) +X_0+ \sum_{b \in \cE} u_b(t) b .
    \end{equation}
    We deduce from the formula \eqref{dot{z}} that
    \begin{equation}
        \dot{\cZ}(t)=g_0(\cZ)+\sum_{b \in \cE} u_b(t) \Omega_b(\cZ),
    \end{equation}
    where
    \begin{equation}
        \begin{split}
            g_0(\cZ)
        & :=
        \sum_{n=0}^\infty \frac{(-1)^n B_n}{n!} \ad_{\cZ}^n\big( X_0 - 
        \exp(-\cZ) X_0 \exp(\cZ) \big) \\
        & =
        - \sum_{n=0}^\infty \frac{(-1)^n B_n}{n!} \ad_{\cZ}^n
        \sum_{k=1}^\infty \frac{(-1)^k}{k!} \ad_{\cZ}^k(X_0)
        \\ & =
        - \sum_{m=1}^\infty \frac{(-1)^m}{m!} \ad_{\cZ}^m (X_0)
        \sum_{n=0}^{m-1} \binom{m}{n} B_n
        =
        \ad_{\cZ}(X_0)
        \end{split}
    \end{equation}
    and the last equality results from \eqref{eq:bernoulli.1}.
\end{proof}

\begin{remark}
    The usual Magnus formula for $S$ is of the form $S = \exp Z$ or equivalently $Z = \log S$. 
    This is why we call \eqref{Magnus_1.1} a ``Magnus-type'' representation formula for $S$.
    In this article, using $\cZ$ instead of $Z$ allows to work with a control system \eqref{EDO_Z_1.1} having an equilibrium at $0$. 
\end{remark}

\subsection{Truncations and the Magnus system}
\label{subsec:Trunc}

\begin{definition}[Free nilpotent algebra]
    For $N \in \N$, we define
    \begin{equation}
        \cA_N(X) = \underset{n \in \intset{0, N}}{\oplus} \cA^n(X),
    \end{equation}
    the space of polynomials with degree $\leq N$ and $\pi_N: \widehat{\cA}(X) \to \cA_N(X)$ the canonical surjection (truncation).
    $\cA_N(X)$ is not a subalgebra of $\widehat{\cA}(X)$, because it is not closed under product, but it can be given a structure of algebra by defining the multiplication of two elements $a,b \in \cA_N(X)$ by $\pi_N(ab)$ i.e.\ the multiplication on $\cA_N(X)$ is the same as on $\widehat{\cA}(X)$ except that monomials of degree $>N$ are discarded. Then $\pi_N$ is a morphism of algebras:
    \begin{equation}
        \label{piN(prod)}
        \forall S, S' \in \widehat{\cA}(X),\qquad
        \pi_N ( S S') = \pi_N(S) \pi_N(S').
    \end{equation}
\end{definition}

\begin{definition}[Free nilpotent Lie algebra]
    The Lie subalgebra of $\cA_N(X)$ generated by $X$ is $\cL_N(X) := \pi_N( \cL(X) )$. 
    The Lie ideal of $\cL_N(X)$ generated by $\{X_1\}$ is $\cL^\oslash_N(X) := \pi_N( \cL^\oslash(X))$. 
    Moreover $\cL_N(X) = \K X_0 \oplus \cL^\oslash_N(X)$.
\end{definition}

To simplify the writing, we use the same notations $\exp$ and $\log$ for the truncated versions on $\cA_N(X)$ of the maps $\exp$ and $\log$ defined in \eqref{def:exp_log} i.e.\ for $a \in \cA_N(X)$ with $a^{\langle0\rangle}=0$ we write $\exp(a)$ instead of $\pi_N(\exp(a))$ and
$\log(1+a)$ instead of $\pi_N(\log(1+a))$. 
Since $\pi_N$ is a morphism of algebras, the equalities \eqref{exp(log)} hold in $\cA_N(X)$.
For $b \in \Br(X) \setminus \{X_0\}$, we also denote $\Omega_b$ the truncated version on $\cL^\oslash_N(X)$ of the map $\Omega_b$ defined in \eqref{def:Omega_b}. 
Since $\pi_N$ is a morphism of algebras, we deduce from \cref{Prop:Magnus} the following results.

\begin{proposition}
    \label{Prop:trunc}
    Let $N\in\N$. 
    For $t \in \R^+$ and $u \in L^1(\R^+;\K^{|\cE|})$, we define
    \begin{equation}
        \begin{aligned}
            S_N(t,X,u) & := \pi_N(S(t,X,u)) \in \cA_N(X),
            \\
            \cZ_N(t,X,u) & := \pi_N( \cZ(t,X,u)) \in \cL^\oslash_N(X).
        \end{aligned}
    \end{equation}
    Then the maps $S_N$ and $\cZ_N$ inherit the same homogeneity properties as $S$ (see \cref{Prop:S-hom}).
    For $t \in \R^+$ and $u \in L^1(\R^+;\K^{|\cE|})$, 
    \begin{equation}
        \label{SN=exp(ZN)}
        \begin{aligned}
        \cZ_N(t,X,u)&=\log( \exp(- t X_0) S_N(t,X,u)),\\
        S_N(t,X,u)&=\exp(t X_0) \exp ( \cZ_N(t,X,u) ).
        \end{aligned}
    \end{equation}
    The maps $t \mapsto S_N(t,X,u)$ and $t \mapsto \cZ_N(t,X,u)$ solve the ordinary differential equations
    \begin{gather}
        \label{EDO_SN}
        \dot{S}_N(t) = S_N(t) \big( X_0+\sum_{b \in \cE} u_b(t) b \big) 
        \quad \text{ and } \quad S_N(0) = 1,
        \\
        \label{EDO_ZN_1.1}
        \dot{\cZ}_N(t)  =[\cZ_N(t),X_0] + \sum_{b \in \cE} u_b(t) \Omega_{b}(\cZ_N(t))
        \quad \text{and} \quad \cZ_N(0) = 0.
    \end{gather}
\end{proposition}

The systems \eqref{EDO_SN} and \eqref{EDO_ZN_1.1}, set on the finite dimensional spaces $\cA_N(X)$ and $\cL^\oslash_N(X)$, have the form \eqref{eq:syst-f0-f1-fm}. 
Thus, we have a notion of solution when $u \in \fD$ (see \cref{s:impulsive}). 
Then the equalities \eqref{SN=exp(ZN)} still hold (pass to the limit in the equality for solutions associated with the $L^1$ regularizations). 
We will call the equation \eqref{EDO_ZN_1.1} the Magnus system.

\subsection{Lie algebra rank condition}
\label{subsec:LARC_Magnus_syst}

\begin{proposition}
    \label{Prop:LARC_Z}
    Let $N \in \N^*$ and $g_0 := [\cdot, X_0]$ and $g_1 := \Omega_{X_1}$  smooth vector fields on $\cL^\oslash_N(X)$.
    As in \cref{def:eval-fb}, denote $b \mapsto g_b$ the Lie algebra morphism generated by $X_i \mapsto g_i$.
    Then $g_b=\Omega_b$ for every $b \in \Br(X) \setminus \{X_0\}$.
    As a consequence $(\Lie \left\{g_0,g_1\right\})(0)=\cL^\oslash_N(X)$.
\end{proposition}

\begin{proof}
    \emph{Step 1: We prove that, for every $b,b' \in \Br(X) \setminus \{X_0\}$ then $[\Omega_{b},\Omega_{b'}] = \Omega_{[b,b']}$.} 
    
    For $b \in \Br(X) \setminus \{X_0\}$, we define a smooth vector field $F_b$ on $\cA_N(X)$ by $F_b(S)=Sb$. Then, for every $b,b' \in \Br(X)$, we have $[F_{b},F_{b'}]=F_{[b,b']}$.
    The set $G^\oslash_N(X)=\{\exp(\cZ) \mid \cZ \in \cL^\oslash_N(X)\}$ is a submanifold of $\cA_N(X)$ with dimension $\dim \cL^\oslash_N(X)$ and the $F_b$ are vector fields on $G^\oslash_N(X)$. 
    The logarithm is a diffeomorphism from $G^\oslash_N(X)$ to $\cL^\oslash_N(X)$ with reciprocal $\exp: \cL^\oslash_N(X) \to G^\oslash_N(X)$.
    With this formalism, for $b \in \Br(X) \setminus \{X_0\}$, the vector field $\Omega_b$ on $\cL^\oslash_N(X)$ is the push-forward of $F_b$ by the logarithm: $\Omega_b (\cZ) = (\log_* F_b)(\cZ) = D \log( \exp(\cZ) ) F_b(\exp(\cZ))$. 
    Thus for all $b,b' \in \Br(X) \setminus \{X_0\}$, $[\Omega_{b},\Omega_{b'}]=\log_*[F_{b},F_{b'}]=\log_* F_{[b,b']}=\Omega_{[b,b']}$ (see \cite[Lemma 89]{P1}).

    \medskip
    \noindent \emph{Step 2: We prove by induction on $j \in \N$ that $g_{M_j}=\Omega_{M_j}$.} The initialization for $j=0$ results from the definitions: $g_{M_0}=g_{X_1}=g_1=\Omega_{X_1}$. Let $j \in \N$. We assume $g_{M_j}=\Omega_{M_j}$. Then
    \begin{equation}
        \begin{split}
            g_{M_{j+1}}(\cZ) & = [g_{M_j},g_0](\cZ)=[\Omega_{M_j},g_0](\cZ) =
        [\Omega_{M_j}(\cZ),X_0] - D(\Omega_{M_j})(\cZ) \cdot [\cZ,X_0]
        \\ & =
        [M_j,X_0] +
        \sum_{n=1}^\infty\frac{(-1)^n B_n}{n!} \left(
        [\ad_{\cZ}^n(M_j),X_0] - \sum_{k=0}^{n-1} \ad_{\cZ}^{n-1-k} \ad_{[\cZ,X_0]} \ad_{\cZ}^k(M_j)
        \right)
        \\ &
        = M_{j+1} + \sum_{n=1}^\infty\frac{(-1)^n B_n}{n!}
        \ad_{\cZ}^n(M_{j+1}) = \Omega_{M_{j+1}}(\cZ)
        \end{split}
    \end{equation}
    where the last line results from the Jacobi identity.

    \medskip
    \noindent \emph{Step 3: We prove that for every $b \in \Br(X) \setminus \{X_0\}$ then $g_b=\Omega_b$.} 
    Recalling \cref{def:oslash} of $\cL^\oslash(X)$, the evaluation in $\cL^\oslash_N(X)$ of $b \in \Br(X) \setminus \{X_0\}$ is a linear combination of iterated Lie brackets of the $M_j$ for $j \in \N$. 
    Thus, by Step 1 and Step 2, $g_b=\Omega_b$. 
    In particular $g_b(0)=b$, which entails the Lie algebra rank condition.
\end{proof}

\section{Control, splitting and the Magnus system}
\label{Sec:Control_split_Magnus}

In this section, $\K=\R$ or $\C$ is the base field of all the vector spaces, and $\cE$ is a finite subset of $\Br(X) \setminus \{ X_0 \}$.
We define the ``controllability'' of the Magnus system \eqref{EDO_ZN_1.1} in \cref{subsec:D-control_Magnus_syst}.
We state error estimates in \cref{subsec:error_Magnus_syst}.
Finally, in \cref{sec:link}, we prove the fundamental link between the controllability of the Magnus system, the existence of an $(\R^+,\K)$ splitting method of order $N$ involving $X_0$ and the elements of $\cE$, and the controllability of the system \eqref{EDOf0fA}.

\subsection{Controllability of the Magnus system}
\label{subsec:D-control_Magnus_syst}

The Magnus system \eqref{EDO_ZN_1.1} is a control-affine system of the form studied in \cref{sec:Prerequis}.
For $T > 0$ and either $\cU = \fD$ or $\cU = L^1$, define the set of reachable states for system \eqref{EDO_ZN_1.1} as
\begin{equation}
    R_T(\cU) := \{ \cZ_N(T,X,u) \mid u \in \cU \} \subset \cL_N^\oslash(X).
\end{equation}
The homogeneity properties of \eqref{EDO_ZN_1.1} entail the following result.

\begin{proposition}
    \label{Prop:equivalence}
    Let $\cU = \fD$ or $\cU = L^1$.
    For system \eqref{EDO_ZN_1.1}, the following notions are equivalent, and will be called \emph{$\cU$-controllability} of system \eqref{EDO_ZN_1.1}:
    \begin{flalign*}
        & \quad - \quad \text{small-time local exact $\cU$-controllability:} 
        && \forall T > 0, \quad 0 \in \operatorname{int} R_T(\cU), &\\
        & \quad - \quad \text{small-time local approximate $\cU$-controllability:} 
        && \forall T > 0, \quad 0 \in \operatorname{int} \overline{R_T(\cU)}, & \\
        & \quad - \quad \text{small-time global exact $\cU$-controllability:}
        && \forall T > 0, \quad R_T(\cU) = \cL^\oslash_N(X), & \\
        & \quad - \quad \text{small-time global approximate $\cU$-controllability:}
        && \forall T > 0, \quad \overline{R_T(\cU)} = \cL^\oslash_N(X), 
    \end{flalign*}
    Moreover, system \eqref{EDO_ZN_1.1} is $\fD$-controllable if and only if it is $L^1$-controllable.
\end{proposition}

\begin{proof}
    We prove the equivalence of these 8 notions.

    \medskip
    \emph{Exact/approximate local controllability.}
    Recalling \cref{def:Omega_b}, we obtain from \eqref{EDO_ZN_1.1} that $\pi_1 \dot{\cZ_N} = 0$ if $X_1 \notin \cE$.
    Thus approximate controllability implies that $X_1 \in \cE$.
    By \cref{Prop:LARC_Z}, $(\Lie \left\{ g_0, g_1 \right\})(0) = \cL^\oslash_N(X)$.
    Thus, by \cref{p:approx=exact}, for $\cU = \fD$ or $L^1$, the local exact/approximate $\cU$-controllability are equivalent.

    \medskip
    \emph{$\fD$/$L^1$.}
    By \cref{p:D=L1}, for any $T > 0$, $\overline{R_T(\fD)} = \overline{R_T(L^1)}$, so each of the four approximate $\fD$ notions is equivalent to its $L^1$ counterpart.
    
    \medskip
    \emph{Local/global.}
    Since we don't require smallness of the intermediate states or of the controls used, the \emph{global} statements imply the \emph{local} ones.
    The scaling properties of system \eqref{EDO_ZN_1.1} entail the converse implications.
    For any $a > 0$, with the notations of \cref{Prop:S-hom}, by \cref{Prop:Magnus}, $\cZ_N(T,X,\lambda_a u) = \Lambda_a \cZ_N(T,X,u)$.
    Since $\Lambda_a$ is continuous with inverse $\Lambda_{a^{-1}}$, for any $a > 0$, there holds $\Lambda_a(R_T(\cU)) = R_T(\cU)$ and $\Lambda_a(\overline{R_T(\cU)}) = \overline{R_T(\cU)}$.
    Finally, for any $\cZ^* \in \cL^\oslash_N(X)$, since $\cZ^*$ has no component along $X_0$, $\Lambda_a(\cZ^*) \to 0$ as $a \to 0$.
    Hence each of the four local notions implies its global counterpart.
\end{proof}

\subsection{Error estimates}
\label{subsec:error_Magnus_syst}

We will need the following error estimates, proved in \cref{sec:error-proofs}.
Recall that $x(t;f,u,x_0)$ denotes the solution to \eqref{EDOf0fA}, $S_N(t,X,u)$ the solution to \eqref{EDO_SN} and $\cZ_N(t,X,u)$ the solution to \eqref{EDO_ZN_1.1}.

\begin{proposition}
    \label{prop:error-split}
    Let $f_0, f_1 \in \cC^\infty(\K^d;\K^d)$, $x_0 \in \K^d$, $u \in \fD$ or $L^1$ and $N \in \N^*$.
    Assume that $u$ is such that $S_N(1,X,u)=\exp(X_0+X_1)$.
    As $T \to 0$,
    \begin{equation}
        \label{eq:prop-error-split}
        x(1;Tf,u,x_0)=e^{T(f_0+f_1)}(x_0) + O\left(T^{N+1}\right).
    \end{equation}
\end{proposition}

\begin{proposition}
    \label{prop:error-Z}
    Let $f_0, f_1 \in \cC^\infty(\K^d;\K^d)$, $x_0 \in \K^d$, $u \in \fD$ or $L^1$ and $N \in \N^*$.
    As $T \to 0$,
    \begin{equation}
        \label{eq:error-x-exp(z)}
        x(1;Tf,u,x_0) = e^{\cZ_N(1,Tf,u)} e^{T f_0}(x_0) + O\left(T^{N+1}\right)
    \end{equation}
    where $\cZ_N(1,Tf,u) \in \cC^\infty(\K^d;\K^d)$ is the image of $\cZ_N(1,X,u)$ by the morphism of $\K$ Lie algebras from $\cL(X)$ to $\cC^\infty(\K^d;\K^d)$ mapping $X_0$ to $T f_0$ and $X_1$ to $T f_1$ (see \cref{lem:C-morphism-L} when $\K = \C$).
\end{proposition}

\begin{remark}
    \label{rem:error-Z-nilpotent}
    Let $f_0, f_1 \in \cC^\infty(\K^d;\K^d)$ and $N \in \N^*$.
    Assume that the Lie algebra $\Lie \left\{f_0,f_1\right\}$ is nilpotent of step $R$.
    For all $N \geq R$, $x_0 \in \K^d$ and $u \in \fD$ or $L^1$, for $T > 0$ small enough,
    \begin{equation}
        x(1;Tf,u,x_0) = e^{\cZ_N(1,Tf,u)} e^{T f_0} (x_0).
    \end{equation}
    Since $\Lie \left\{f_0,f_1\right\}$ is nilpotent of step $R$, for all $N \geq R$, $\cZ_N(1,Tf,u) = \cZ_R(1,Tf,u)$.
    When $f_0$ and $f_1$ are analytic, the maps $T \mapsto x(1;Tf,u,x_0)$ and $T \mapsto e^{\cZ_R(1,Tf,u)} e^{T f_0} (x_0)$ are also analytic.
    By \cref{prop:error-Z}, the error estimate \eqref{eq:error-x-exp(z)} holds for all $N$.
    Thus both maps are equal.

    The result is also true without assuming analyticity of the vector fields.
    This can be proved by defining $y(t) := e^{\cZ_N(t,Tf,u)} e^{t T f_0} (x_0)$ and remarking that it satisfies the same ODE as $x(t;Tf,u,x_0)$.
    It is a particular case of \cite[Proposition 120]{P1}.
\end{remark}

\subsection{Fundamental link}
\label{sec:link}

We emphasize a fundamental link between the $\fD$-controllability of the Magnus system \eqref{EDO_ZN_1.1}, the existence of an $(\R^+,\K)$ splitting method of order $N$ involving $X_0$ and the elements of~$\cE$ and the controllability of the system \eqref{EDOf0fA}.

We first remark that an $(\R^+,\K)$ splitting method of order $N$ involving $X_0$ and the enables brackets~$\cE$ can be interpreted as a trajectory of the system \eqref{EDO_SN} reaching a particular target.

\bigskip

\begin{theorem}
    \label{prop:splitting-iff}
    Let $N \in \N^*$.
    The following statements are equivalent:
    \begin{enumerate}
        \item \label{it:pCSN-1}
        There exists $u \in \fD$ such that $S_N(1,X,u)=\exp (X_0+X_1)$ in $\cA_N(X)$.
        \item \label{it:pCSN-2}
        There exists $u \in \fD$ such that $\cZ_N(1,X,u)=\log( \exp( -X_0) \exp ( X_0+X_1) )$ in $\cL^\oslash_N(X)$.
        \item \label{it:pCSN-3}
        There exists an $(\R^+,\K)$ splitting method of order $N$ involving $X_0$ and the elements of $\cE$.
    \end{enumerate}
\end{theorem}

\begin{proof}
    The equivalence between \cref{it:pCSN-1} and \cref{it:pCSN-2} results from \eqref{SN=exp(ZN)}.

    \medskip

    \noindent \emph{\cref{it:pCSN-1} $\Rightarrow$ \cref{it:pCSN-3}:}
    Assume that there exists a control $u \in \fD$ such that $S_N(1,X,u)=\exp (X_0+X_1)$ in $\cA_N(X)$.
    Let $f_0, f_1 \in \cC^\infty(\K^d;\K^d)$ and $x_0 \in \K^d$.
    By \cref{prop:error-split},
    \begin{equation}
        x(1; T f,u,x_0)=e^{T(f_0+f_1)}x_0 + \underset{T \to 0}{O}(T^{N+1})
    \end{equation}
    and the left-hand side is an $(\R^+,\K)$ splitting method involving $X_0$ and the elements of $\cE$.

    \medskip
    
    \noindent \emph{\cref{it:pCSN-3} $\Rightarrow$ \cref{it:pCSN-2}:}
    Assume that there exists an $(\R^+,\K)$ splitting method of order $N$ involving $X_0$ and the elements of $\cE$. 
    Let $u \in \fD$ be the associated control, i.e.\ for all smooth vector fields $f_0,f_1$ on $\K^d$ and for every $x_0 \in \K^d$, the solution to \eqref{EDOf0fA} satisfies $x(1;Tf,u,x_0)=e^{T(f_0+f_1)} x_0 +O(T^{N+1})$ as $T \to 0$. 
    This holds in particular with the vector fields $g_0, g_1$ on $\cL^\oslash_N(X)$ defined in \cref{Prop:LARC_Z}, involved  in \eqref{EDO_ZN_1.1} and $x_0=0$.
    Thus, the following equality holds in $\cL^\oslash_N(X)$
    \begin{equation} \label{Z-e^=O(TM+1}
        x(1;Tg,u,0) - e^{T(g_0+g_1)} (0) = \underset{T \to 0}{O}(T^{N+1}).
    \end{equation}
    
    On the one hand, let $\underline{\cE} = \{ X_1 \}$ and $\underline{u} = \underline{u}_{X_1} = 1$. 
    Using~\eqref{EDO_ZN_1.1} and \eqref{SN=exp(ZN)}, we obtain the following equality in $\cL^\oslash_N(X)$ for all $T \geq 0$:
    \begin{equation}
        \begin{split}
            e^{T(g_0+g_1)}(0)=\cZ_N(T,X,\underline{u})
            & =\log( \exp(-T X_0) S_N(T,X,\underline{u})) \\
            & =\log( \exp ( -T X_0) \exp ( T(X_0+X_1) )).
        \end{split}
    \end{equation}
    In particular, $e^{T(g_0+g_1)} (0)$ is a polynomial in $T$ of degree $\leq N$.
    
    On the other hand, let $\mu_T$ be the Lie algebra morphism on $\cL(X)$ defined by $\mu_T(X_i) = T X_i$.
    A computation proves that $x(1;Tg,u,0) = \mu_T(\cZ_N(1,X,u))$, so is a polynomial in $T$ of degree $\leq N$
    
    Therefore, the estimate \eqref{Z-e^=O(TM+1} implies that
    \begin{equation}
        \forall T \in \R, \qquad x(1;Tg,u,0) = e^{T(g_0+g_1)} (0).
    \end{equation}
    Finally, for $T = 1$, we get $\cZ_N(1,X,u) = x(1;g,u,0) = \log (\exp ( -X_0 ) \exp( X_0+X_1) )$.
\end{proof}

\begin{theorem}
    \label{prop:ZN=>both}
    Let $N \in \N^*$.
    Assume that system \eqref{EDO_ZN_1.1} is controllable (see \cref{Prop:equivalence}).
    Then
    \begin{enumerate}
        \item \label{it:prop:ZN=>1}
        there exists an $(\R^+,\K)$ splitting method of order $N$ involving $X_0$ and the elements of $\cE$,
        \item \label{it:prop:ZN=>2}
        for every pair of smooth vector fields $f_0, f_1$ on a neighborhood of $0$ in $\K^d$ such that $f_0(0)=0$ and $\cL_N(f)(0) = \K^d$,  system \eqref{EDOf0fA} is small-state small-time locally controllable.
    \end{enumerate}
\end{theorem}

\begin{proof}
    We prove \cref{it:prop:ZN=>1}.
    If system \eqref{EDO_ZN_1.1} is controllable, in particular, there exists $u \in \fD$ such that $\cZ_N(1,X,u) = \log(\exp(-X_0) \exp(X_0+X_1))$ in $\cL^\oslash_N(X)$.
    Thus, by \cref{prop:splitting-iff}, there exists an $(\R^+,\K)$ splitting method of order $N$ involving $X_0$ and the elements of $\cE$.

    We prove \cref{it:prop:ZN=>2}.
    Let $r$ be the dimension of $\cL^\oslash_N(X)$ as an $\K$-vector space and $b_1,\dots,b_r \in \Br(X)$ whose evaluations in $\cL^\oslash_N(X)$ form a basis of this space over $\K$.
    Since \eqref{EDO_ZN_1.1} is controllable, for each $j \in \intset{1,r}$, there exist controls $\underline{u}^{j,\pm}=(\underline{u}^{j,\pm}_b)_{b \in \cE} \in L^1((0,1);\K^{|\cE|})$ such that
    \begin{equation}
        \label{biortho-bis}
        \cZ_N(1,X,\underline{u}^{j,\pm})=\pm b_j.
    \end{equation}
    For $T>0$, define $u^{T,j,\pm} = (u^{T,j,\pm}_b)_{b \in \cE}$ as
    \begin{equation}
        \forall b \in \cE, \forall t \in [0,T], \qquad
        u^{T,j,\pm}_b(t) := T^{|b|-1} \underline{u}^{j,\pm}_b \left( \frac{t}{T} \right).
    \end{equation}
    In particular, for $0 \leq T \leq 1$, one has $\| u^{T,j,\pm} \|_{L^1(0,T)} \leq T \| \underline{u}^{j,\pm} \|_{L^1(0,1)}$.
    
    Let $d \in \N^*$ and $f_0, f_1$ be smooth vector fields on a neighborhood of $0 \in \K^d$ such that $f_0(0)=0$.
    As $T \to 0$, the solution to \eqref{EDOf0fA} satisfies
    \begin{equation}
        \begin{aligned}
            x(T;f,u^{T,j,\pm},0) 
            & = x(1;Tf,\underline{u}^{j,\pm},0) 
            && \text{by time rescaling} \\
            & = e^{\cZ_N(1,Tf,\underline{u}^{j,\pm})} e^{T f_0}(0) + O\left(T^{N+1}\right) 
            && \text{by \eqref{eq:error-x-exp(z)} of \cref{prop:error-Z}} \\
            & = e^{\pm T^{|b_j|} f_{b_j}}(0) + O\left(T^{N+1}\right) 
            && \text{by \eqref{biortho-bis} and $f_0(0)=0$} \\
            & = \pm T^{|b_j|} f_{b_j}(0) + O\left(T^{2|b_j|} + T^{N+1}\right) 
            && \text{by Taylor expansion} \\
            & = \pm T^{|b_j|} f_{b_j}(0) + O\left(T^{|b_j|+1}\right)
            && \text{since $|b_j| \geq 1$.}
        \end{aligned}
    \end{equation}
    In control theory, one says that $\pm f_{b_j}(0)$ are ``tangent vectors'' of order $|b_j|$.

    \emph{Case $\K = \R$.}
    Since $\cL_N(f)(0)=\R^d$ and $f_0(0)=0$, one can extract a subset $B\subset\{b_1,\dots,b_r\}$ such that $\{f_b(0)\mid b\in B\}$ is a basis of $\R^d$.
    By \cref{prop:tgt-Brouwer}, this implies that \eqref{EDOf0fA} is small-state small-time locally controllable in the sense of \cref{def:STLC}.

    \emph{Case $\K = \C$.}
    Since \eqref{EDO_ZN_1.1} is controllable, for each $j \in \intset{1,r}$, there also exist two controls $\underline{u}^{j,\pm i} \in L^1((0,1);\K^{|\cE|})$ such that $\cZ_N(1,X,\underline{u}^{j,\pm i}) = \pm i b_j$.
    As above, this entails that $\pm i f_{b_j}(0)$ are tangent vectors.
    Since $\cL_N(f)(0) = \C^d$ and $f_0(0) = 0$, we can extract a subset $B \subset \{ b_1, \dotsc, b_r \}$ such that $\{f_b(0) \mid b\in B\} \cup \{ i f_b(0) \mid b \in B \}$ is an $\R$ basis of $\C^d$.
    \cref{prop:tgt-Brouwer} concludes.
\end{proof}

\section{Complex controls and complex methods}
\label{Sec:Complex}

In this section, $\K=\C$ and we prove \cref{thm:main-complex-split,thm:main-complex-control}. 
After a preliminary remark on complex-valued systems, we reduce in \cref{sec:complex-reduction} these theorems to the controllability of the Magnus system, of which we give different proofs in \cref{sec:C-Sussmann,subsec:Complex_chow,subsec:Complex_extension}.

\subsection{A remark on complex-valued systems}
\label{sec:holomorphic-systems}

For complex-valued systems, some care is needed.
In particular, we explain why we always implicitly assume that the vector fields are \emph{holomorphic} in this paper when defined on $\C^d$.

\paragraph{Lie brackets of holomorphic vector fields.}

For $f, g \in \Hol(\C^d;\C^d)$, one can define their Lie bracket exactly as announced in \cref{def:Lie-fields} by setting
\begin{equation}
    \label{eq:Lie-C}
    [f,g](z) 
    := Dg(z) f(z) - Df(z) g(z)
    = \sum_{k = 1}^d f_k(z) \frac{\partial g}{\partial z_k} - g_k(z) \frac{\partial f}{\partial z_k},
\end{equation}
where the differentials are the $\C$-linear derivatives of $f$ and $g$.
Then $[f,g] \in \Hol(\C^d;\C^d)$, and the bracket $[\cdot, \cdot]$ defines a $\C$-bilinear operation (which is also anti-symmetric and satisfies the Jacobi identity).
In particular, this entails the existence of the following morphisms.

\begin{lemma}
    \label{lem:C-morphism-L}
    Given $f_0, f_1 \in \Hol(\C^d;\C^d)$, there exists a unique morphism of $\C$-Lie algebras from $\cL_\C(X)$ to $\Hol(\C^d;\C^d)$ mapping $X_0$ to $f_0$ and $X_1$ to $f_1$.
\end{lemma}

\begin{lemma}
    \label{lem:C-morphism-A}
    Let $\Op_\C$ denote the algebra of linear endomorphisms of $\Hol(\C^d;\C)$.
    We identify a vector field $f \in \Hol(\C^d;\C^d)$ with the associated first-order differential operator of $\Op_\C$.

    Given $f_0, f_1 \in \Hol(\C^d;\C^d)$, there exists a unique morphism of $\C$ associative algebras from $\cA_\C(X)$ to $\Op_\C$, mapping $X_0$ to $f_0$ and $X_1$ to $f_1$.
\end{lemma}

\paragraph{Realification.}

Let $\iota : \C^d \to \R^d \times \R^d$ be defined by $\iota (x+iy) := (x,y)$.
Given a vector field $f \in \cC^\infty(\C^d;\C^d)$, one can define its \emph{realification} as its pushforward by the diffeomorphism $\iota$: 
\begin{equation}
    (\iota_* f)(x,y) := 
    \begin{pmatrix}
        \Re f(x+iy) \\
        \Im f(x+iy)
    \end{pmatrix}.
\end{equation}
This defines a vector field $\iota_* f \in \cC^\infty(\R^{2d};\R^{2d})$.

If $f,g \in \Hol(\C^d;\C^d)$, the Cauchy--Riemann relations imply that
\begin{equation}
    \iota_* [f,g]_\C = [\iota_* f, \iota_* g]_\R,
\end{equation}
where $[\cdot, \cdot]_\C$ is the $\C$-bilinear Lie bracket of holomorphic vector fields defined in \eqref{eq:Lie-C} and $[\cdot, \cdot]_\R$ is the Lie bracket on $\R^{2d}$ of \cref{def:Lie-fields}.

\paragraph{Counterexamples.}

With non-holomorphic vector fields, things break apart.

Given $f, g \in \cC^\infty(\C^d;\C^d)$, smooth but not necessarily holomorphic, one can still try to define their Lie bracket through realification as
\begin{equation}
    [[f,g]] := \iota_*^{-1} \circ [\iota_* f, \iota_* g]
\end{equation}
However, the bracket $[[\cdot, \cdot]]$ is not $\C$-bilinear.
For instance, take $d := 1$ and $f(z) := \Re(z)$.
Then $\iota_* f(x,y) = (x,0)$ and $\iota_* (if) (x,y) = (0,x)$, which entails that
\begin{equation}
    \label{eq:[f,if]}
    [[f, if]] = i f \neq 0 = i [[f, f]].
\end{equation}
In particular, this prevents one from constructing a morphism of $\C$ Lie algebras from $\cL(X)$ to $\cC^\infty(\C^d;\C^d)$ mapping $X_0$ to $f_0$ and $X_1$ to $f_1$ when $f_0$ and $f_1$ are not holomorphic.

\paragraph{Flows.}

First, realification is compatible with flows (even without holomorphy).

\begin{lemma}
    \label{prop:flow-realification}
    Let $f\in \cC^\infty(\C^d;\C^d)$.
    Write $\Phi_t^f$ for the local flow of $\dot{z}=f(z)$ and $\Phi_t^{\iota_* f}$ for the local flow of $\dot{\xi}=(\iota_* f)(\xi)$ on $\R^{2d}$. Then, for all $z_0 \in \C^d$ and $t>0$ for which both sides are defined,
    \begin{equation}\label{eq:flow-conjugacy}
        \iota \circ \Phi_t^f (z_0) = \Phi_t^{\iota_* f} \circ \iota (z_0).
    \end{equation}
\end{lemma}

\begin{proof}
    Fix $z_0\in \C^d$ and set $z(t):=\Phi_t^f(z_0)$, so $\dot z(t)=f(z(t))$ and $z(0)=z_0$.
    Define $\xi(t):=\iota(z(t))\in \R^{2d}$. 
    By the chain rule and the definition of pushforward,
    \begin{equation}
        \dot \xi(t) 
        = D\iota(z(t)) \dot z(t)
        = D\iota(z(t)) f(z(t))
        = (\iota_*f)(\xi(t)),
        \quad \text{and} \quad
        \xi(0) = \iota(z_0).
    \end{equation}
    Thus $\xi(t)$ solves the real ODE with vector field $\iota_* f$ and initial data $\iota(z_0)$.
    By uniqueness of solutions, $\xi(t) = \Phi_t^{\iota_* f}(\iota(z_0))$, which is exactly \eqref{eq:flow-conjugacy}.
\end{proof}

Second, when $f \in \cC^\infty(\C^d;\C^d)$, for $t > 0$ and $\alpha \in \R$, real-time flows of $\alpha f$ and $f$ satisfy the compatibility relation $\Phi^{\alpha f}_t = \Phi^f_{\alpha t}$ by real time reparametrization.
If $f$ is holomorphic, its flow is also well defined for complex times (see \cite[Chapter I, Section 1A]{IlyashenkoYakovenko2008}) and this relation holds for $\alpha \in \C$ by complex time reparametrization.
In particular, the complex flows of the form $e^{\alpha t f}$ in our complex splitting methods can be unambiguously understood as one or the other.

If $f$ is not holomorphic, due to the possible lack of commutation of $f$ and $i f$ (recall \eqref{eq:[f,if]}), this relation may break.
For example, with $d = 1$ and $f(z) = \Re(z)$ as above, on the one hand $\Phi^{if}_t (z_0) = x_0 + i (y_0 + t x_0)$ and on the other hand $\Phi^f_t (z_0) = e^t x_0 + i y_0$. 
This is a real-analytic expression of $t$ which admits a holomorphic extension, albeit different from $\Phi^{if}_t(z_0)$.

\subsection{Reduction to the controllability of the Magnus system}
\label{sec:complex-reduction}

\cref{thm:main-complex-control} states an equivalence of the form STLC $\Leftrightarrow$ LARC, in which the forward implication is classical thus we only prove the reverse one. 
Thanks to \cref{prop:ZN=>both}, the key point is to prove the $\fD$-controllability in $\cL^\oslash_N(X)$ of the associated Magnus system
\begin{equation}
    \label{EDO_ZN_C}
    \dot{\cZ}_N(t) = [\cZ_N(t),X_0] + u(t) \Omega_{X_1}(\cZ_N(t)).
\end{equation}
where $u \in \fD$ is a finite sum of Dirac masses with complex-valued amplitudes.
More precisely, by \cref{prop:ZN=>both}, the following result entails both \cref{thm:main-complex-split,thm:main-complex-control}.

\begin{proposition}
    \label{Prop:ZN_controllable_complexe}
    For every $N \in \N^* $, system \eqref{EDO_ZN_C} is controllable.
\end{proposition}

\cref{Prop:ZN_controllable_complexe} can be deduced from a general result by Sussmann (see \cref{sec:C-Sussmann}).
Since this result is quite involved, we give below two easier self-contained direct proofs.
The first one, in \cref{subsec:Complex_chow}, is very natural for people from control theory because it mimics that of Chow's theorem (\cref{thm:main-Chow-control}).
The second one, in \cref{subsec:Complex_extension}, is perhaps more friendly for people from splitting methods; it is also closer to the other controllability proofs in this article.
Both contain an inversion argument (hidden in the proof of \cref{p:approx=exact} for the second one).

\subsection{A proof based on Sussmann's general result}
\label{sec:C-Sussmann}

With controls $u \in L^1(\R^+;\C)$, one can apply Sussmann's general theorem \cite[Theorem 2.4]{MR872457} to the control-affine system \eqref{EDO_ZN_C}.
Let $N \geq 2$ and $\xi := e^{\frac{2 i \pi}{N}}$.
Consider the morphism of algebra $\Lambda_\xi$ of \cref{Prop:S-hom}. 
With Sussmann's terminology $\Upsilon := \{ \Id , \Lambda_\xi, \dots , \Lambda_\xi^{N-1} \}$ is a ``finite group of input symmetries''.
The $\Upsilon$-fixed elements of $\cL(X)$ are spanned by the $b \in \Br(X)$ such that $n_1(b) \in N \N$, i.e.\ $b=X_0$ or $n_1(b) \in N \N^*$.
In any case, $g_b(0)=0$ in $\cL^\oslash_N(X)$, thus the neutralization condition holds. 
Moreover, since no compensation has to be done, there is no need to provide a one parameter group of dilations.
For $N = 1$, one can use $\Upsilon = \{ \Id, \Lambda_{-1} \}$.

\subsection{A proof similar to that of Chow's theorem}
\label{subsec:Complex_chow}

As for Chow's theorem, the proof uses a trajectory from $0$ to $0$ around which an inversion argument works. 
For Chow's theorem, one uses the time reversibility of the driftless system to build this trajectory. 
For the system with drift \eqref{EDO_ZN_C}, we need another argument, proved in \cref{Prop:a0}.

\begin{definition}[Concatenation]
    \label{Def:concat}
    For $u \in L^1((0,T);\C)$ and $v \in L^1((0,T');\C)$, we denote by $u \diamond v :(0,T+T') \to \C$ their concatenation
    \begin{equation}
        (u \diamond v) (t) =
        \begin{cases}
            u(t) \quad & \text{ if } t \in (0,T), \\
            v(t-T) & \text{ if } t \in (T , T+T').
        \end{cases}
    \end{equation}
    By extension, we also write $u \diamond v$ when $u,v \in \fD$ (and there is no ambiguity with the time intervals).
\end{definition}

\begin{definition}[Valuation]
    For $S \in \cA(X)$, define the valuation of $S$, $\val(S)$ as the largest integer $\ell \geq 0$ such that $S \in \cA^{\ell}(X) + \cA^{\ell+1}(X) + \dotsb$ (or, equivalently, $\pi_{\ell-1} S = 0$).
\end{definition}

\begin{proposition}
    \label{Prop:Concat}
    Let $N, \ell \in\N^*$ and, for $j=1,2$, $u_j \in \fD$ such that $\val (\cZ_N(T_j,X,u_j)) \geq \ell$.
    Then $\cZ_N(T_1+T_2,X,u_1 \diamond u_2)-\cZ_N(T_1,X,u_1)-\cZ_N(T_2,X,u_2)$ has valuation at least $\ell+1$.
\end{proposition}

\begin{proof}
    To simplify the notations, we omit $X$ in $S_N$ and $\cZ_N$. 
    By uniqueness of the solution to~\eqref{EDO_SN}, we have $S_N(T_1 + T_2, u_1 \diamond u_2) = S_N(T_1,u_1) S_N(T_2,u_2)$.
    By \eqref{SN=exp(ZN)}, this implies
    \begin{equation}
        \cZ_N(T_1+T_2,u_1 \diamond u_2)= \log \left( \exp(-T_2 X_0) \exp(\cZ_N(T_1,u_1)) \exp(T_2 X_0) \exp(\cZ_N(T_2,u_2)) \right).
    \end{equation}
    The BCH formula \eqref{eq:BCH} of \cref{prop:BCH} gives the conclusion.
\end{proof}

\begin{proposition}
    \label{Prop:a0}
    Let $N \geq 2$.
    For any $T > 0$ and $u \in \fD$, there exists $v \in \fD$ such that
    \begin{equation}
        \cZ_N(N^N T, X, u \diamond v) = 0.
    \end{equation}
\end{proposition}

\begin{proof}
    To simplify the notations, we omit the $X$ in $\cZ_N$.
    
    Let $\xi_k := e^{\frac{2 i k \pi}{N}}$.
    For any $u \in \fD$, by the homogeneity property of \cref{Prop:Magnus},
    \begin{equation}
        \cZ_N(T, \xi_k u) = \Lambda_{\xi_k} \cZ_N(T,u)
    \end{equation}
    where $\Lambda$ is defined in \cref{Prop:S-hom}.
    In particular, $\val \cZ_N(T,\xi_k u) = \val \cZ_N(T,u)$.
    Define
    \begin{equation}
        \Gamma(u) := u \diamond \xi_1 u \diamond \dotsb \diamond \xi_{N-1} u.   
    \end{equation}
    By \cref{Prop:Concat},
    \begin{equation}
        \val \Big[ \cZ_N(NT, \Gamma(u)) - \sum_{k=0}^{N-1} \cZ_N(T,\xi_k u) \Big] \geq \val \cZ_N(T,u) + 1.
    \end{equation}
    Moreover, by the properties of the roots of unity, $\Lambda_{\xi_0}+ \dotsb + \Lambda_{\xi_{N-1}} = 0$ on $\cL^\oslash_N(X)$.
    Indeed, when $N \geq 2$, for any $b \in \Br(X)$ with $n_1(b) \geq N$, $\pi_N \eval(b) = 0$.
    Hence,
    \begin{equation}
        \sum_{k=0}^{N-1} \cZ_N(T,\xi_k u) = \left(\sum_{k=0}^{N-1} \Lambda_{\xi_k}\right) \cZ_N(T,u) = 0.
    \end{equation}
    By induction, we obtain $\val \cZ_N(N^\ell T, \Gamma^\ell(u)) \geq \ell+1$, so that $\cZ_N(N^N T, \Gamma^N(u)) = 0$.
\end{proof}

Now, we can prove the controllability of \eqref{EDO_ZN_C}.

\begin{proof}[Proof of \cref{Prop:ZN_controllable_complexe}]
    Let $N \geq 2$ (since $\cZ_1 = \pi_1 \cZ_2$, the controllability of \eqref{EDO_ZN_C} for $N = 1$ will follow from the one of $\cZ_2$).
    Let $r$ be the dimension of the $\C$-vector space $\cL^\oslash_N(X)$.
    By \cref{Prop:LARC_Z} and \cref{p:access-C}, there exist $h_1,\dots h_{2r} \in \{g_0, g_1, i g_1\}$ and $\ft^0 \in (0,\infty)^{2r}$ such that $D\Phi(\ft^0)$ has rank $2r$, where
    \begin{equation}
        \Phi : 
        \begin{cases}
            (0,\infty)^{2r} & \to \cL^\oslash_N(X) \\
             \ft = (t_1, \dotsc, t_{2r}) & \mapsto e^{t_{2r} h_{2r}} \dotsb e^{t_1 h_1}(0).
        \end{cases}
    \end{equation}
    Let $\varepsilon>0$ be small enough so that the same property holds for the map $\Phi_\varepsilon(\ft) := e^{\varepsilon g_0} \Phi(\ft)$.
    These points can be seen as trajectories of \eqref{EDO_ZN_C} with ordered impulsive controls.
    More precisely, for any $\ft \in (0,\infty)^{2r}$, there exists a unique pair $(T^{\ft}, u^{\ft}) \in (0,\infty) \times \fD$ such that $\Phi_\varepsilon(\ft) = \cZ_N(T^{\ft},X,u^{\ft})$.
    Let $v^{\ft^0} \in \fD$ given by \cref{Prop:a0} so that $\cZ_N(N^N T^{\ft^0},X, u^{\ft^0} \diamond v^{\ft^0})=0$. 
    Then the map
    \begin{equation}
        \Psi : 
        \begin{cases}
            (0,\infty)^{2r} & \to \cL^\oslash_N(X) \\
            \ft = (t_1, \dotsc, t_{2r}) & \mapsto \cZ_N( T^{\ft} + (N^N-1)T^{\ft^0} , X , u^{\ft} \diamond v^{\ft^0})
        \end{cases}
    \end{equation}
    satisfies $\Psi(\ft^0) = 0$ and $D\Psi(\ft^0)$ is surjective.
    Thus $\Psi$ is a local diffeomorphism at $\ft^0$. 
    Since $\ft^0$ and $\varepsilon$ can be taken arbitrarily small, this proves the small-time local exact $\fD$-controllability of~\eqref{EDO_ZN_C}.
\end{proof}

\subsection{A proof based on an extension argument}
\label{subsec:Complex_extension}

We give a proof relying on the strategy presented in \cref{subsec:extended_syst}.
Let $m\in\N^*$ and $f_0,\dotsc,f_m$ be holomorphic vector fields on $\C^d$ such that $f_0(0)=0$. 
We consider system \eqref{eq:syst-f0-f1-fm}.

\bigskip
\bigskip

\begin{proposition}
    \label{Prop:ext_C}
    Assume that:
    \begin{itemize}
        \item the flows associated with the $f_j$ are globally defined;
        \item the Lie algebra $\Lie \left\{ f_0, f_1, \dotsc, f_m \right\}$ is nilpotent of step $R \in \N^*$.
    \end{itemize}
    System \eqref{eq:syst-f0-f1-fm} is small-time locally approximately $\fD$-controllable if and only if the extended system
    \begin{equation}
        \label{eq:syst-f0-f1-fm-plus-f1f0}
        \dot{x} = f_0(x) + u_1 f_1(x) + \dotsb + u_m f_m(x) + u_{m+1} [f_1,f_0](x)
    \end{equation}
    is small-time locally approximately $\fD$-controllable.
\end{proposition}

\begin{proof}
    One implication is immediate: if \eqref{eq:syst-f0-f1-fm} is small-time locally approximately $\fD$-controllable, then so is the extended system \eqref{eq:syst-f0-f1-fm-plus-f1f0}, since the latter reduces to the former when $u_{m+1}=0$.

    We prove the converse implication. 
    It is enough to treat the case $m=1$ and show that an extra flow $e^{\beta [f_1,f_0]}$, $\beta\in\C$, can be approximated arbitrarily well by $\fD$-controlled trajectories of~\eqref{eq:syst-f0-f1-fm}.
    
    Let $\varepsilon>0$. 
    We define $u_{\varepsilon} := \frac{\beta}{\varepsilon} \delta_{t=0} - \frac{\beta}{\varepsilon} \delta_{t=\varepsilon}$.
    Then, the solution map of system \eqref{eq:syst-f0-f1-fm} satisfies
    \begin{equation}
        \label{x=eZ0}
        x(\varepsilon;f, u_{\varepsilon},x_0)=
        e^{- \frac{\beta}{\varepsilon} f_1} e^{\varepsilon f_0}
        e^{+ \frac{\beta}{\varepsilon} f_1} (x_0) =
        e^{H_0^\varepsilon} (x_0)
    \end{equation}
    where, by \cref{Lem:conjug_flots},
    \begin{equation}
        \label{Z0=}
        H_0^\varepsilon = \varepsilon f_0 + \beta [f_1,f_0] + \sum_{k=2}^{R-1}
        \frac{\beta^k}{k! \varepsilon^{k-1}} \ad_{f_1}^k(f_0).
    \end{equation}
    Fix coefficients $\alpha_1, \dotsc, \alpha_R \in \C$ such that
    \begin{equation}
        \label{hyp_alpha_k}
        \sum_{j=1}^R \alpha_j = 1
        \qquad \text{ and } \qquad
        \forall k \in \intset{2,R}, \quad
        \sum_{j=1}^R \alpha_j^k = 0.
    \end{equation}
    For example, using symmetric sums and Newton's identities, one can consider the roots of the polynomial $p(z) := \sum_{r=0}^R (-1)^r z^{R-r} / r!$.
    For a given control $u \in \fD$, we define
    \begin{equation}
        \Gamma(u) := \alpha_1 u \diamond \dotsb \diamond \alpha_R u.
    \end{equation}
    For $n, k \geq 0$, let $\mathcal{V}_{n,k}$ be the span of all iterated Lie brackets in $f_0$ and $f_1$ with exactly $k$ occurrences of $f_1$ and total length $>n+k$.
    We claim that, for every $n \geq 0$,
    \begin{equation}
        \label{eq:proof-extC-induction}
        x(R^n\varepsilon;f,\Gamma^n(u_\varepsilon),x_0) = e^{H_n^\varepsilon}(x_0)
        \quad \text{where} \quad
        H_n^\varepsilon
        = \beta [f_1,f_0]
        +\sum_{k=2}^R \Xi_{n,k}^\varepsilon
        +O(\varepsilon),
    \end{equation}
    where $\Xi_{n,k}^\varepsilon\in \mathcal{V}_{n,k}$ and $O(\varepsilon)$ is a finite linear combination of brackets whose coefficients are $O(\varepsilon)$.

    The case $n=0$ follows from \eqref{Z0=}.
    Assume now that \eqref{eq:proof-extC-induction} holds for some $n \geq 0$.
    By homogeneity, for every $\alpha\in\C$,
    \begin{equation}
        x(R^n\varepsilon;f,\alpha \Gamma^n(u_\varepsilon),x_0)
        =
        e^{H_n^\varepsilon(\alpha)}(x_0)
        \quad \text{where} \quad 
        H_n^\varepsilon(\alpha)
        =
        \alpha \beta [f_1,f_0]
        +\sum_{k=2}^R \alpha^k \Xi_{n,k}^\varepsilon
        +O(\varepsilon).
    \end{equation}
    By definition of $\Gamma$,
    \begin{equation}
        x(R^{n+1}\varepsilon;f,\Gamma^{n+1}(u_\varepsilon),x_0)
        =
        e^{H_n^\varepsilon(\alpha_R)}
        \cdots
        e^{H_n^\varepsilon(\alpha_1)}(x_0).
    \end{equation}
    Let $H_{n+1}^\varepsilon$ be the $\BCH_R$ logarithm given by \cref{prop:BCH}.
    By \eqref{eq:BCH} and \eqref{hyp_alpha_k}, its linear part is
    \begin{equation}
        \beta\Bigl(\sum_{j=1}^R \alpha_j\Bigr)[f_1,f_0]
        + \sum_{k=2}^R \Bigl(\sum_{j=1}^R \alpha_j^k\Bigr)\Xi_{n,k}^\varepsilon = \beta [f_1, f_0].
    \end{equation}
    Any other $\BCH_R$ term is a bracket involving at least two factors among the $H_n^\varepsilon(\alpha_j)$. 
    Since
    \begin{equation}
        [[f_1,f_0],\mathcal{V}_{n,k}] \subset \mathcal{V}_{n+1,k+1},
        \qquad
        [\mathcal{V}_{n,k},\mathcal{V}_{n,\ell}] \subset \mathcal{V}_{2n,k+\ell},
        \qquad 
        [\mathcal{V}_{n,k}, O(\varepsilon)] \subset \underset{k' \geq k}{\oplus} \mathcal{V}_{n+1,k'},
    \end{equation}
    all these terms belong to $\oplus_{k \geq 2} \mathcal{V}_{n+1,k} + O(\varepsilon)$, which proves the induction step.

    Taking $n=R-2$ and using the step $R$ nilpotency, we obtain $H_{R-2}^\varepsilon=\beta [f_1,f_0]+O(\varepsilon)$.
    Moreover, the total time is $R^{R-2}\varepsilon\to0$, which concludes the proof.
\end{proof}

Now, we can prove the controllability of \eqref{EDO_ZN_C}.

\begin{proof}[Proof of \cref{Prop:ZN_controllable_complexe}]
    Let $g_0, g_1$ be as in \cref{Prop:LARC_Z}.
    We want to prove the controllability of
    \begin{equation}
        \label{syst_C_A}
        \dot{\cZ}_N = g_0(\cZ_N) +
        \sum_{b \in \cE} u_b(t) g_b(\cZ_N)
    \end{equation}
    with $\cE=\{X_1\}$.
    By \cref{Prop:ext_C}, it suffices to prove the $\fD$-controllability of the system \eqref{syst_C_A} with $\cE=\{ X_1, [X_1,X_0] \}$. 
    Iterating the argument, it suffices to prove the $\fD$-controllability of the system \eqref{syst_C_A} with $\cE=\{M_0, M_1, \dots,M_{N-1}\}$. 
    By \cref{p:enrich-f1f2}, it suffices to prove that this system is $\fD$-controllable with $\cE$ the set of iterated brackets of length at most $N$ of the elements of $\{X_1, M_1, \dots,M_{N-1}\}$. 
    By \cref{def:oslash}, they span $\cL^\oslash_N(X)$. 
    Moreover, by \cref{Prop:LARC_Z}, for every $b \in \cL^\oslash_N(X)$, $g_b(0)=b$.
    Thus the Lie algebra rank condition is satisfied and, by \cref{Lem_gen}, this last extended system is $\fD$-controllable.
\end{proof}

\section{High order methods using commutator flows}
\label{sec:high_order_comm}

In this section, the base field of all vector spaces is $\K=\R$.
Our goal is to prove \cref{thm:Control_2N} and the existence part of \cref{thm:max-N}.
By \cref{prop:ZN=>both}, these are consequences of the following statement on the controllability of the Magnus system:

\begin{proposition}
    \label{Prop:Z2N-D-cont}
    Let $N \in \N^*$ and $\cE=\{X_1,W_1,\dots,W_{N-1}\}$, then the Magnus system
    \begin{equation}
        \label{eq:Z2N}
        \dot{\cZ}_{2N}(t) = [\cZ_{2N},X_0] + \sum_{b \in \cE} u_b(t) \Omega_b(\cZ_{2N}(t))
    \end{equation}
    set on $\cL^\oslash_{2N}(X)$, is $\fD$-controllable.
\end{proposition}

\subsection{An extension argument}

Our proof of \cref{Prop:Z2N-D-cont} relies on an extension argument, as introduced in \cref{subsec:extended_syst}. 
Let $f_0, f_1, \dotsc, f_m \in \cC^\infty(\R^d;\R^d)$ with $f_0(0) = 0$ and consider the control-affine system \eqref{eq:syst-f0-f1-fm}.

\begin{proposition}
    \label{p:enrich-f0f1}
    Assume that:
    \begin{itemize}
        \item the flows associated with the $f_j$ are globally defined;
        \item the Lie algebra $\Lie \left\{ f_0, f_1, \dotsc, f_m \right\}$ is nilpotent of step $R \in \N^*$;
        \item $\ad_{f_1}^2(f_0)=0$ or $\ad_{f_1}^2(f_0) \in \{f_1,\dots,f_m\}$.
    \end{itemize}
    Then system \eqref{eq:syst-f0-f1-fm} is small-time locally
    approximately $\fD$-controllable iff the extended system
    \begin{equation}
        \label{eq:syst-extended-f1f0}
        \dot{x}=f_0(x)+u_1 f_1(x)+\dots+u_m f_m(x)+u_{m+1} [f_1,f_0](x)
    \end{equation}
    is small-time locally approximately $\fD$-controllable.
\end{proposition}

\begin{proof}
    One implication is immediate: if \eqref{eq:syst-f0-f1-fm} is small-time locally approximately $\fD$-controllable, then so is the extended system \eqref{eq:syst-extended-f1f0}, since the latter reduces to the former when $u_{m+1}=0$.

    We prove the converse implication. 
    It is enough to show that an extra flow $e^{\alpha [f_1,f_0]}$, $\alpha\in\R$, can be approximated arbitrarily well by $\fD$-controlled trajectories of \eqref{eq:syst-f0-f1-fm}.
    Since $e^{\varepsilon f_0+\alpha [f_1,f_0]} = e^{\alpha [f_1,f_0]}+O(\varepsilon)$ as $\varepsilon \to 0$, it is enough to approximate $e^{\varepsilon f_0+\alpha [f_1,f_0]}$ for arbitrarily small $\varepsilon>0$.

    \medskip
    
    \noindent\emph{First case: $\ad_{f_1}^2(f_0)=0$.}
    By \cref{Lem:conjug_flots}, for any $\varepsilon > 0$,
    \begin{equation}
        e^{- \frac{\alpha}{\varepsilon}f_1}e^{\varepsilon f_0}e^{+ \frac{\alpha}{\varepsilon}f_1}
        =
        e^{\varepsilon f_0+\alpha [f_1,f_0]}.
    \end{equation}
    Hence $e^{\varepsilon f_0+\alpha [f_1,f_0]}$ is exactly realized by the original system with the ordered impulsive control $u = ((0, \alpha/\varepsilon, 1), (\varepsilon, -\alpha/\varepsilon, 1))$.

    \medskip
    \noindent\emph{Second case: $\ad_{f_1}^2(f_0)=f_j$ for some $j\in\{1,\dots,m\}$.}
    Set $g_k := \ad_{f_1}^k (f_0)$ for $k \geq 0$.
    Since $\Lie \left\{f_0,\dots,f_m\right\}$ is nilpotent of step $R$, we have $g_k=0$ for $k > R$.
    Moreover, by \cref{p:enrich-f1f2}, we may assume that $g_k \in \{ f_1, \dotsc, f_m \}$ for $2 \leq k \leq R$.

    We first note that for any $\beta_2,\dots,\beta_R\in\R$, the flow $e^{\varepsilon f_0+\beta_2 g_2+\cdots+\beta_R g_R}$ can be approximated arbitrarily well, in time $\varepsilon$, by $\fD$-controlled trajectories of \eqref{eq:syst-f0-f1-fm}. 
    Indeed, each $g_k$ with $k \geq 2$ is one of the control vector fields, so the Lie--Trotter formula \cref{lem:LT} gives
    \begin{equation}
        \left(
        e^{\frac{\varepsilon}{N} f_0}
        e^{\frac{\beta_2}{N} g_2}
        \dotsb
        e^{\frac{\beta_R}{N}g_R}
        \right)^N
        \longrightarrow
        e^{\varepsilon f_0+\beta_2 g_2+\cdots+\beta_R g_R}
        \quad 
        \text{as } N \to \infty.
    \end{equation}
    It therefore remains to choose $\beta_2,\dots,\beta_R$ so that conjugation by $e^{\frac \alpha \varepsilon f_1}$ produces exactly $e^{\varepsilon f_0+\alpha [f_1,f_0]}$.
    For $k=2,\dots,R$, choose $\beta_k$ recursively so that
    \begin{equation}
        \label{eq:beta-recursive}
        \frac{\alpha^k}{k! \varepsilon^{k-1}}+\sum_{l=0}^{k-2}\frac{\alpha^l}{l!\varepsilon^l}\beta_{k-l}
        =0.
    \end{equation}
    This is a triangular linear system, hence it has a unique solution.
    By \cref{Lem:conjug_flots}, 
    \begin{equation}
        e^{- \frac{\alpha}{\varepsilon}f_1}
        e^{\varepsilon f_0+\sum_{r=2}^R \beta_r g_r}
        e^{+ \frac{\alpha}{\varepsilon}f_1} = e^F,
        \quad \text{where} \quad
        F := \sum_{l=0}^R \frac{\alpha^l}{l! \varepsilon^l} \ad_{f_1}^l\left(\varepsilon f_0+\sum_{r=2}^R\beta_r g_r\right).
    \end{equation}
    Since $\ad_{f_1}^l(f_0)=g_l$ and $\ad_{f_1}^l(g_r)=g_{r+l}$, this becomes
    \begin{equation}
        F = \varepsilon f_0+\alpha [f_1,f_0]
        +
        \sum_{k=2}^R
        \left(
        \frac{\alpha^k}{k!\,\varepsilon^{k-1}}
        +\sum_{l=0}^{k-2}\frac{\alpha^l}{l!\,\varepsilon^l}\beta_{k-l}
        \right) g_k.
    \end{equation}
    By \eqref{eq:beta-recursive}, all coefficients of $g_k$ for $k \geq 2$ vanish, so $F = \varepsilon f_0 + \alpha [f_1, f_0]$. 
    Thus $e^{\varepsilon f_0+\alpha [f_1,f_0]}$ can be approximated arbitrarily well by $\fD$-controlled trajectories of \eqref{eq:syst-f0-f1-fm}.
\end{proof}

\subsection{Proof of the controllability of the Magnus system}

\begin{proof}[Proof of \cref{Prop:Z2N-D-cont}]
    Let $g_0, g_1$ be the vector fields on $\cL^\oslash_{2N}(X)$ defined by $g_0(Z)=[Z,X_0]$ and $g_1(Z)=\Omega_{X_1}(Z)$. 
    Then $g_b = \Omega_b$ for every $b \in \Br(X) \setminus \{X_0\}$ (see \cref{Prop:LARC_Z}) thus $g_b=0$ for every $b \in \Br(X)$ with $|b|>2N$ (see \cref{def:Omega_b}).
    In particular, $\Lie \left\{ g_0, g_1 \right\}$ is nilpotent of step $2N$.

    We define $\cE_0 := \{ X_1, W_1, \dotsc, W_{N-1} \}$. 
    If $N=1$ then $g_{W_1}=0$. 
    If $N \geq 2$, $\cE_0$ contains $W_1$. 
    Thus, by \cref{p:enrich-f0f1}, it suffices to prove that this system is $\fD$-controllable with $\cE = \cE_0 \cup \{ (X_1, X_0) \}$.
    Iterating the argument, it suffices to prove that this system is $\fD$-controllable with $\cE = \cE_0 \cup \{ M_1, \dotsc, M_{2N-1} \}$.
    By \cref{p:enrich-f1f2}, it suffices to prove that this system is $\fD$-controllable with $\cE$ the set of iterated brackets of length at most $2N$ of elements of $\cE_0 \cup \{ M_1, \dotsc, M_{2N-1} \}$. 
    By \cref{def:oslash}, they span $\cL^\oslash_{2N}(X)$. 
    Moreover, by \cref{Prop:LARC_Z}, for every $b \in \cL^\oslash_{2N}(X)$, $g_b(0)=b$.
    Thus the Lie algebra rank condition is satisfied, and, by \cref{Lem_gen}, this last extended system is $\fD$-controllable.
\end{proof}

\section{Order restrictions for signed real-valued methods}
\label{sec:max-order}

In this section the base field is $\K=\R$. 
We prove the upper bound of \cref{thm:max-N}, using a representation of the solutions $S_N(t,X,u)$ to \eqref{EDO_SN} as a product over an ordered basis, involving the \emph{coordinates of the second kind} (see \cite{OwrenMarthinsen2001} for an application to integration methods). 

\subsection{Representation of the state as a product}
\label{sec:wei-norman}

Let $N \in \N$ and $\cB$ be an ordered basis of $\cL_N(X)$.
For $b \in \cB$, let $b^*$ be the dual basis element, $V_b^+ := \vect \{ c \in \cB \mid c > b \}$ and $\mathbb{P}_b^+$ the projection on $V_b^+$ (parallel to $\vect \{ c \in \cB \mid c \leq b \}$).

\begin{definition}
    \label{def:triangular}
    We say that $\cB$ is \emph{triangular} when, for all $b < c \in \cB$, one has $[b,c] \in V_b^+$.
\end{definition}

\begin{proposition}
    \label{prop:chevalley}
    If the basis $\cB = \{ b_1 < \dotsb < b_\ell \}$ is triangular, then the following map is a global analytic diffeomorphism of $\R^\ell$ onto the submanifold $G_N(X) := \exp (\cL_N(X))$ of $\cA_N(X)$:
    \begin{equation}
        \label{eq:glob-diff}
        (\xi_b)_{b \in \cB} \mapsto \exp(\xi_{b_\ell} b_\ell) \dotsb \exp(\xi_{b_1} b_1).
    \end{equation}
    The $\xi_b$ are called \emph{coordinates of the second kind}.
\end{proposition}

\begin{proof}
    It is a particular case of \cite{Chevalley1941} or \cite[Theorem 3.18.11]{Varadarajan1984} for solvable Lie groups.
    
    Since $\cL_N(X)$ is nilpotent, $\exp$ and $\log$ are globally-defined polynomial inverses between $\cL_N(X)$ and $G_N(X)$, and the $\BCH_N$ polynomial of \cref{prop:BCH} is finite. 

    \emph{Surjectivity}.
    Let $S = \exp Z_1 \in G_N(X)$.
    Define $\xi_{b_k}$ and $Z_{k+1}$ for $k = 1, \dotsc, \ell$, by the recursion
    \begin{equation}
        \xi_{b_k} := \langle b_k^*, Z_k \rangle
        \quad \text{and} \quad 
        Z_{k+1} := \BCH(Z_k, -\xi_{b_k} b_k).
    \end{equation}
    Since $\cB$ is triangular, $Z_{k+1} \in V_{b_k}^+$.
    In particular, $Z_{\ell+1} = 0$ and $S = \exp (\xi_{b_\ell} b_\ell) \dotsb \exp (\xi_{b_1} b_1)$.

    \emph{Injectivity}.
    Given $S = \exp(\xi_{b_\ell} b_\ell) \dotsb \exp(\xi_{b_1} b_1) = \exp(\xi_{b_\ell}' b_\ell) \dotsb \exp(\xi_{b_1}' b_1)$, the same recursion iteratively proves that $\xi_{b_k} = \xi_{b_k}'$ for $k = 1, \dotsc, \ell$.

    \emph{Regularity}.
    Since $\cL_N(X)$ is nilpotent, the forward map is polynomial.
    The inverse map is obtained by finitely many compositions of $\log$ and $\BCH_N$, hence is polynomial too.
\end{proof}

Thus, it makes sense to represent the state $S_N(t,X,u)$ as a product of the form \eqref{eq:glob-diff}, where the coefficients $\xi_b$ will depend on $t$ and $u$.
As a counterpart to Magnus' expansion of \cite{Magnus54ote} representing the solution as a single exponential, Wei and Norman proved in \cite{WeiNorman1964} the following result.

\begin{proposition}
    Let $N \in \N$ and $\cB = \{ b_1 < \dotsb < b_\ell \}$ be an ordered basis of $\cL_N(X)$.

    For $\cE \subset \Br(X) \setminus \{ X_0 \}$ and $u = (u_b)_{b \in \cE} \in L^1(\R^+;\R^{|\cE|})$, the solution to \eqref{EDO_SN} is given by
    \begin{equation}
        \label{eq:SN=ProdExp}
        S_N(t,X,u) = \exp(\xi_{b_\ell}(t,u) b_\ell) \dotsb \exp(\xi_{b_1}(t,u) b_1)
    \end{equation}
    if and only if the coordinates $(\xi_b)_{b \in \cB}$ satisfy $\xi_b(0,u) = 0$ and the following relation in $\cL_N(X)$:
    \begin{equation}
        \label{eq:WeiNorman}
        \sum_{k = 1}^\ell \dot{\xi}_{b_k} \left( \prod_{j = 1}^{k-1} \exp(-\xi_{b_j} \ad_{b_j}) \right) b_k = X_0 + \sum_{b \in \cE} u_b b.
    \end{equation}
\end{proposition}

\begin{proof}
    Let $P(t)$ denote the right-hand side of \eqref{eq:SN=ProdExp}.
    First, $P(0) = 1$ if and only if $\xi_b(0,u) = 0$ for all $b \in \cB$.
    Moreover, explicitly differentiating in time, we obtain
    \begin{equation}
        \dot{P} = P \sum_{k=1}^\ell \left( \prod_{j=1}^{k-1} \exp (-\xi_{b_j} b_j) \right) \dot{\xi}_{b_k} b_k \left( \prod_{j=k-1}^{1} \exp (\xi_{b_j} b_j) \right).
    \end{equation}
    By the conjugation \cref{lem:conjugation}, $P$ and $S_N$ satisfy the same ODE if and only if \eqref{eq:WeiNorman} holds.
\end{proof}

The relations \eqref{eq:WeiNorman} are of the form $M(\xi) \dot{\xi} = U$ where $U = X_0 + \sum u_b b$ is given.
Since $M$ depends polynomially on $\xi$ and $M(0) = \Id$, one can always solve \eqref{eq:WeiNorman} locally in time \cite[Theorem~1]{WeiNorman1964}.
Under a \emph{stronger} triangularity assumption of the form $[b,c] \in V_c^+$ for all $b < c \in \cB$, $M$ is triangular, and system \eqref{eq:WeiNorman} can be directly inverted globally \cite[Theorem 2]{WeiNorman1964}.

In fact, our triangularity assumption \cref{def:triangular} is sufficient to solve \eqref{eq:WeiNorman} globally.

\begin{proposition}
    Let $N \in \N$ and $\cB = \{ b_1 < \dotsb < b_\ell \}$ be a triangular basis of $\cL_N(X)$.

    For $\cE \subset \Br(X) \setminus \{ X_0 \}$ and $u = (u_b)_{b \in \cE} \in L^1(\R^+;\R^{|\cE|})$, the solution to \eqref{eq:WeiNorman} is explicitly given by
    \begin{equation}
        \label{eq:xib-multi-bis}
        \xi_b(t,u) = \sum
        \left\langle b^*, \ad_{a_r}^{m_r} \mathbb{P}_{a_r}^+ \dotsb \ad_{a_1}^{m_1} \mathbb{P}_{a_1}^+ (c) \right\rangle
        \int_0^t \frac{\xi_{a_r}^{m_r}(s,u) \dotsb \xi_{a_1}^{m_1}(s,u)}{m_r! \dotsb m_1!} u_c(s) \dd s,
    \end{equation}
    summing over $r \in \N$, $a_1 < \dotsb < a_r < b \in \cB$, $m_1, \dotsc, m_r \in \N^*$ and $c \in \cE \cup \{X_0\}$, with $u_{X_0} \equiv 1$.
\end{proposition}

\begin{proof}
    Set $U_1(t) := X_0 + \sum_{b \in \cE} u_b(t) b$.
    For $k = 1, \dotsc, \ell$, we define recursively $\xi_{b_k}$ and $U_{k+1}$ by
    \begin{equation}
        \label{eq:rec-xi-U}
        \dot{\xi}_{b_k} := \langle b_k^*, U_k \rangle
        \quad \text{and} \quad 
        U_{k+1} := \exp(\xi_{b_k} \ad_{b_k}) \mathbb{P}_{b_k}^+ U_k.
    \end{equation}
    Since $\cB$ is triangular, $U_{k+1} \in V_{b_k}^+$.
    By induction, \eqref{eq:WeiNorman} is equivalent to the reduced system
    \begin{equation}
        \label{eq:WeiNorman_Uk}
        \sum_{i=k}^\ell \dot{\xi}_{b_i} \left( \prod_{j=k}^{i-1} \exp(-\xi_{b_j} \ad_{b_j}) \right) b_i = U_k.
    \end{equation}
    For $k=1$ this is exactly \eqref{eq:WeiNorman}. 
    For $k > 1$, the $i=k$ term is $\dot\xi_{b_k}b_k$, while the $i>k$ terms are in $V_{b_k}^+$. 
    Thus, pairing \eqref{eq:WeiNorman_Uk} with $b_k^*$ yields $\dot\xi_{b_k}=\langle b_k^*,U_k\rangle$, and  applying $\exp(\xi_{b_k}\ad_{b_k})\mathbb P_{b_k}^+$ to
    \eqref{eq:WeiNorman_Uk} removes the $i=k$ term and gives \eqref{eq:WeiNorman_Uk} with $k$ replaced by $k+1$ and $U_{k+1}$ as in \eqref{eq:rec-xi-U}.
    
    Finally, unfolding the recurrence relation for $U_k$ gives
    \begin{equation}
        U_k = \exp(\xi_{b_{k-1}} \ad_{b_{k-1}}) \mathbb{P}^+_{b_{k-1}} \dotsb \exp(\xi_{b_1} \ad_{b_1}) \mathbb{P}^+_{b_1} U_1.
    \end{equation}
    Expanding each exponential and using $\xi_{b_k}(0) = 0$ yields the explicit global formula \eqref{eq:xib-multi-bis}.
\end{proof}

\subsection{Generalization to impulsive controls and splitting methods}

In \cref{sec:wei-norman}, we worked with $L^1$ controls.
We now extend our definitions to $\fD$ controls, and use them to derive a characterization of a splitting method of order $N$.

\begin{definition}
    \label{def:xib-impulsive}
    Let $N \in \N^*$ and $\cB$ a triangular basis of $\cL_N(X)$.
    Let $\cE \subset \Br(X) \setminus \{ X_0 \}$.
    Given $u \in \fD$ and $t > 0$, define 
    \begin{equation}
        \big( \xi_b(t,u) \big)_{b \in \cB} := \Theta^{-1}(S_N(t,X,u)),
    \end{equation}
    where $\Theta$ is the diffeomorphism of \cref{prop:chevalley}.
    For $u \in \fD$, $\xi_b(\cdot,u)$ is càdlàg.
\end{definition}

\begin{proposition}
    \label{prop:split-xi}
    Let $N \in \N^*$ and $\cB$ a triangular basis of $\cL_N(X)$.
    Let $\cE \subset \Br(X) \setminus \{ X_0 \}$ containing $X_1$.
    The following statements are equivalent:
    \begin{enumerate}
        \item \label{it:prop-split-xi-1} 
        there exists an $(\R^+,\R)$ splitting method of order $N$ involving $X_0$ and the elements of $\cE$,
        \item \label{it:prop-split-xi-2} 
        there exists $u=(u_b)_{b \in \cE} \in \fD$ such that, for all $b \in \cB$, $\xi_b(1,u)=\xi_b(1,\underline{u})$, where $\underline{u}=(\underline{u}_b)_{b \in \cE}$, $\underline{u}_{X_1} = 1$ and $\underline{u}_b=0$ for $b \in \cE \setminus \{X_1\}$.
    \end{enumerate}
\end{proposition}

\begin{proof}
    By \cref{prop:splitting-iff}, \cref{it:prop-split-xi-1} is equivalent to the existence of $u \in \fD$ such that $S_N(1,X,u) = \exp (X_0 + X_1) = S_N(1,X,\underline{u})$, which is equivalent to \cref{it:prop-split-xi-2} by the injectivity of \cref{prop:chevalley}.
\end{proof}

\subsection{Link with Hall sets and Sussmann's product}

Unlike $\cA(X)$ -- whose monomials form a basis -- there is no canonical basis of $\cL(X)$. 
Hall sets (stemming from \cite{Hall1950} and generalized in \cite{Shirshov1962,Viennot1978}) constitute a wide family of bases of $\cL(X)$, which includes many well-known bases of $\cL(X)$ such as the historical length-compatible Hall bases of \cite{Hall1950} or the Chen--Fox--Lyndon basis of \cite{ChenFoxLyndon1958}. 
We refer the interested reader to \cite[Section 1.4]{A1} for a more gentle introduction and more thorough details (see also \cite[Chapter 4]{Reutenauer1993}).

\begin{definition}
    \label{def:Hall}
    A \emph{Hall set} is a subset $\cB$ of $\Br(X)$, containing $X$, with a total order $<$ such that
    \begin{itemize}
        \item for all $a, b \in \Br(X)$, $(a, b) \in \cB$ iff $a < b \in \cB$ and, either $b \in X$ or $b=(b', b'')$ with $b' \leq a$,
        \item for all $a, b \in \cB$ such that $(a,b) \in \cB$, one has $a < (a,b)$.
    \end{itemize}
\end{definition}

\begin{proposition}
    \label{prop:Hall-triangular}
    If $\cB$ is a Hall set, then $\eval(\cB)$ is a triangular basis of $\cL(X)$.
\end{proposition}

\begin{proof}
    By \cite[Theorem 1.2]{Viennot1978}, $\eval(\cB)$ is a basis of $\cL(X)$.
    The triangularity comes from the axioms of a Hall set and the associated decomposition algorithm of $[b,c]$ on $\eval(\cB)$ (see \cite[Theorem 2.1]{A1})
\end{proof}

Given a Hall set $\cB$ and $N \in \N$, one can consider $\cB_N := \{ b \in \cB ; |b| \leq N \}$.
Then, up to identifying brackets with their evaluation, $\cB_N$ is a (triangular) basis of $\cL_N(X)$.

\medskip

Using an iterative proof based on Lazard's elimination, as in \cref{sec:wei-norman}, Sussmann proved in~\cite{Sussmann1986} the following representation result (see \cite[Theorem 64]{P1} for the proof for non-length-compatible Hall sets) for the solution to \eqref{EDO_SN} in the case $\cE = \{ X_1 \}$.
In this particular case, the inductive formula \eqref{eq:xib-sussmann} for the coordinates of the second kind has a very nice structure which has proven fruitful for control theory (see \cite[Section 1.6]{P2}).

\begin{proposition}
    Let $N \in \N$ and $\cB$ a Hall set.
    The solution to $\dot{S}_N = S_N (X_0 + u X_1)$ for a given $u \in L^1(\R^+;\R)$ satisfies
    \begin{equation}
        S_N(t,X,u) = \underset{b \in \cB_N}{\overleftarrow{\prod}} \exp \left( \xi_b(t,u) b \right)
    \end{equation}
    where for any $b \in \cB_N$, decomposing $b = \ad_{a_r}^{m_r} \dotsb \ad_{a_1}^{m_1} (X_j)$, one has (with $u_{X_0} \equiv 1$),
    \begin{equation}
        \label{eq:xib-sussmann}
        \xi_{b}(t,u) = \int_0^t \frac{\xi_{a_r}^{m_r}(s,u) \dotsb \xi_{a_1}^{m_1}(s,u)}{m_r! \dotsb m_1!} u_{X_j}(s) \dd s.
    \end{equation}
\end{proposition}

\begin{proof}
    By \cref{prop:Hall-triangular}, $\cB_N$ is a triangular basis of $\cL_N(X)$ so \cref{sec:wei-norman} applies and it suffices to check that the general formula \eqref{eq:xib-multi-bis} reduces to \eqref{eq:xib-sussmann} for the case of a Hall basis and $\cE = \{ X_1 \}$.

    On the one hand, by \cref{def:Hall}, for any $b \in \cB$, there exists a unique maximal factorization $b = \ad_{a_r}^{m_r} \dotsb \ad_{a_1}^{m_1}(X_j)$ where $r \in \N$, $a_1 < \dotsb < a_r \in \cB$, $m_1, \dotsc, m_r \in \N^*$ and $j \in \{0,1\}$.

    On the other hand, by \cref{def:Hall} given any $r \in \N$, $a_1 < \dotsb < a_r \in \cB$, $m_1, \dotsc, m_r \in \N^*$ and $j \in \{0,1\}$, the term
    \begin{equation}
        \ad_{a_r}^{m_r} \mathbb{P}_{a_r}^+ \dotsb \ad_{a_1}^{m_1} \mathbb{P}_{a_1}^+ (X_j)
    \end{equation}
    either vanishes in $\cL(X)$ or belongs to $\cB$.

    Thus \eqref{eq:xib-sussmann} is the only non-vanishing term in \eqref{eq:xib-multi-bis}.
\end{proof}

\subsection{Choice of a good basis and computation of the coordinates}

Designing a splitting method for the flow of $f_0+f_1$ that uses only forward flows of $f_0$ (but both forward and backward flows of $f_1$) is an asymmetric question. 
This asymmetry appears in the associated control-affine system \eqref{eq:x-f0f1} and must be reflected when selecting a basis of $\cL(X)$.
In \cite[Section 3]{P2}, we constructed a Hall set $\cB^\star$ tailored to control applications, capturing this asymmetry.
We will not need its full definition here, only the following fact.

\begin{proposition}
    \label{p:Bstar}
    There exists a Hall set $\cB^\star$ containing the elements of \eqref{def:Mnu-Wj}, as well as the $W_{j,\nu} := ((\dotsb(W_j, X_0), \dotsc), X_0)$ with $X_0$ appearing $\nu \geq 1$ times, ordered as follows
    \begin{equation}
        \label{eq:order-Bstar}
        X_1 = M_0 < \cdots < M_{\nu} < M_{\nu+1} < \dots < W_j < W_{j,\nu} < W_{j+1} < \cdots < X_0.
    \end{equation}
    The evaluations of these elements in $\cL(X)$ form a basis of $\vect \{ \eval(b) \mid b \in \Br(X), n_1(b) \leq 2 \}$.
\end{proposition}

\begin{proposition}
    For $\cE \subset \cB^\star \setminus \{X_0\}$ and $u = (u_b)_{b \in \cE} \in L^1$, the associated coordinates of the second kind satisfy
    \begin{equation}
        \label{eq:xi-Bstar}
        \begin{split}
            & \dot{\xi}_{X_1} = u_{X_1}, \quad 
            \dot{\xi}_{M_\nu} = \xi_{M_{\nu-1}} + u_{M_\nu}, \quad 
            \dot{\xi}_{X_0} = 1, \\
            & \dot{\xi}_{W_j} = \frac 1 2 \xi_{M_{j-1}}^2 + u_{W_j} + \sum_{i=0}^{j-1} (-1)^{j-1-i} \xi_{M_i} u_{M_{2j-1-i}},
        \end{split}
    \end{equation}
    for $\nu, j \geq 1$ with the convention that $u_b = 0$ when $b \notin \cE$.
\end{proposition}

\begin{proof}
    We use the homogeneity properties of Hall bases: for $b \in \cB^\star$ and $a \in \Br(X)$, if $\langle b^*, a \rangle \neq 0$, then $|a|=|b|$, $n_0(a)=n_0(b)$ and $n_1(a) = n_1(b)$.
    Differentiating \eqref{eq:xib-multi-bis}, for any $b \in \cB^\star$,
    \begin{equation}
        \label{eq:xib-local}
        \dot{\xi_b} = \sum
        \left\langle b^*, \ad_{a_r}^{m_r} \mathbb{P}_{a_r}^+ \dotsb \ad_{a_1}^{m_1} \mathbb{P}_{a_1}^+ (c) \right\rangle
        \frac{\xi_{a_r}^{m_r} \dotsb \xi_{a_1}^{m_1}}{m_r! \dotsb m_1!} u_c,
    \end{equation}
    summing over $r \in \N$, $a_1 < \dotsb < a_r < b \in \cB^\star$, $m_1, \dotsc, m_r \in \N^*$ and $c \in \cE \cup \{X_0\}$, with $u_{X_0} \equiv 1$.

    \begin{itemize}
        \item For $b \in \{X_0,X_1\}$, since $|b|=1$, the only non-zero term in \eqref{eq:xib-local} is $r=0$ and $c=b$.

        \item For $b=M_\nu$ with $\nu \geq 1$, we have $n_1(M_\nu)=1$. 
        Hence either ($r=0$ and $c=M_\nu$) or ($r=1$, $m_1=1$, $a_1=M_{\nu-1}$ and $c=X_0$). 
        Indeed, terms with $r \geq 2$ or with $r=1$ and $m_1 \geq 2$ contain at least two occurrences of $X_1$.

        \item For $b=W_j$ with $j \geq 1$, we have $n_1(W_j)=2$ and $|W_j|=2j+1$.
        \begin{itemize}
            \item If $r=0$, then $c=W_j$, which gives the term $u_{W_j}$.

            \item If $r=1$, we distinguish cases.
            \begin{itemize}
                \item If $m_1 = 1$, $c = X_0$ yields no contribution, so $c = M_l$ and $a_1 = M_i$ with $M_i < M_l$.
                Repeated Jacobi yields $\langle b^*, [M_i, M_l] \rangle = (-1)^{i-j+1} \mathbf{1}_{i+l=2j-1}$, which gives the sum in \eqref{eq:xi-Bstar}.
                \item If $m_1 = 2$, then $a_1 = M_{j-1}$, which gives the term $\frac 12 \xi_{M_{j-1}}^2$.
            \end{itemize}
            
            \item If $r = 2$, then we are considering a term of the form $[M_i, [M_l, X_0]]$ with $M_i > M_l$ and $M_i < (M_l, X_0)$, impossible by \eqref{eq:order-Bstar}.\qedhere
        \end{itemize}
    \end{itemize}
\end{proof}

\subsection{A positivity argument}

By \cref{prop:splitting-iff}, a splitting method is a $u \in \fD$ such that $S_N(1,X,u) = \exp(X_0+X_1)$.
An explicit reference control $\underline{u} \in L^1$ such that $S_N(1,X,\underline{u}) = \exp(X_0+X_1)$ is $\underline{u}_{X_1} = 1$ and $\underline{u}_b = 0$ for $b \neq X_1$.
The following positivity argument will be used multiple times in the sequel.

\begin{proposition}
    \label{prop:xi>WN}
    Let $N \in \N^*$ and $\cE \subset \cB^\star \setminus \{X_0\}$ such that $M_N, \dotsc, M_{2N} \notin \cE$ and $W_N \notin \cE$.
    Let $u \in \fD$ such that $\xi_b(1,u)=\xi_b(1,\underline{u})$ for all $b \in \{M_{N},\dots,M_{2N}\}$. 
    Then
    \begin{equation}
        \label{eq:WN>WN}
        \xi_{W_N}(1,u) > \xi_{W_N}(1,\underline{u}).
    \end{equation}
\end{proposition}

\begin{proof}
    Let $u^\varepsilon \in L^1((0,1);\R^{|\cE|})$ be an order-preserving regularization of $u$ given by \cref{prop:impulsive-reg}. 
    For all $b \in \cB^\star_N$ and $t \in (0,1]$, by \cref{prop:chevalley} and \cref{def:xib-impulsive}, one has $\xi_b(t,u^\varepsilon)\to \xi_b(t,u)$.
    Set
    \begin{equation}
        v^\varepsilon(t):=\xi_{M_{N-1}}(t,u^\varepsilon),
        \qquad
        v(t):=\xi_{M_{N-1}}(t,u),
        \qquad
        \underline{v}(t):=\xi_{M_{N-1}}(t,\underline{u})=\frac{t^N}{N!}.
    \end{equation}
    Since $M_N,\dots,M_{2N}\notin \cE$, for every $\varepsilon>0$ and every $\nu\in\intset{0,N}$, the equations \eqref{eq:xi-Bstar} give
    \begin{equation}
        \delta_\nu^\varepsilon
        :=
        \xi_{M_{N+\nu}}(1,u^\varepsilon)-\xi_{M_{N+\nu}}(1,\underline{u})
        =
        \int_0^1 \frac{(1-t)^\nu}{\nu!}\bigl(v^\varepsilon(t)-\underline{v}(t)\bigr)\dd t.
    \end{equation}
    For each $\nu \in \intset{0,N}$, since $\xi_{M_{N+\nu}}(1,u^\varepsilon)\to \xi_{M_{N+\nu}}(1,u) = \xi_{M_{N+\nu}}(1,\underline{u})$, we have $\delta_\nu^\varepsilon \to 0$.
    Since $\underline{v}(t)=t^N/N!$ belongs to the span of the polynomials $(1-t)^\nu$, we have
    \begin{equation}
        \int_0^1 \underline{v}(t)\bigl(v^\varepsilon(t)-\underline{v}(t)\bigr) \dd t = o(1).
    \end{equation}
    Since $M_N,\dots,M_{2N-1},W_N\notin\cE$, the formula \eqref{eq:xi-Bstar} yields
    \begin{equation}
        \xi_{W_N}(1,u^\varepsilon)-\xi_{W_N}(1,\underline{u})
        =
        \frac12\|v^\varepsilon-\underline{v}\|_{L^2(0,1)}^2
        +\int_0^1 \underline{v}(t)\bigl(v^\varepsilon(t)-\underline{v}(t)\bigr)\dd t.
    \end{equation}
    Since $\xi_{W_N}(1,u^\varepsilon)\to \xi_{W_N}(1,u)$ and
    $v^\varepsilon(t)\to v(t)$ for every $t\in(0,1]$, Fatou's lemma gives
    \begin{equation}
        \xi_{W_N}(1,u)-\xi_{W_N}(1,\underline{u})
        \geq \frac12 \| v - \underline{v} \|_{L^2}^2.
    \end{equation}
    It remains to prove that $v\neq \underline{v}$ in $L^2(0,1)$.
    Since $u \in \fD$, let $I \subset (0,1)$ be a nonempty interval containing no impulse time.
    For $\varepsilon$ small enough, $u^\varepsilon \equiv 0$ on $I$, so $(v^{\varepsilon})^{(N)} \equiv 0$, so $v^{(N)} \equiv 0$.
    On the contrary, $\underline{v}^{(N)} \equiv 1$.
    Hence $v \neq \underline{v}$ and $\xi_{W_N}(1,u) > \xi_{W_N}(1,\underline{u})$.
\end{proof}

\subsection{Proof of the upper bound of \cref{thm:max-N}}

\begin{proof}
    Let $N \in \N^*$. 
    Assume that there exists an $(\R^+,\R)$ splitting method of order $(2N+1)$ involving $X_0$ and $X_1, W_1, \dotsc, W_{N-1}$.
    By \cref{prop:split-xi}, there exists $u \in \fD$ such that, for all $b \in \cB^\star$ with $|b| \leq 2N+1$, one has $\xi_b(1,u) = \xi_b(1,\underline{u})$.
    Since $| M_\nu | = \nu + 1$, this holds in particular for any $b \in \{ M_N, \dotsc, M_{2N} \}$.
    Thus, by \cref{prop:xi>WN}, $\xi_{W_N}(1,u) > \xi_{W_N}(1,\underline{u})$.
    This gives a contradiction since $|W_N| = 2N+1$.
\end{proof}

\section{High order methods relying on degeneracies}
\label{sec:high_order_deg}

In this section the base field of all the vector spaces is $\K=\R$.
Our goal is to prove \cref{Thm:Deg}.
A key ingredient is the error estimate associated with Sussmann's product, presented in \cref{subsec:Suss_err}. 
In this section, the estimates on products of flows have the same sense as in \cref{def:SM}.

\subsection{Error estimate for Sussmann's product}
\label{subsec:Suss_err}

We consider the system \eqref{eq:x-f0f1} with scalar control $u$.

\begin{proposition}
    \label{Prop:Suss_Err}
    Let $\cB$ be a Hall set, $(\xi_b)_{b \in \cB}$ be the coordinates of the second kind associated with $\cB$ (see \eqref{eq:xib-sussmann}), $N \in \N^*$ and $f_0, f_1 \in \cC^\infty(\R^d;\R^d)$.
    For every $x_0 \in \R^d$ and $u \in L^1$ or $\fD$,
    \begin{equation}
        x(1;Tf,u,x_0)= \underset{b \in \cB_N}{\overrightarrow{\prod}} e^{\xi_b(1,u) T^{|b|} f_b} x_0 + \underset{T \to 0}{O}(T^{N+1}).
    \end{equation}
\end{proposition}

\begin{proof}
    We deduce from \cite[Proposition 104]{P1} that
    for every $x_0 \in \R^d$, there exists $C=C(f, x_0, N)$ such that, for every $u \in L^1((0,1);\R)$ and $T>0$ small enough,
    \begin{equation}
        \label{err_L1}
        \left| x(1;Tf,u,x_0) -
        \underset{b \in \cB_N}{\overrightarrow{\prod}}
        e^{ \xi_b(1,u) T^{|b|} f_b } x_0 \right|
        \leq C \|u\|_{L^1}^{N+1} T^{N+1}.
    \end{equation}
    For $u \in \fD$, we consider the regularization $u^{\varepsilon} \in L^1$ of \cref{prop:impulsive-reg}. 
    They satisfy $\|u^{\varepsilon}\|_{L^1}=O(1)$ and $\xi_b(1,u^{\varepsilon}) \to \xi_b(1,u)$ as $\varepsilon \to 0$ by \cref{prop:chevalley} and \cref{def:xib-impulsive}.
    We conclude by passing to the limit $\varepsilon \to 0$ in the estimate \eqref{err_L1} for $u^{\varepsilon}$.
\end{proof}

\subsection{Proof of the \cref{Thm:Deg}}

We first formulate a technical lemma.

\begin{lemma}
    \label{Lem:Tj}
    Let $N \in \N$ and $f_1, g_1, \dotsc, f_N, g_N$ be smooth vector fields on $\R^d$ such that
    \begin{equation}
        \label{hyp_LemTj}
        \underset{j \in \intset{1, N}}{\overrightarrow{\prod}}
        e^{ T^j f_j }
        =
        \underset{j \in \intset{1, N}}{\overrightarrow{\prod}}
        e^{ T^j g_j}
        + \underset{T \to 0}{O}(T^{N+1}).
    \end{equation}
    Then, $f_j=g_j$ for every $j \in \intset{1, N}$.
\end{lemma}

\begin{proof}
    We prove by induction on $k \in \intset{1, N}$ that $f_k = g_k$. 
    Let $k \in \intset{1, N}$. 
    Assume that $f_j = g_j$ for $j<k$. 
    Then \eqref{hyp_LemTj} implies
    \begin{equation}
        \label{hyp_LemTj_rec}
        \underset{j \in \intset{k, N}}{\overrightarrow{\prod}}
        e^{ T^j f_j }
        =
        \underset{j \in \intset{k, N}}{\overrightarrow{\prod}}
        e^{ T^j g_j }
        + \underset{T \to 0}{O}(T^{N+1}).
    \end{equation}
    By considering the Taylor expansion of \eqref{hyp_LemTj_rec}, we obtain, $e^{T^k f_k} =e^{T^k g_k} + O(T^{k+1})$.
    Moreover $(e^{T^k f_k}-e^{T^k g_k})=T^k(f_k-g_k)+O(T^{k+1})$ thus $f_k=g_k$.
\end{proof}

\begin{proof}[Proof of \cref{Thm:Deg}]
    Let $N \in \N^*$ and $f_0, f_1$ be smooth vector fields on $\R^d$ such that $f_{W_j}=0$ for any $j \in \intset{1, N-1}$. 
    We assume that there exists a splitting method of order $(2N+1)$ relative to $(f_0, f_1)$. 
    Let $u \in \fD$ be the associated control and $\underline{u}:[0,1] \to \R$ identically equal to $1$. 
    Then,
    \begin{equation}
        \label{eq:proof-deg-1}
        \forall x_0 \in \R^d, \qquad
        x(1;Tf,u,x_0)=x(1;Tf,\underline{u},x_0)+O(T^{2N+2}).
    \end{equation}
    Consider the Hall set $\cB^\star$, and $\cB^\star_{2N+1}$ its elements of length at most $2N+1$.
    By \cref{Prop:Suss_Err}, equation \eqref{eq:proof-deg-1} implies that
    \begin{equation}
        \label{prod1}
        \underset{b \in\cB^\star_{2N+1}}{\overrightarrow{\prod}}
        e^{ T^{|b|} \xi_b(1,u) f_b } =
        \underset{b \in\cB^\star_{2N+1}}{\overrightarrow{\prod}}
        e^{ T^{|b|} \xi_b(1,\underline{u}) f_b }
        + O(T^{2N+2}).
    \end{equation}
    Let $\cR_N := \{X_1, M_1,\dots,M_{2N}, W_{N},X_0\}$, $\cE:=\{W_1,\dots,W_{N-1}\}$ and $\mathcal{I}_\cE$ the ideal of $\cL_{2N+1}(X)$ generated by $\cE$. 
    For any $b \in \mathcal{I}_\cE$, $f_b = 0$.
    Moreover, $\cB^\star_{2N+1} \setminus \cR_N \subset \mathcal{I}_\cE$.
    Thus \eqref{prod1} implies
    \begin{equation}
        \underset{b \in \cR_N}{\overrightarrow{\prod}}
        e^{ T^{|b|} \xi_b(1,u) f_b } =
        \underset{b \in \cR_N}{\overrightarrow{\prod}}
        e^{ T^{|b|} \xi_b(1,\underline{u}) f_b }
        + O(T^{2N+2}).
    \end{equation}
    Since $X_0$ is the maximal element of $\cB^\star$ (thus involved in the last term of both products) and $\xi_{X_0}(1,u)=\xi_{X_0}(1,\underline{u})=1$, we obtain
    \begin{equation}
        \underset{b \in \cR_N^{\oslash}}{\overrightarrow{\prod}}
        e^{ T^{|b|} \xi_b(1,u) f_b } =
        \underset{b \in \cR_N^{\oslash}}{\overrightarrow{\prod}}
        e^{ T^{|b|} \xi_b(1,\underline{u}) f_b }
        + O(T^{2N+2})
    \end{equation}
    where $\cR_N^{\oslash} := \cR_N \setminus \{X_0\}$.
    From this estimate, using that $e^{\varepsilon g} e^{\varepsilon h} = e^{\varepsilon(g+h)} + O(\varepsilon^2)$ to gather the terms for $M_{2N}$ and $W_N$, we obtain
    \begin{equation}
        \underset{j \in \intset{0, 2N-1}}{\overrightarrow{\prod}}
        e^{ T^{j+1} \xi_{M_j}(1,u) f_{M_j} }
        e^{T^{2N+1} F_{u}}
        =
        \underset{j \in \intset{0, 2N-1}}{\overrightarrow{\prod}}
        e^{ T^{j+1} \xi_{M_j}(1,\underline{u}) f_{M_j} }
        e^{T^{2N+1} F_{\underline{u}}}
        + O(T^{2N+2})
    \end{equation}
    where $F_u=\xi_{M_{2N}}(1,u) f_{M_{2N}}+ \xi_{W_N}(1,u) f_{W_N}$.
    By \cref{Lem:Tj}, this implies
    \begin{equation}
        \label{prod4}
        \forall j \in \intset{0, 2N-1},
        \qquad
        (\xi_{M_{j}}(1,u)-\xi_{M_{j}}(1,\underline{u}))f_{M_j}=0
        \qquad \text{ and } \qquad
        F_u=F_{\underline{u}}
    \end{equation}
    If there exists $j \in \intset{0, 2N-1}$ such that $f_{M_j}=0$ then $f_{M_{2N}}=0$ thus the conclusion holds.
    Otherwise, we deduce from \eqref{prod4} that, for every $j \in \intset{0, 2N-1}$, $\xi_{M_{j}}(1,u)=\xi_{M_{j}}(1,\underline{u})$. Then, by \cref{prop:xi>WN},
    $(\xi_{M_{2N}}(1,u)-\xi_{M_{2N}}(1,\underline{u}) ,
    \xi_{W_{N}}(1,u)-\xi_{W_{N}}(1,\underline{u})) \neq (0,0)$ and
    the equality $F_u=F_{\underline{u}}$ means that $f_{M_{2N}}$ and $f_{W_N}$ are linearly dependent.
\end{proof}

\subsection{Using commutator flows versus using degeneracies}
\label{sec:flows-vs-degeneracies}

As highlighted in the introduction, one must distinguish between two related but different notions: using commutator flows to obtain a method valid for all systems (for which one can compute these flows) as in \cref{sec:high_order_comm}, or using degeneracies to obtain a method valid only for systems exhibiting these as in \cref{sec:high_order_deg}.
With our vocabulary, one has the following result.

\begin{proposition}
    \label{Prop:+flow->rel}
    Let $\cE \subset \Br(X) \setminus \{ X_0 \}$. 
    If there exists an $(\R^+,\R)$ splitting method of order $N$ involving $X_0$ and $\cE$, then there exists a splitting method of order $N$ involving only $X_0$ and~$X_1$, relative to the set $\{ (f_0,f_1) \in \cC^\infty(\R^d;\R^d) \mid f_b=0, \enskip \forall b \in \cE \setminus \{X_1\} \}$. 
    The reciprocal is false.
\end{proposition}

\begin{proof}
    The direct implication is straightforward.
    If one assumes that $f_b = 0$ for all $b \in \cE \setminus \{ X_1 \}$, the associated flows satisfy $e^{\alpha f_b} = \operatorname{Id}$, so they disappear from the splitting method involving $f_0$ and the elements of $\cE$. 

    A counterexample to the reciprocal is given by the case $\cE=\{X_1, M_1\}$.
    For every $N\in\N^*$, we know a splitting method of order $N$ relative to the smooth vector fields satisfying $f_{M_1}=0$ since the relation $[f_1,f_0]=0$ implies that $e^{T(f_0+f_1)}=e^{T f_1} e^{T f_0}$. 
    However, there does not exist a splitting method of order $5$ involving $X_0$ and $X_1$, $M_1$.

    Indeed, let us assume that such a method exists. 
    By \cref{prop:split-xi}, there exists a control $u=(u_{X_1},u_{M_1}) \in \fD$ such that $\xi_b(1,u)=\xi_b(1,\underline{u})$ for every $b \in \cB^\star_5$, where the $\xi_b$ are the coordinates of the second kind associated with $\cB^\star$, $\cE$ and $\underline{u}=(\underline{u}_{X_1},\underline{u}_{M_1})=(1,0)$.
    Since $\cE = \{ X_1, M_1 \}$, by \cref{prop:xi>WN} with $N = 2$, $\xi_{W_2}(1,u)>\xi_{W_2}(1,\underline{u})$, which is a contradiction since $W_2 \in \cB^\star_5$.
\end{proof}

\subsection{Characterization of a relative splitting method}

In \cref{prop:splitting-iff}, we gave a characterization making a link between the controllability of $\cZ_N$ and the existence of a splitting method of order $N$ valid for all systems.
In this paragraph, we give an analogue characterization for splitting methods relative to a class of vector fields satisfying degeneracy conditions.
Note that, for any smooth vector fields $f_0$, $f_1$, the set of $b \in \cL(X)$ such that $f_b = 0$ is an ideal of $\cL(X)$.

\begin{proposition}
    Let $N \in \N^*$ and $I$ an ideal of $\cL^\oslash_N(X)$.
    The following statements are equivalent:
    \begin{enumerate}
        \item \label{it:charac-1}
        There exists an $(\R^+,\R)$ splitting method of order $N$ involving $X_0$ and $X_1$ relative to the smooth vector fields satisfying $f_b=0$ for every $b \in I$.
        \item \label{it:charac-2}
        There exists $u \in \fD$ such that $\sigma \big( \cZ_N(1,X,u) - \log( \exp(-X_0) \exp(X_0+X_1))\big)=0$, where $\sigma : \cL^\oslash_N(X) \to \cL^\oslash_N(X)/ I$ is the canonical surjection. 
    \end{enumerate}
\end{proposition}

\begin{proof}
    \emph{\cref{it:charac-1} $\Rightarrow$ \cref{it:charac-2}.}
    Let $g_0, g_1$ be the vector fields on $\cL^\oslash_N(X)$ defined in \cref{Prop:LARC_Z}.
    Let $\widetilde{g_0}, \widetilde{g_1}$ be the induced vector fields on the quotient $\cL^\oslash_N(X) / I$.
    They are $\sigma$-related to $g_0, g_1$, i.e.\ they satisfy $\mathrm{d} \sigma |_z g_i(z) = \widetilde{g_i}(\sigma(z))$ for all $z \in \cL^\oslash_N(X)$.
    Since $b \mapsto \widetilde{g_b}$ is a Lie algebra morphism factoring through $\cL_N^\oslash(X)/I$, it vanishes on $I$. 
    So $\widetilde{g_b}=0$ for every $b \in I$.
    
    The function $Y(t)=\sigma(\cZ_N(t,X,u))$ solves the affine system $\dot{Y} = \widetilde{g_0}(Y) + u \widetilde{g_1}(Y)$ with $Y(0) = 0$. 
    We conclude by applying the splitting method to this system.
    Let $u \in \fD$ be given by \cref{it:charac-1}.
    Then $x(1;T\widetilde{g},u,0) = e^{T(\widetilde{g_0}+\widetilde{g_1})}(0) + O(T^{N+1})$ as $T \to 0$.
    Let $\mu_T$ be the Lie algebra morphism on $\cL(X)$ such that $\mu_T(X_i) = T X_i$.
    By the definitions of $g_0$ and $g_1$, we have $\mu_T (g_i(z)) = T g_i(\mu_T(z))$.
    Thus $Z_T(t) := \mu_T \cZ_N(t,X,u)$ satisfies $\dot{Z_T} = T g_0(Z_T) + u T g_1(Z_T)$.
    Applying $\mathrm{d} \sigma$, we obtain that $Y_T := \sigma(\mu_T \cZ_N(t,X,u))$ satisfies $\dot{Y}_T = T \widetilde{g_0}(Y_T) + u T \widetilde{g_1}(Y_T)$.
    Hence we have $\sigma \mu_T (\cZ_N(1,X,u) - \cZ_N(1,X,\bar{u})) = O(T^{N+1})$ where $\bar{u} = 1$.
    Since the left-hand side is a polynomial in $T$ of degree at most $N$, we conclude that it is zero.
    Evaluating at $T = 1$, we have $\sigma(\cZ_N(1,X,u) - \log(\exp(-X_0) \exp(X_0+X_1))) = 0$.
    
    \medskip
    \emph{\cref{it:charac-2} $\Rightarrow$ \cref{it:charac-1}.}
    There exists $u \in \fD$ and $V(X) \in I$ such that $\cZ_N(1,X,u) = Q_N(X)+V(X)$
    where $Q_N(X):=\log( \exp(-X_0) \exp(X_0+X_1)) = \cZ_N(1,X,\bar{u})$ in $\cL_N^\oslash(X)$ where $\bar{u} := 1$.
    Let $f_0, f_1$ be smooth vector fields on $\R^d$ such that $f_b=0$ for every $b \in I$. 
    By \cref{prop:error-Z}, as $T \to 0$,
    \begin{equation}
        \begin{split}
            x(1;Tf,u,x_0)
            & = e^{\cZ_N(1,Tf,u)} e^{Tf_0} (x_0) + O(T^{N+1}) \\
            & = e^{Q_N(Tf) + V(Tf)} e^{Tf_0} (x_0) + O(T^{N+1}) \\
            & = e^{Q_N(Tf)} e^{Tf_0} (x_0) + O(T^{N+1}).
        \end{split}
    \end{equation}
    Again by \cref{prop:error-Z}, as $T \to 0$,
    \begin{equation}
        \begin{split}
            e^{T(f_0+f_1)}(x_0) 
            & = x(1;Tf,\bar{u},x_0) \\
            & = e^{\cZ_N(1,Tf,\bar{u})} e^{T f_0}(x_0) + O(T^{N+1}) \\
            & = e^{Q_N(Tf)} e^{T f_0} (x_0) + O(T^{N+1}). 
        \end{split} 
    \end{equation}
    Combining both concludes the proof.
\end{proof}

\section{Proofs of the error estimates}
\label{sec:error-proofs}

We prove the error estimates stated in \cref{subsec:error_Magnus_syst}.
Such estimates are classical in control theory (see \cite{P1} for a survey).
We nevertheless here provide self-contained proofs since our estimates concern system~\eqref{EDOf0fA} which is only studied with $\cE = \{ X_1 \}$ in the literature.
Both estimates rely on an error estimate for the Chen--Fliess expansion, which we first prove.

We refer to \cite[Section 4.1]{P1} for a description of the key linearization principle which allows to make a link between the nonlinear ODE \eqref{eq:syst-f0-f1-fm} and the formal linear ODE \eqref{formalODE_A} by considering the operator $\varphi \mapsto \varphi \circ x(t;f,u,\cdot)$.
This principle underlies the following estimates.

\subsection{Error estimate for the Chen--Fliess expansion}

\begin{definition}
    Let $\Op(\cC^\infty(\K^d;\K))$ be the algebra of linear operators on $\cC^\infty(\K^d;\K)$. 
    In particular, we identify a vector field $f \in \cC^\infty(\K^d;\K^d)$ with the first-order differential operator $\varphi \mapsto (f \cdot \nabla) \varphi$.
    When $\K = \C$, we always implicitly consider holomorphic functions (see \cref{sec:holomorphic-systems}).
\end{definition}

\begin{proposition}
    \label{prop:error-Chen}
    Fix $f_0, f_1 \in \cC^\infty(\K^d;\K^d)$, $x_0 \in \K^d$ and $u \in \fD$ or $L^1$.
    For all $N \in \N$ and $\varphi \in \cC^\infty(\K^d;\K)$, there exists $C_{N,\varphi}$ such that, for $T$ small enough and $t \in [0,1]$,
    \begin{equation}
        \label{eq:Chen-T^N}
        \left| \varphi(x(t;Tf,u,x_0)) - \left(S_N(t,Tf,u) \varphi\right) (x_0) \right| \leq C_{N,\varphi} T^{N+1},
    \end{equation}
    where $S_N(t,Tf,u) \in \Op(\cC^\infty(\K^d;\K))$ is the image of $S_N(t,X,u)$ by the morphism of algebras from $\cA(X)$ to $\Op(\cC^\infty(\K^d;\K))$ mapping $X_0$ to $T f_0$ and $X_1$ to $T f_1$ (see \cref{lem:C-morphism-A} when $\K = \C$).
\end{proposition}

\begin{proof}
    Fix $u \in L^1((0,1);\K^{|\cE|})$.
    Since the following estimates only depend on $\| u \|_{L^1}$, the proof when $u \in \fD$ follows by the regularization of \cref{prop:impulsive-reg}.
    
    We prove by induction on $N \in \N$ that \eqref{eq:Chen-T^N} holds for all $\varphi \in \cC^\infty(\K^d;\K)$.
    For brevity, we write $S_N(t) := S_N(t,Tf,u)$ and $x(t) := x(t;Tf,u,x_0)$, which satisfies on $[0,1]$
    \begin{equation}
        \label{eq:edo-x-Tf}
        \dot{x} = T f_0(x) + \sum_{b \in \cE} u_b T^{|b|} f_b(x).
    \end{equation}
    \emph{Initialization.}
    Since $S_0(t,X,u) = 1$, we have $(S_0(t) \varphi)(x_0) = \varphi(x_0)$.
    For $T$ small enough, the trajectory $x(\cdot)$ remains in some fixed closed ball $B$ around $x_0$ for $t \in [0,1]$.
    Then
    \begin{equation}
        | \varphi(x(t)) - \varphi(x_0) | 
        \leq \sup_B | \nabla \varphi | \sup_{t \in [0,1]} |x(t) - x_0| 
        \leq \| \varphi \|_{\cC^1(B)} \int_0^1 | \dot{x}(s) | \dd s
        \leq C_{0,\varphi} T
    \end{equation}
    using \eqref{eq:edo-x-Tf} since $f_0,f_b$ are bounded on $B$ and $u_b \in L^1(0,1)$.
    Thus \eqref{eq:Chen-T^N} holds for $N=0$.

    \medskip \noindent
    \emph{Induction step.}
    Assume that \eqref{eq:Chen-T^N} holds up to $N-1$.
    By the chain rule, for all $t \in [0,1]$,
    \begin{equation}
        \label{eq:true-int}
        \varphi(x(t)) = \varphi(x_0) + \int_0^t T (f_0\varphi)(x(s)) + \sum_{b \in \cE} u_b(s) T^{|b|} (f_b\varphi)(x(s)) \dd s.
    \end{equation}
    Summing \eqref{eq:xn.xn1} of \cref{def:formalODE} for $n = 1, \dotsc, N$, we obtain
    \begin{equation}
        \label{eq:trunc-int}
        \big(S_N(t) \varphi\big) (x_0) = \varphi(x_0) + \int_0^t T \big(S_{N-1}(s) (f_0 \varphi)\big) (x_0) + \sum_{b \in \cE} u_b(s) T^{|b|} \big(S_{N-|b|}(s) (f_b \varphi)\big) (x_0) \dd s,
    \end{equation}
    with the convention that $S_m \equiv 0$ when $m < 0$.
    Subtracting \eqref{eq:true-int} and \eqref{eq:trunc-int}, we obtain
    \begin{equation}
        \begin{split}
            \varphi(x(t)) - \big(S_N(t) \varphi\big) (x_0) 
            = T & \int_0^t (f_0 \varphi)(x(s)) - \big(S_{N-1}(s) (f_0 \varphi)\big) (x_0) \dd s \\
            & + \sum_{b \in \cE} T^{|b|} \int_0^t u_b(s) (f_b \varphi)(x(s)) - \big(S_{N-|b|}(s) (f_b \varphi)\big) (x_0) \dd s.
        \end{split}
    \end{equation}
    Applying the induction hypothesis to $f_0 \varphi$ and $f_b \varphi$ for $b \in \cE$, we obtain
    \begin{equation}
        \begin{split}
            \left| \varphi(x(t)) - \big(S_N(t) \varphi\big) (x_0) \right|
            \leq T C_{N-1, f_0 \varphi} T^N & + \sum_{b \in \cE, |b| \leq N} T^{|b|} \| u_b \|_{L^1} C_{N-|b|, f_b \varphi} T^{N-|b|+1} \\ 
            & + \sum_{b \in \cE, |b| > N} T^{|b|} \| u_b \|_{L^1} \sup_B | f_b \varphi |. 
        \end{split}
    \end{equation}
    Summing over the finite set $\cE$ yields \eqref{eq:Chen-T^N} at order $N$. 
    This completes the proof.
\end{proof}

\subsection{Truncation of flows}

\begin{lemma}
    \label{lem:flot_tronc}
    Let $g \in \cC^\infty(\K^d;\K^d)$ and $x_0 \in \K^d$ such that $|g| \leq 1$ on $B := B(x_0, 2)$.
    For any $N \in \N$ and $\varphi \in \cC^\infty(\K^d;\K)$, there exists $C_N > 0$ such that, for all $x_1 \in B(x_0,1)$,
    \begin{equation}
        \label{eq:flot_tronc}
        \left| \varphi(e^{g}(x_1)) - \sum_{k=0}^N \frac{g^k \varphi}{k!} (x_1) \right| \leq C_N \|g\|_{\cC^N(B)}^{N+1} \|\varphi\|_{\cC^{N+1}(B)}.
    \end{equation}
\end{lemma}

\begin{proof}
    Since $|g| \leq 1$ on $B$, the flow $e^{tg}(x_1)$ is well defined for $t \in [0,1]$ and contained in $B$.
    Moreover, for $k \in \N$ and $t \in [0,1]$
    \begin{equation}
        \frac{\dd^k}{\dd t^k} \varphi(e^{tg}(x_1)) = (g^k \varphi)(e^{tg}(x_1)).
    \end{equation}
    Thus, the Taylor expansion of order $N$ of the map
    $t \mapsto \varphi(e^{tg}(x_1))$ at $t = 0$ evaluated at $t = 1$ is
    \begin{equation}
        \varphi(e^{g}(x_1)) =
        \sum_{k=0}^N \frac{g^k \varphi}{k!}(x_1)
        + \int_0^1 \frac{(1-s)^{N}}{N!} (g^{N+1} \varphi)(e^{sg}(x_1)) \dd s
    \end{equation}
    which implies \eqref{eq:flot_tronc} with $C_N$ such that $\sup_B |g^{N+1} \varphi| \leq C_N \|g\|_{\cC^N(B)}^{N+1} \|\varphi\|_{\cC^{N+1}(B)}$.
\end{proof}

\subsection{Proof of \cref{prop:error-split}}

\begin{proof}[Proof of \cref{prop:error-split}]
    By assumption $u$ is such that, in $\cA_N(X)$,
    \begin{equation}
        S_N(1,X,u) = \exp(X_0+X_1) = \sum_{k=0}^N \frac{(X_0+X_1)^k}{k!}.
    \end{equation}
    Let $\varphi \in \cC^\infty(\K^d;\K)$.
    We deduce from \cref{prop:error-Chen} with $t = 1$ that, as $T \to 0$,
    \begin{equation}
        \begin{split}
            \varphi(x(1;Tf,u,x_0)) & = \left(S_N(1,Tf,u)\varphi\right)(x_0) + O\left(T^{N+1}\right) \\
            & = \sum_{k=0}^{N} \frac{T^k}{k!} (f_0+f_1)^k (\varphi)(x_0) + O\left(T^{N+1}\right) \\
            & = \varphi(e^{T(f_0+f_1)}(x_0)) + O\left(T^{N+1}\right),
        \end{split}
    \end{equation}
    where the last equality comes from \cref{lem:flot_tronc} with $g = T(f_0+f_1)$ and $x_1 = x_0$.
    
    Thus \eqref{eq:prop-error-split} follows by taking $\varphi$ to be the coordinate functions.
\end{proof}

\subsection{Proof of \cref{prop:error-Z}}

\begin{lemma}
    \label{lem:cut-eval-Tf}
    For $f_0, f_1 \in \cC^\infty(\K^d;\K^d)$ and $T > 0$, let $L_T$ be the morphism of associative algebras from $\cA(X)$ to $\Op(\cC^\infty(\K^d;\K))$ mapping $X_0$ to $T f_0$ and $X_1$ to $T f_1$ (see \cref{lem:C-morphism-A} when $\K = \C$).
    Let $K \subset \K^d$ compact.
    \begin{itemize}
        \item 
        Let $a \in \cA(X)$ and $N \in \N$ such that $\pi_N(a) = 0$.
        For all $\varphi \in \cC^\infty(\K^d;\K)$, as $T \to 0$,
        \begin{equation}
            \label{eq:eval-Tf-a}
            \| (L_T a) \varphi \|_{\cC^0(K)} = O\left(T^{N+1}\right).
        \end{equation}
        \item 
        Let $b \in \cL(X)$ and $M \in \N$.
        Identifying $L_T b$ with a vector field, as $T \to 0$,
        \begin{equation}
            \label{eq:eval-Tf-b}
            \| L_T b \|_{\cC^M(K)} = O(T).
        \end{equation}
    \end{itemize}
\end{lemma}

\begin{proof}
    We use the grading of $\cA(X)$.
    For $m \in \N$, $a \in \cA^m(X)$ and $T > 0$, $L_T a = T^m (L_1 a)$. 
    
    Proof of \eqref{eq:eval-Tf-a}.
    By linearity, using the grading of $\cA(X)$ (see \cref{def:free.algebra}), it suffices to prove the estimate when $a \in \cA^m(X)$ for some $m > N$ and is of the form $a = X_{i_1} \dotsb X_{i_m}$.
    Then $L_T a = T^m (f_{i_1} \dotsb f_{i_m})$ is a composition of $m$ first-order differential operators.
    A standard Leibniz estimate yields
    \begin{equation}
        \| f_{i_1} \dotsb f_{i_m} \varphi \|_{\cC^0(K)} \leq C \| \varphi \|_{\cC^m(K)},
    \end{equation}
    with $C$ depending only on $K$, $m$ and $\| f_0 \|_{\cC^{m-1}(K)}$ and $\| f_1 \|_{\cC^{m-1}(K)}$.

    Proof of \eqref{eq:eval-Tf-b}.
    Similarly, decomposing $b$ using the grading induced by $\cA(X)$ on $\cL(X)$, it suffices to prove the estimate when $b \in \cL(X) \cap \cA^m(X)$ for some $m \geq 1$ (since $\cL(X) \cap \cA^0(X) = \{ 0 \}$).
    Then $L_T b = T^m (L_1 b)$ where $L_1 b \in \cC^\infty(\K^d;\K^d)$.
    So \eqref{eq:eval-Tf-b} follows.
\end{proof}

\begin{proof}[Proof of \cref{prop:error-Z}]
    Define $a \in \cA(X)$ by
    \begin{equation}
        a := S_N(1,X,u) - \sum_{l=0}^N \sum_{k=0}^N \frac{1}{l! k!} X_0^l (\cZ_N(1,X,u))^k.
    \end{equation}
    By \eqref{SN=exp(ZN)}, $\pi_N(a) = 0$.
    Fix $\varphi \in \cC^\infty(\K^d;\K)$ and $x_0 \in \K^d$.
    Let $B := B(x_0,2)$.
    By \cref{lem:cut-eval-Tf},
    \begin{equation}
        (S_N(1,Tf,u) \varphi)(x_0)
        = \sum_{l=0}^N \sum_{k=0}^N \frac{1}{l! k!} \left( (Tf_0)^l (\cZ_N(1,Tf,u))^k \varphi \right)(x_0) + O\left(T^{N+1}\right)
    \end{equation}
    For $k \in \intset{0,N}$, let $\psi_k := \frac{1}{k!} (\cZ_N(1,Tf,u))^k \varphi$.
    By \cref{lem:flot_tronc} applied to $Tf_0$, $\psi_k$ and $x_1 = x_0$,    
    \begin{equation}
        \left| \psi_k(e^{T f_0}(x_0)) - \sum_{l=0}^N \frac{1}{l!} ((T f_0)^l \psi_k )(x_0) \right| \leq C_N \| T f_0 \|_{\cC^N(B)}^{N+1} \| \psi_k \|_{\cC^{N+1}(B)}.
    \end{equation}
    By \cref{lem:flot_tronc} applied to $\cZ_N(1,Tf,u)$, $\varphi$ and $x_1 = e^{T f_0}(x_0)$,
    \begin{equation}
        \left| \varphi(e^{\cZ_N(1,Tf,u)} e^{Tf_0} (x_0)) - \sum_{k=0}^N \psi_k (e^{T f_0}(x_0)) \right| \leq C_N \| \cZ_N(1,Tf,u) \|_{\cC^N(B)}^{N+1} \| \varphi \|_{\cC^{N+1}(B)}.
    \end{equation}
    By definition, we have $\cZ_N(1,Tf,u) = L_T(\cZ_N(1,X,u))$.
    By \cref{lem:cut-eval-Tf}, $\| \cZ_N(1,Tf,u) \|_{\cC^{2N}(B)} = O(T)$, and thus $\| \psi_k \|_{\cC^{N+1}(B)} = O(T^k)$ for all $k \in \intset{0,N}$.
    
    Using estimate \eqref{eq:Chen-T^N} of \cref{prop:error-Chen} at $t = 1$, and gathering the above, we obtain as $T \to 0$,
    \begin{equation}
        \varphi(x(1;Tf,u,x_0)) = \varphi(e^{\cZ_N(1,Tf,u)} e^{Tf_0} (x_0)) + O\left(T^{N+1}\right).
    \end{equation}
    Thus \eqref{eq:error-x-exp(z)} follows by taking $\varphi$ to be the coordinate functions.
\end{proof}

\appendix

\section{Classical arguments from control theory}
\label{s:appendix}

For the sake of giving a self-contained exposition, we gather here the proofs of some well-known results in control theory used throughout this paper.

\subsection{Rank condition implies accessibility}
\label{s:access}

We present the inversion argument \cref{p:access} which is a classical result of control theory, constructing a ``normal control'' as in \cite[page 169]{MR872457}. 
Its proof requires the following lemma on vector fields tangent to a submanifold.

\begin{lemma}
    \label{lem:tangent-local}
    Let $V \subset \R^d$ open, $M \subset V$ a smooth embedded submanifold, and $f$, $g$ smooth vector fields on $V$, tangent to $M$, i.e.\ $f(x),g(x)\in T_x M$ for all $x\in M$. 
    Then $[f,g]$ is tangent to $M$.
\end{lemma}

\begin{proof}
    Fix $x \in M$. 
    Since $M$ is an embedded submanifold, there exist an open neighborhood $W \subset V$ of $x$ and a smooth submersion $\psi : W\to \R^{d-k}$, where $k=\dim M$, such that $M \cap W = \psi^{-1}(0)$.
    Moreover, for all $y \in M \cap W$, $T_y M = \ker D\psi(y)$.
	
	Since $g(x)\in T_x M$, there exists a smooth curve $\gamma:(-\varepsilon,\varepsilon) \to M \cap W$ such that
	$\gamma(0)=x$ and $\gamma'(0)=g(x)$. 
	Since $f$ is tangent to $M$, $D\psi(\gamma(t)) f(\gamma(t)) = 0$ for all $t$.
	Differentiating with respect to $t$ at $t=0$ yields
	\begin{equation}
		D^2\psi(x)\bigl(g(x),f(x)\bigr)+D\psi(x)Df(x)g(x)=0.
	\end{equation}
	Swapping $f$ and $g$, and using the symmetry of $D^2\psi(x)$, we get
	\begin{equation}
	D\psi(x) [f,g](x) = D\psi(x)\bigl(Dg(x)f(x)-Df(x)g(x)\bigr)=0.
	\end{equation}
    Thus $[f,g](x) \in T_x M$.
\end{proof}

\begin{proposition}
    \label{p:access}
    Let $V \subset \R^d$ be an open neighborhood of $x_0 \in \R^d$, and $\mathcal{F}$ a family of smooth vector fields on $V$.
    Assume that $(\Lie \mathcal{F})(x_0) = \R^d$.
    Then, for all $T>0$, there exist $f_1,\dots,f_d \in \mathcal{F}$,
    and an open subset $\Omega \subset (0,T)^d$ such that the map
    \begin{equation}
        \Phi : \ft = (t_1, \dotsc, t_d) \mapsto e^{t_d f_d} \dotsb e^{t_1 f_1}(x_0)
    \end{equation}
    is a smooth embedding of $\Omega$ into $V$.
    In particular, $D \Phi$ has rank $d$ on $\Omega$.
\end{proposition}

\begin{proof}
    Choose $g_1,\dots,g_d\in \Lie \mathcal{F}$ such that $g_1(x_0),\dots,g_d(x_0)$ is a basis of $\R^d$.
    By continuity, there exists an open neighborhood $W \subset V$ of $x_0$ such that, for every $x \in W$, the vectors $g_1(x),\dots,g_d(x)$ are still linearly independent, hence $(\Lie \mathcal{F})(x) = \R^d$.

    We prove by induction on $k\in\{1,\dots,d\}$ the following statement:
    \emph{There exist $f_1,\dots,f_k\in\mathcal F$ and an open subset $\Omega_k \subset (0,T)^k$ such that the map $\Phi_k : \ft = (t_1, \dotsc, t_k) \mapsto e^{t_k f_k} \dotsb e^{t_1 f_1}(x_0)$ is a smooth embedding of $\Omega_k$ into $W$.}

    \medskip
    \noindent\emph{Initialization: $k=1$.}
    First, there exists $f_1 \in \mathcal{F}$ such that $f_1(x_0)\neq 0$.
    Otherwise, all elements of $\Lie \mathcal{F}$ would vanish at $x_0$, contradicting $(\Lie \mathcal{F})(x_0) = \R^d$.
    Choosing such an $f_1$, $\Phi_1$ is well-defined and smooth on a neighborhood of $0 \in \R^1$.
    Since $D \Phi_1(0) = f_1(x_0) \neq 0$, by continuity, there exists an open subset $\Omega_1 \subset (0,T)$ on which $\Phi_1$ is an embedding.

    \medskip
    \noindent\emph{Induction step.}
    Assume the statement proved for some $k\in\{1,\dots,d-1\}$.

    Since $\Phi_k$ is a smooth embedding, $M_k := \Phi_k(\Omega_k)$ is a smooth embedded submanifold of $\R^d$ of dimension $k$.
    Moreover, for any $\mathfrak{s} \in \Omega_k$ and $y = \Phi_k(\mathfrak{s}) \in M_k$, $T_y M_k = D \Phi_k(\mathfrak{s}) (\R^k)$.
        
    Suppose, for contradiction, that every field in $\mathcal{F}$ is tangent to $M_k$.
    Then, by \cref{lem:tangent-local}, every field in $\Lie \mathcal{F}$ is tangent to $M_k$.
    Therefore, for every $x \in M_k$, $(\Lie \mathcal{F})(x) \subset T_x M_k$.
    But $M_k \subset W$, so $(\Lie \mathcal{F})(x) = \R^d$ for every $x \in M_k$, while $\dim T_x M_k = k < d$, a contradiction.
    Hence there exist $f_{k+1} \in \mathcal{F}$ and $y \in M_k$ such that $f_{k+1}(y) \notin T_y M_k$.
    
    Write $y = \Phi_k(\mathfrak{s})$ with $\mathfrak{s} \in \Omega_k$.
    For $\ft = (t_1, \dots, t_{k+1})$ in a neighborhood of $(\mathfrak{s},0)$, the map $\Phi_{k+1}$ is well-defined, smooth, with values in $W$.
    Moreover,
    \begin{equation}
        D\Phi_{k+1}(\mathfrak{s},0)(h,\lambda)
        =
        D\Phi_k(\mathfrak{s})h+\lambda f_{k+1}(y).
    \end{equation}
    Since $f_{k+1}(y) \notin D\Phi_k(\mathfrak{s})(\R^k)$, the map
    $D\Phi_{k+1}(\mathfrak{s},0)$ has rank $k+1$.
    Thus there exists an open subset $\Omega_{k+1} \subset \Omega_k \times (0,T)$ on which $\Phi_{k+1}$ is a smooth embedding.
\end{proof}

\cref{p:access} has a natural generalization for a set of holomorphic vector fields.

\begin{corollary}
    \label{p:access-C}
    Let $V \subset \C^d$ be an open neighborhood of $z_0 \in \C^d$, and $\mathcal{F}$ a family of holomorphic vector fields on $V$.
    Assume that $(\Lie_\R \mathcal{F})(z_0) = \C^d$.
    For all $T>0$, there exist $f_1,\dots,f_{2d} \in \mathcal{F}$,
    and an open subset $\Omega \subset (0,T)^{2d}$ such that the map
    \begin{equation}
        \Phi : \ft = (t_1, \dotsc, t_{2d}) \mapsto e^{t_{2d} f_{2d}} \dotsb e^{t_1 f_1}(z_0)
    \end{equation}
    is a smooth embedding of $\Omega$ into $V$.
    In particular, $D \Phi$ has rank $2d$ on $\Omega$.
\end{corollary}

\begin{proof}
    As in \cref{sec:holomorphic-systems}, work with the associated family $\mathcal{G}$ of smooth vector fields on $\R^{2d}$ defined from the realification of the vector fields of $\mathcal{F}$.
    Since $(\Lie_{\R} \mathcal{F})(z_0)=\C^d \Leftrightarrow (\Lie_{\R}\mathcal{G})(\iota(z_0))=\R^{2d}$, one can apply \cref{p:access} to the family $\mathcal{G}$, and the conclusion follows.
\end{proof}

\subsection{Exact and approximate local controllability}
\label{s:exact=app-proof}

We prove the equivalence of small-time local exact and approximate controllability, with controls in $\cU = \fD$ or $\cU = L^1$, under the Lie algebra rank condition assumption.

\begin{proof}[Proof of \cref{p:approx=exact} for $\K = \R$]
    Let $\cU = \fD$ or $\cU = L^1$.
    Assume that small-time local approximate $\cU$-controllability holds and fix $T > 0$.
    Let $\delta > 0$ such that $B(0,\delta) \subset \overline{R_{T/3}(\cU)}$.
    
    Since the Lie algebra rank condition $(\Lie \left\{ f_0, \dotsc, f_m \right\})(0) = \R^d$ holds, by \cref{p:access}, for any $\eta > 0$ there exist $g_1,\dots,g_d \in \{f_0,\dots, f_m\}$ and an open subset $\Omega \subset (0,\eta)^d$ such that the endpoint map $\Phi(\ft) := e^{-t_d g_d} \dotsb e^{-t_1 g_1} (0)$ is a smooth embedding of $\Omega$ into $\R^d$.
    Up to reducing $\eta$, we can assume that $\eta < T/(3d)$ and that $\Phi(\Omega) \subset B(0,\delta)$.

    Using the approximate controllability, there exists $\bar{u} \in \cU$ such that $\bar{x} := x(\frac{T}{3};f,\bar{u},0) \in \Phi(\Omega)$. 
    Thus, there exists a unique $\bar{\ft} \in \Omega$ such that $\bar{x} = \Phi(\bar{\ft})$, meaning $e^{\bar{t}_1 g_1} \dotsb e^{\bar{t}_d g_d} (\bar{x}) = 0$.
    
    For $\varepsilon \in [0,\eta)$ and $\ft \in \Omega$, define
    \begin{equation}
        \Psi_\cU(\varepsilon, \ft) := 
        \begin{cases} 
            e^{\varepsilon f_0} e^{t_1 g_1} \dotsb e^{t_d g_d} (\bar{x}) & \text{if } \cU = \fD, \\
            e^{\varepsilon f_0 + t_1 g_1} \dotsb e^{\varepsilon f_0 + t_d g_d} (\bar{x}) & \text{if } \cU = L^1.
        \end{cases}
    \end{equation}
    Since $\bar{x} \in R_{T/3}(\cU)$ and $\eta < T/(3d)$, using that $e^{\tau f_0}(0) = 0$ for any $\tau \geq 0$ (since $f_0(0) = 0$), we obtain by construction that, for any $\varepsilon \in (0,\eta)$ and $\ft \in \Omega$, $\Psi_\cU(\varepsilon,\ft) \in R_T(\cU)$.

    Let $P(\ft,y) := e^{t_1 g_1} \dotsb e^{t_d g_d}(y)$.
    Since $P(\ft,\Phi(\ft)) = 0$, we have $D_\ft P(\bar{\ft},\bar{x}) + D_y P(\bar{\ft},\bar{x}) D_\ft \Phi(\bar{\ft}) = 0$.
    $D_\ft \Phi(\bar{\ft})$ is an isomorphism because $\Phi$ is a smooth embedding by \cref{p:access}.
    $D_y P(\bar{\ft},\bar{x})$ is an isomorphism because $P(\bar{\ft},\cdot)$ is a diffeomorphism.
    Thus $D_\ft P(\bar{\ft},\bar{x}) = D_\ft \Psi_\cU(0,\bar{\ft})$ is an isomorphism.

    Hence, by the implicit function theorem, for $\varepsilon > 0$ and $x^*$ small enough, there exists $\ft \in \Omega$ such that $\Psi_\cU(\varepsilon,\ft) = x^*$. 
    Hence $0 \in \operatorname{int} R_T(\cU)$ and small-time local exact $\cU$-controllability holds.
\end{proof}

\begin{proof}[Proof of \cref{p:approx=exact} for $\K = \C$]
    Assume that $(\Lie_\C \left\{ f_0, f_1, \dots, f_m \right\})(0) = \C^d$.
    This implies that $(\Lie_\R \left\{ f_0, f_1, \dots, f_m, i f_1, \dotsc, i f_m \right\})(0) = \C^d$ using that $f_0(0) = 0$.
    Hence the proof is identical, using \cref{p:access-C} instead of \cref{p:access}.
\end{proof}

\subsection{The tangent vectors method}
\label{sec:tgt_vect}

We recall a classical method to prove the controllability of control-affine systems, relying on so-called ``tangent vectors'' (see \cite{Frankowska1986,kawski} for early occurrences, and \cite[Section~4]{Marbach2026} for an in-depth presentation).
A light version of this method combines elementary concatenation estimates (see \cref{Lem:concat}) with a Brouwer fixed-point argument (see \cref{prop:tgt-Brouwer}).

\begin{lemma}
    \label{Lem:concat}
    Let $m\in\N^*$, $f_0, f_1,\dots,f_m$ be smooth vector fields on a neighborhood of $0$ in $\K^d$ with $f_0(0)=0$. 
    There exist $C,\delta,r>0$ such that, for every $x_0 \in B(0,r)$, $u\in L^1(\R^+;\K^m)$ and $t \in \R^+$ such that $t+\|u\|_{L^1} \leq \delta$,  the solution to \eqref{eq:syst-f0-f1-fm} satisfies
    \begin{equation}
        |x(t;f,u,x_0)-x_0-x(t;f,u,0)| \leq C |x_0|(t+\|u\|_{L^1}).
    \end{equation}
    Moreover, for every $n\in\N^*$, there exists $C_n>0$ such that, given times $T_1, \dotsc, T_n \geq 0$ and $u_i \in L^1((0,T_i);\K^m)$, $T = T_1 + \dotsb + T_n$ and $u = u_1 \diamond \dotsb \diamond u_n$ such that $T + \| u \|_{L^1} \leq \delta$, then
    \begin{equation}
        \big| x(T;f,u,0) - \sum_{j=1}^{n} x(T_j;f,u_j,0) \big| \leq C_n (T+\|u\|_{L^1}) \sum_{j=1}^{n-1} |x(T_j;f,u_j,0)|.
    \end{equation}
\end{lemma}

\begin{proof}
    The first estimate follows from Grönwall's lemma.
    It entails the second one by induction.
\end{proof}

\begin{proposition}
    \label{prop:tgt-Brouwer}
    Let $m \in \N^*$ and $f_0, f_1, \dotsc, f_m$ smooth vector fields on a neighborhood of $0 \in \K^d$ with $f_0(0) = 0$.
    Let $(e_1,\dots,e_D)$ be an $\R$-basis of $\K^d$ ($D = d$ for $\K = \R$ and $D = 2d$ for $\K = \C$).
    Assume that for each $k \in \intset{1, D}$, there exist $n_k \in \N^*$ and two continuous maps $T \in [0,1] \mapsto u^{T,k,\pm} \in L^1((0,T);\K^m)$ such that, as $T \to 0$,
    \begin{equation}
        \label{eq:hyp-tgt}
        \| u^{T,k,\pm} \|_{L^1} = O(T) 
        \quad \text{and} \quad
        x(T;f,u^{T,k,\pm},0) = \pm T^{n_k} e_k + O\left(T^{n_k+1}\right).
    \end{equation}
    Then system \eqref{eq:syst-f0-f1-fm} is small-state STLC in the sense of \cref{def:STLC}.
\end{proposition}

\begin{proof}
    Let $N := \max n_k$.
    For $z = \sum_{k=1}^{D} z_k e_k \in \K^d$, we introduce
    \begin{itemize}
        \item for each $k \in \intset{1,D}$, the time
        $T_{z_k}:=|z_k|^{\frac{1}{n_k}}$ and the control
        $v^{z_k}:=u^{T_{z_k},k,\sign(z_k)}$; by \eqref{eq:hyp-tgt},
        \begin{equation}
            \label{eq:Tvzk}
            T_{z_k}+ \|v^{z_k}\|_{L^1} = \underset{z_k \to 0}{O}\left( |z_k|^{\frac{1}{n_k}} \right)
            \quad \text{and} \quad
            x(T_{z_k};f,v^{z_k},0)=z_k e_k + \underset{z_k \to 0}{O}(|z_k|^{1+\frac{1}{n_k}}),
        \end{equation}
        \item the time $T^z=T_{z_1} + \dotsb + T_{z_D}$ and the control $v^z=v^{z_1} \diamond \dots \diamond v^{z_D}$ (see \cref{Def:concat}) so that
        \begin{equation}
            \label{eq:Tz+uz}
            T^z+\|v^z\|_{L^1}
            =
            \sum_{k=1}^D \left(T_{z_k}+\|v^{z_k}\|_{L^1}\right)
            =O\left( |z|^{\frac{1}{N}} \right).
        \end{equation}
    \end{itemize}
    Thus, by \cref{Lem:concat} and using \eqref{eq:Tvzk}, we obtain
    \begin{equation}
        x(T^z;f,v^z,0)=z+O\left( |z|^{1+\frac{1}{N}} \right).
    \end{equation}
    In other words, there exist $C,R>0$ such that, for every $z \in \overline{B}(0,R)$,
    $|x(T^z;f,v^z,0) - z| \leq C |z|^{1+\frac{1}{N}}$.
    We may assume $R$ small enough so that $2 CR^{\frac{1}{N}}<1$.
    Let $x^* \in B(0,R/2)$. 
    The map
    \begin{equation}
        F : 
        \begin{cases}
            \overline{B}(0,R) \to \K^d \\
            z \mapsto z - x(T^z;f,v^z,0) + x^*
        \end{cases}
    \end{equation}
    is continuous and takes values in $\overline{B}(0,R)$, since, for every $z \in \overline{B}(0,R)$, $|F(z)| \leq C|z|^{1+\frac{1}{N}} + |x^*| \leq ( CR^{\frac{1}{N}}+\frac{1}{2} ) R \leq R$.
    By the Brouwer fixed-point theorem, there exists $z^* \in \overline{B}(0,R)$ such that $F(z^*)=z^*$ i.e.\ $x(T^{z^*};f,v^{z^*},0) = x^*$. Moreover, for $x^*$ small enough, by taking $R = 4|x^*|$, we deduce from \eqref{eq:Tz+uz} that
    \begin{equation}
        T^{z^*} + \| v^{z^*} \|_{L^1} \leq C' |x^*|^{\frac{1}{N}}
    \end{equation}
    i.e.\ small targets $x^*$ are reached with small times and controls ensuring small states all along the trajectory, so that \cref{def:STLC} holds.
\end{proof}

\subsection{Useful formula} \label{sec:AppendixA}

\subsubsection{Lie--Trotter product formula}

\begin{lemma}
    \label{lem:abstract-product}
    Let $f \in \cC^1(\K^d;\K^d)$ and $x_0 \in \K^d$ such that the solution to $\dot{x} = f(x)$ with $x(0) = x_0$ is well-defined on $[0,1]$.
    Let $(\Phi_t)_{t \geq 0}$ be a family of $\cC^1$ maps on $\K^d$, such that, for some $\sigma \in (0,1]$, uniformly on a closed ball containing $x([0,1])$ in its interior,
    \begin{equation}
        \Phi_t(y) = y + t f(y) + O(t^{1+\sigma})
    \end{equation}
    Then, as $N \to +\infty$,
    \begin{equation}
        \left| \Phi_{\frac 1 N}^N(x_0) - e^f(x_0) \right| = O\left(N^{-\sigma}\right).
    \end{equation}
\end{lemma}

\begin{proof}
    Let $K$ denote the assumed closed ball containing the full trajectory $x([0,1])$ in its interior.
    By assumption, and regularity of $f$, there exists $C, T > 0$ such that, for all $t \in [0,T]$ and $y, y' \in K$,
    \begin{equation}
        | \Phi_t(y) - e^{t f}(y) | \leq C t^{1+\sigma}
        \quad \text{and} \quad 
        | e^{t f}(y) - e^{t f}(y') | \leq (1+Ct) |y-y'|
    \end{equation}
    Let $y_0 := x_0$.
    For $0 \leq k < N$, let $x_{k+1} := x((k+1)/N) = e^{f/N} (x_k)$ and $y_{k+1} := \Phi_{1/N}(y_k)$.
    We want to prove that $y_N \to x_N = x(1)$.
    As long as $y_k \in K$, we have
    \begin{equation}
        |y_{k+1}-x_{k+1}| \leq \left(1+\frac CN\right)|y_k-x_k|+\frac {C}{N^{1+\sigma}}.
    \end{equation}
    Hence $\varepsilon_k := (1+C/N)^{-k} |y_k - x_k|$ satisfies $\varepsilon_0 = 0$ and $\varepsilon_{k+1} \leq \varepsilon_k + C / N^{1+\sigma}$ so $\varepsilon_k \leq C k / N^{1+\sigma}$.
    Thus $|y_k - x_k| \leq \frac{C}{N^{\sigma}} (1+\frac C N)^N \leq \frac{C e^C}{N^\sigma}$.
    In particular $y_k$ stays in $K$ and the estimate closes.
\end{proof}

\begin{lemma}
    \label{lem:LT}
    Let $h_1, \dotsc, h_q \in \cC^\infty(\K^d;\K^d)$ and $H := h_1 + \dotsb + h_q$.
    Let $x_0 \in \K^d$.
    If the solution to $\dot{x} = H(x)$ with $x(0) = x_0$ is well-defined on $[0,1]$, then
    \begin{equation}
        e^H(x_0) = \lim_{N \to +\infty} \left(e^{\frac{h_1}{N}}\cdots e^{\frac{h_q}{N}}\right)^N(x_0).
    \end{equation}
\end{lemma}

\begin{proof}
    For $t \geq 0$, define $\Phi_t := e^{t h_1} \dotsb e^{t h_q}$.
    Uniformly within any compact set, one has
    \begin{equation}
        \Phi_t(y) = y + t (h_1 + \dotsb + h_q)(y) + O(t^2)
        = y + t H(y) + O(t^2).
    \end{equation}
    Thus the result follows from \cref{lem:abstract-product}.
\end{proof}

\subsubsection{Bernoulli numbers} \label{subsec:Bernoulli}

\begin{definition}
    We use the notation $(B_n)_{n\in\N}$ to denote the Bernoulli numbers, which are defined (using the modern NIST sign and indexing convention) by the identity
    \begin{equation}
        \label{eq:def:bernoulli}
        \forall z \in \C, |z|<2\pi, \qquad
        \frac{z}{e^{z}-1} = \sum_{n=0}^{+\infty} B_n \frac{z^n}{n!} = 1 - \frac{z}{2} + \sum_{n=1}^{+\infty} B_{2n} \frac{z^{2n}}{(2n)!}.
    \end{equation}
\end{definition}

\begin{lemma}
    The Bernoulli numbers satisfy, for every $n \geq 2$,
    \begin{align}
        \label{eq:bernoulli.1}
        \sum_{k=0}^{n-1} \binom{n}{k} B_k
        & = 0, \\
        \label{eq:bernoulli.2}
        \sum_{k=0}^{n} \binom{n}{k} \frac{B_k}{n+1-k}
        & = 0.
    \end{align}
\end{lemma}

\begin{proof}
    These can be proved using the generating series \eqref{eq:def:bernoulli} of the Bernoulli numbers, by identification in $z = (e^z-1) \times (z / (e^z-1))$ for \eqref{eq:bernoulli.1} and in $1 = ((e^z-1)/z) \times (z/(e^z-1))$ for~\eqref{eq:bernoulli.2}.
\end{proof}

\begin{lemma}
    Let $T>0$ and $z \in \cC^1([0,T]; {\widehat{\cA}(X)})$. Then, {for every $t \in [0,T]$,}
    \begin{equation}
        \label{eq:dez-dt}
        \frac{\dd}{\dd t} \exp(z(t))
        = \exp(z(t)) \sum_{n=0}^{+\infty} \frac{(-1)^n}{(n+1)!} \ad^n_{z(t)} (\dot{z}(t)),
    \end{equation}
    \begin{equation}
        \label{dot{z}}
        \dot{z}(t) = \sum_{k = 0}^{+\infty} \frac{(-1)^k B_k}{k!} \ad^{k}_{z(t)} \left( \exp(-z(t)) \frac{\dd}{\dd t} \exp(z(t)) \right)
    \end{equation}
\end{lemma}

\begin{proof}
    The regularity assumption $z \in \cC^1([0,T]; \widehat{\cA}(X))$ is to be understood for projections on the finite dimensional spaces $\cA^n(X)$.
    We have
    \begin{equation}
        \begin{split}
            \frac{\dd}{\dd t} \exp(z(t))
            & = \frac{\dd}{\dd t} \left( \sum_{k=0}^{+\infty} \frac{z^k(t)}{k!} \right)
            = \sum_{k=0}^{+\infty} \frac{1}{(k+1)!} \sum_{j=0}^{k} z^j(t) \dot{z}(t) z^{k-j}(t)
            \\
            & = \exp(z(t)) \left( \sum_{l=0}^{+\infty} \frac{(-1)^l}{l!}z^l(t) \right)
            \left( \sum_{k=0}^{+\infty} \frac{1}{(k+1)!} \sum_{j=0}^{k} z^j(t) \dot{z}(t) z^{k-j}(t) \right).
        \end{split}
    \end{equation}
    Letting $n := k + l$ and $i := l + j$, we obtain that
    \begin{equation}
        \begin{split}
            \frac{\dd}{\dd t} \exp(z(t))
            & = \exp(z(t)) \sum_{n=0}^{+\infty} \frac{1}{(n+1)!} \sum_{i=0}^n z^{i}(t) \dot{z}(t) z^{n-i}(t)
            \sum_{l=0}^i (-1)^l \binom{n+1}{l}
        \end{split}
    \end{equation}
    The following formulas, which can be proved by induction using Pascal's rule,
    \begin{align}
        \sum_{l=0}^i (-1)^l \binom{n+1}{l} & = (-1)^i \binom{n}{i}, \\
        \sum_{i=0}^n (-1)^i \binom{n}{i} z^i y z^{n-i} & = (-1)^n \ad^n_z(y)
    \end{align}
    give the conclusion.
    Of course, if $z \in W^{1,1}((0,T);\widehat{\cA}(X))$ (i.e.\ absolutely continuous), equation \eqref{eq:dez-dt} remains true as an equality in $L^1((0,T);\widehat{\cA}(X))$, i.e.\ holding for almost every $t \in (0,T)$.

    From the change of index $n = k + \ell$ and the combinatorial relation \eqref{eq:bernoulli.2} we obtain
    \begin{equation}
        \sum_{k = 0}^{+\infty} \frac{(-1)^k B_k}{k!} \sum_{\ell = 0}^{+\infty} \frac{(-1)^\ell}{(\ell+1)!} \ad^{k+\ell}_{z(t)}(\dot{z}(t)) = \dot{z}(t),
    \end{equation}
    which implies \eqref{dot{z}}.
\end{proof}

\subsubsection{Conjugation formulas}

\begin{lemma}
    \label{lem:conjugation}
    For every $a,b \in \cL(X)$, the following equality holds in $\widehat{\cL}(X)$:
    \begin{equation}
        \exp(a) b \exp(-a) 
        = \sum_{m=0}^\infty \frac{1}{m!} \ad_a^m(b)
        = \exp(\ad_a) b.
    \end{equation}
\end{lemma}

\begin{proof}
    By definition, we have the following equality in $\widehat{\cA}(X)$:
    \begin{equation}
        \exp(a) b \exp(-a)
        = 
        \left( \sum_{q=0}^\infty \frac{1}{q!} a^q \right)
        b
        \left( \sum_{p=0}^\infty \frac{(-1)^p}{p!} a^p \right).
    \end{equation}
    Using the change of index $m=p+q$, we deduce
    \begin{equation}
        \exp(a) b \exp(-a) =
        \sum_{m=0}^\infty \frac{1}{m!}
        \sum_{p=0}^{m} (-1)^p \frac{m!}{(m-p)! p!} a^{m-p} b a^p
        = \sum_{m=0}^\infty \frac{1}{m!} \ad_a^m(b),
    \end{equation}
    which is the claimed equality.
\end{proof}

\begin{lemma}
    \label{Lem:conjug_flots}
    If $f,g$ are smooth vector fields on $\K^d$ such that $\Lie \left\{ f, g \right\}$ is nilpotent of step $N$, as long as the flows are defined, one has
    \begin{equation}
        e^{-f} e^g e^f = e^h
        \qquad \text{ where } \qquad
        h=\sum_{k=0}^{N-1} \frac{1}{k!} \ad_{f}^k(g).
    \end{equation}
\end{lemma}

\begin{proof}
    For any $a, b \in \cL(X)$, from \cref{lem:conjugation}, we obtain
    \begin{equation}
        \exp(a) \exp(b) \exp(-a) 
        = \exp(\exp(a) b \exp(-a))
        = \exp \left(\sum_{k=0}^\infty \frac{1}{k!} \ad_a^k(b) \right).
    \end{equation}
    Thus, for any $a, b \in \cL_N(X)$, we have the equality in $\cA_N(X)$:
    \begin{equation}
        \BCH_N(\BCH_N(a,b),-a) = \sum_{k=0}^{N-1} \frac{1}{k!} \ad_a^k(b).
    \end{equation}
    Hence the conclusion follows from \cref{prop:BCH}.
\end{proof}

\subsubsection{The Baker--Campbell--Hausdorff formula} 

\begin{proposition}
    \label{prop:BCH}
    Let $N \in \N^*$.
    There exists a Lie polynomial $\BCH_N$ in two indeterminates $A$ and $B$ such that, for any $A, B \in \cL_N(X)$, the following equality holds in $\cL_N(X)$:
    \begin{equation}
        \label{eq:BCH}
        \BCH_N(A,B) 
        = \log \left(\exp (A) \exp (B)\right)
        = A + B + \text{finite sum of brackets of $A$ and $B$.}
    \end{equation}
    Moreover, if $f$ and $g$ are smooth vector fields on $\K^d$ such that $\Lie \left\{ f, g \right\}$ is nilpotent of step $N$, as long as the flows are defined, one has
    \begin{equation}
        e^g e^f = e^{\BCH_N(f,g)}.
    \end{equation}
\end{proposition}

\begin{proof}
    The first statement is merely the historical Baker--Campbell--Hausdorff formula.
    The consequence for nilpotent vector fields is proved in \cite[Remark A.1]{Jean2014} for analytic vector fields, and in \cite[Section 5.2.1]{P1} without analyticity.
\end{proof}

\section*{Acknowledgments}

Frédéric Marbach thanks the organizers of the \emph{Normal forms and splitting methods} conference in Pornichet, June 2022, for providing a stimulating environment, particularly during the talk by Fernando Casas on \cite{Blanes22asm}, where the initial ideas for this research were conceived.

The three authors acknowledge support from grants ANR-25-CE40-2862-01 (Project MaStoC), ANR-24-CE40-5470 (Project CHAT), ANR-11-LABX-0020 (Labex Lebesgue), as well as from the Fondation Simone et Cino Del Duca -- Institut de France.

\bibliographystyle{abbrv}
\bibliography{biblio,Ma_Bibliographie}

\end{document}